\documentclass[11pt]{article}

\usepackage{amsmath}
\usepackage{epsfig, amssymb, amsfonts}
\usepackage{amsthm}
\usepackage{amstext}
\usepackage{amsopn}
\usepackage{mathrsfs}
\usepackage{subfigure}
\usepackage{graphicx}
\usepackage[left=2.3cm,top=2cm,right=2.3cm,bottom=2cm, nohead,foot=1cm]{geometry}
\usepackage[makeroom]{cancel}
\usepackage{color}
\usepackage{enumerate}
\usepackage{comment}

\numberwithin{equation}{section}

\newtheorem{theorem}{Theorem}[section]

\newtheorem{lemma}[theorem]{Lemma}

\newtheorem{corollary}[theorem]{Corollary}

\theoremstyle{definition}
\newtheorem{definition}[theorem]{Definition}
\newtheorem{remark}[theorem]{Remark}

\newcommand{\R}{{\mathbb R}}

\newcommand{\Z}{{\mathbb Z}}
\newcommand{\N}{{\mathbb N}}

\newcommand{\sbf}{{s}}
\newcommand{\bbf}{{b}}

\newcommand{\ds}{{\displaystyle}}

\newcommand{\TS}{{\mkern-1mu\times\mkern-1mu}}
\newcommand{\DD}{{\mkern-1mu\diamond\mkern-1mu}}
\newcommand{\BA}{{\widehat{\beta}}}

\newcommand{\BR}{{\mathbf{R}}}
\newcommand{\I}{{\textup{i}}}
\newcommand{\Try}{{\triangleleft}}

\date{\vspace{-9ex}}

\title{The intermediate disorder regime for a directed polymer model on a hierarchical lattice\vspace{.3cm}}

 \author{\textbf{Tom Alberts}\footnote{Department of Mathematics,
 University of Utah: {\tt   alberts@math.utah.edu}}  \hspace{.3cm} \& \hspace{.3cm} \textbf{Jeremy Clark}\footnote{Department of Mathematics,  University of Mississippi: {\tt
jeremy@olemiss.edu}}  \hspace{.3cm} \& \hspace{.3cm} \textbf{Sa\v{s}a Koci\'c}\footnote{Department of Mathematics,  University of Mississippi: {\tt skocic@olemiss.edu }}\vspace{.5cm} }

\begin{document}
\maketitle

\begin{abstract}
We study a directed polymer model defined on a hierarchical diamond lattice, where the lattice is constructed recursively through a recipe depending on a branching number $b\in \mathbb{N}$ and a segment number $s\in \mathbb{N}$.  When $b\leq s$ previous work \cite{lacoin} has established that the model exhibits strong disorder for all positive values of the inverse temperature $\beta$, and thus weak disorder reigns only for $\beta=0$ (infinite temperature).  Our focus is on the so-called \textit{intermediate disorder regime} in which the inverse temperature $\beta\equiv \beta_{n}$ vanishes at an appropriate rate as the size $n$ of the system grows. Our analysis requires separate treatment for the cases $b<s$ and $b=s$. In the case $b<s$ we prove that when the inverse temperature is taken to be of the form $\beta_{n}=\BA (b/s)^{n/2}$ for $\BA>0$, the normalized partition function of the system converges weakly as $n \to \infty$ to a distribution  $\mathbf{L}(\BA)$ depending continuously on the parameter $\BA$.  In the case $b=s$ we find a critical point in the behavior of the model when the inverse temperature is scaled as $\beta_{n}=\widehat{\beta}/n$; for an explicitly computable critical value $\kappa_{b} > 0$ the variance of the normalized partition function converges to zero with large $n$ when $\widehat{\beta}\leq \kappa_{b}$ and grows without bound when $\widehat{\beta}>\kappa_{b}$.   Finally, we prove a central limit theorem for the normalized partition function when $\BA\leq \kappa_{b}$.

\end{abstract}

\section{Introduction}

Probabilistic models of directed polymers are often constructed as modifications of directed random walk models on a particular lattice. For the particular model that we study, directed polymers in a random environment, the path probabilities are randomly perturbed by a Gibbsian reweighting that is determined by each realization of the environment, and the strength of the perturbation is controlled by an inverse temperature parameter $\beta$. Broadly speaking, the main question of interest in such models is how the presence of the random environment affects these path probabilities, and to what degree any effect can be quantified. The values of $\beta$ for which the presence of the environment has no substantial effect are called the \textit{weak disorder regime}, while those $\beta$ for which the environment has a meaningful influence make up the \textit{strong disorder regime}. The precise characterization of weak versus strong disorder is given in terms of the positivity of a martingale limit, the martingale being the normalized partition function associated with the model.

In recent years there has been substantial attention paid to determining the exact structure of the weak and strong disorder regimes. On the integer lattice $\Z^d$ (which is commonly referred to as $d+1$-dimensional to indicate that the polymer is directed) it is known \cite{Comets2} that there is a critical value $\beta_c \geq 0$ such that $\beta < \beta_c$ is in the weak disorder regime and $\beta > \beta_c$ is in the strong disorder regime. When $d = 1$ or $2$ it is known that the polymer is in the strong disorder regime for all finite temperatures (i.e. $\beta_c = 0$), whereas for $d \geq 3$ it is known that there is a non-trivial interval of finite temperatures that make up the weak disorder regime (i.e. $\beta_c > 0$). In the latter case there are only bounds for the value of $\beta_c$ and how it depends on the dimension $d$ and the statistical distribution of the underlying environment.

In contrast, for directed polymers on a self-similar tree there is an \textit{exact} characterization \cite{Biggins, Kahane} of the value of $\beta_c$ (for any choice of the statistical distribution of the environment satisfying weak moment assumptions), and much more is known about the strong and weak disorder regimes \cite{Buffet, Morters}. The tree model is easier to analyze because of the underlying geometry of the space, namely that the tree is self-similar and that paths, once they split, can never recombine. In particular, the non-recombining feature means that the random energies assigned to different paths are essentially independent after a long time, which is very different from what happens in the lattice case.

Recently a new phenomenon has emerged in directed polymer models: the \textit{intermediate disorder regime}. The terminology \textit{intermediate disorder} was first introduced in \cite{alberts} to indicate that it sits between weak and strong disorder, and the regime itself is accessed by scaling the inverse temperature $\beta$ towards its critical value as the size of the system goes to infinity. The correct scaling produces new and interesting behavior that is fundamentally different from what happens for strong and weak disorder. At the same time, the intermediate regime often contains a full range of models that act (or are conjectured to act) as a bridge between the two classical regimes.

Typically intermediate disorder is easier to study in situations where $\beta_c = 0$, since in this case the polymer model is only a small perturbation of the underlying random walk. In $d=1$ it was shown in \cite{alberts} that the proper scaling for intermediate disorder is to replace $\beta$ with $\beta n^{-1/4}$, and that under a diffusive scaling of space and time the entire random Gibbs measure converges (in law) to a random measure on continuous paths called the continuum directed random polymer. Properties of the continuum polymer were analyzed in \cite{alberts2}, but it is important to note that the continuum polymer is actually an entire \textit{family} of models indexed by the parameter $\beta$. It is not yet understood how or even if this family acts as a bridge between weak and strong disorder, although encouraging results in this direction were obtained in \cite{Corwin}. For $d = 2$ the situation for intermediate disorder is largely open, and for $d \geq 3$ it seems that little has been done for scalings around $\beta_c$. A model somewhat similar to the $d=1$ case, the \textit{continuum disordered pinning model} (CDPM), was considered in \cite{Carav2}. Again in that case $\beta_c = 0$, and the authors were able to determine the precise scaling for intermediate disorder and take a scaling limit of the random Gibbs measure. In the tree case, where $\beta_c > 0$, the proper scaling window for the intermediate disorder regime was determined in \cite{Alberts_Ortgiese}, although nothing could be said about infinite volume limits of the corresponding Gibbs measures.

In this paper we study intermediate disorder on the \textit{diamond hierarchical lattice}. This particular lattice has an interesting geometrical structure that
still possesses a self-similarity like the tree does, but at the same time allows paths to reintersect like they do on $\Z^d$. In this sense the diamond lattice is a natural generalization of the tree but in many important ways is actually more like $\Z^d$, and for the polymer model with site disorder (which is what we consider) this is especially true. See \cite[Remark 1.1]{lacoin3} for a convincing argument in favor of this viewpoint. The structure of the diamond lattice is controlled by two parameters $b, s \in \N$: the \textit{branching number} and the \textit{segment number}. The ``lattice'' is actually a sequence of graphs constructed by a recursive procedure determined by the branching and segment numbers, and the depth of the recursive determines the length of the polymer. For directed polymers on the diamond lattice it was proved in \cite{lacoin} that there is a $\beta_c$ that strictly separates weak disorder from strong disorder, as in the $d+1$-dimensional case, and that $\beta_c = 0$ for $b \leq s$ while $\beta_c > 0$ for $b > s$. These results are precisely summarized in Section \ref{SecMain}.

Our main focus  is on the intermediate disorder regime in the $b<s$ and $b = s$ cases, i.e., when $\beta_c = 0$, which exhibit widely disparate behaviors.   Our analysis, in a sense, extends the results of \cite{lacoin} to obtain a richer description of the phase diagram around the critical point.   In the case $b<s$,  one sees through a simple Taylor expansion of the partition function that the correct intermediate disorder scaling is to replace $\beta$ with $\BA( b/s)^{n/2}$ for a parameter $\BA>0$; see Section~\ref{SecHeuristic}.   Similarly to the  $1+1$-dimensional polymer and the continuum disordered pinning model, we find that there is a one-parameter  family of limit laws $\mathbf{L}(\BA)$ for the normalized partition functions of the model. The limiting family depends continuously on $\BA$ and for $b < s$ satisfies
\begin{itemize}
\item  For $\BA\ll 1$, $\mathbf{L}(\BA)$ is approximately a normal distribution  with mean $1$ and variance  $\BA^2\frac{b(s-1)}{s-b}$.

\item If $X^{(i,j)}(\BA) $ are independent random variables with distribution $\mathbf{L}(\BA)$, then
$$ \mathbf{L}\Big(\BA \sqrt{s/b}\Big)\, \stackrel{d}{=}\,\frac{1}{b}\sum_{i=1}^{b}\prod_{1\leq j\leq s}X^{(i,j)}(\BA)  \, . $$
  \end{itemize}
As in \cite{alberts,Carav1,Carav2} the limiting law appears as a \textit{universal} object, i.e. the same limit appears for any disorder variables with exponential moments that obey proper normalization constraints, but the limit itself is not Gaussian. Furthermore these limiting laws can be used to construct a continuum diffusion process interacting with a random environment on the continuum diamond lattice, similar to what is done in \cite{alberts2, Carav2}. This is discussed in Section \ref{Secb<s}. We emphasize that in contrast to \cite{alberts,Carav1,Carav2} our methods for showing the existence of these limit laws and their universality properties do not use polynomial chaos expansions; instead they come from renormalization group type ideas.

In the case $b=s$, nontrivial behavior emerges in the intermediate disorder regime by replacing $\beta$ with $\BA/n$, but in contrast to the $1+1$-dimensional polymer and the CDPM the parameter $\BA$ does not seem to act as any sort of bridge. In fact, in the $b = s$ case, there is a critical value $\kappa_b$ such that the variance of the rescaled partition function converges to zero with large $n$ for $\BA \leq \kappa_b$, but also grows to infinity  for $\BA > \kappa_b$. This is discussed in Section \ref{SecWeakDisorder}.


The structure of this paper is as follows: in Section \ref{SecMain} we give a full description of the model and precisely state our main results. In Section \ref{SecDef} we describe our coordinate system for the diamond lattice and use it to derive recursive formulas for moments of the partition function.  We especially point out Section \ref{SecHeuristic}, which gives some heuristic reasoning for the choice of scalings $\BA (b/s)^{n/2}$ and $\BA/n$ in the $b < s$ and $b = s$ cases, respectively.  Section~\ref{Secb<s} contains a proof of a limit theorem for the normalized partition function in the $b<s$ case, and   Section \ref{SecWeakDisorder} contains our results on the $b=s$ case.  In Sections~\ref{SecNormalWeakDisorder} and~\ref{SecEdgeModel} we give, respectively,  very brief discussions of the $b>s$ case and the equivalent model in the $b=s$ case in  which random weights are placed on the bonds of the graph rather than the sites.      Finally,  we relegate proofs that are particularly calculus-based to Section~\ref{SecMiscProofs}.   

\medskip
\noindent
\textbf{Acknowledgements:} Tom  and   Sa\v{s}a thank Kostya Khanin for useful discussions.

\section{The Model and Main Results}\label{SecMain}

Diamond graphs are constructed inductively by replacing each edge in $D_{n}$ by a `diamond' with $\bbf \in \mathbb{N}$ branches, each of which is split into $\sbf \in \mathbb{N}$ segments:
\begin{center}
\hspace{1cm}\includegraphics[scale=.5]{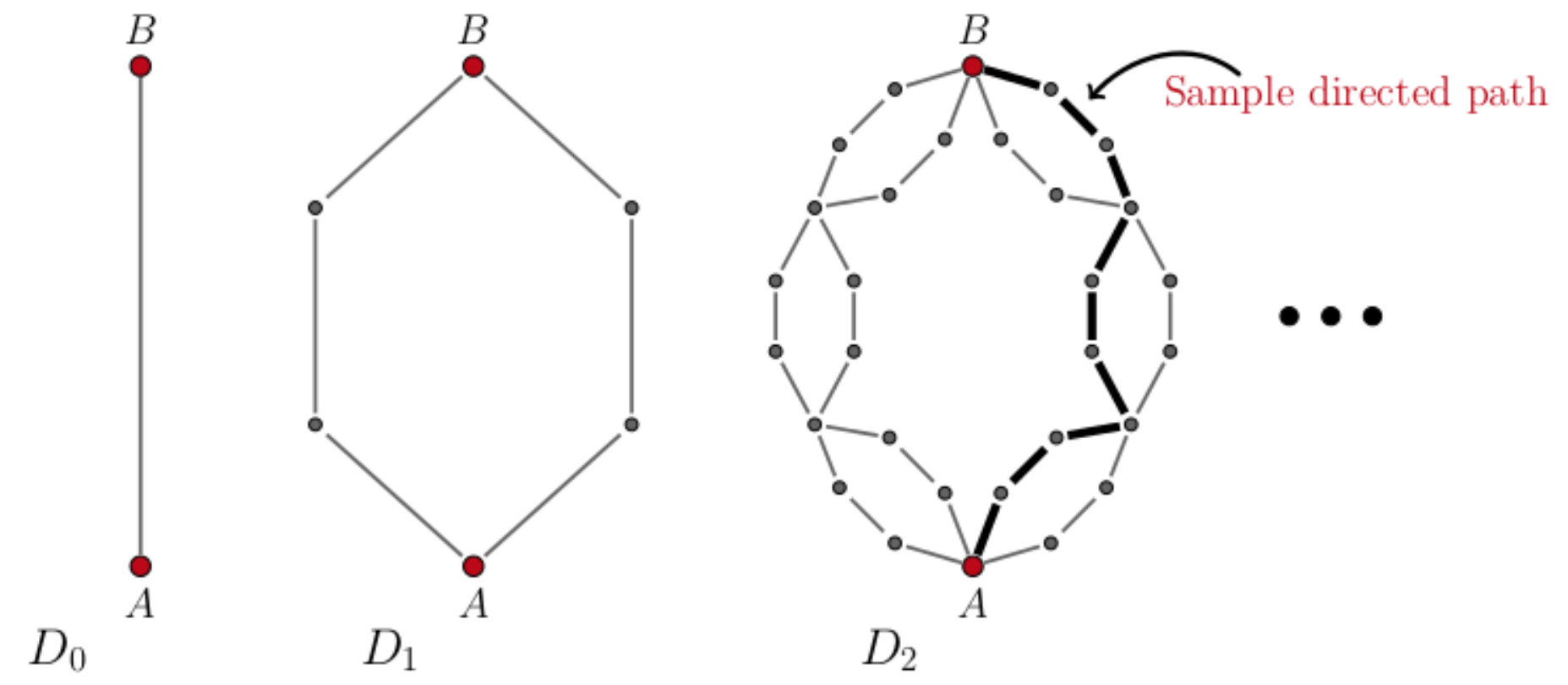}\\
\small The first few diamond graphs constructed with $\bbf=2$ and $\sbf=3$.
\end{center}
The $n$th diamond graph defines a set of directed paths between the root vertices $A$ and $B$, which we denote by $\Gamma_{n}$.  The inductive nature of the construction implies that $ D_{m+n}$ is equivalent to replacing each edge on $D_{m}$ by a copy of $D_{n}$. We assume throughout that $b > 1$, since otherwise the graph structure is trivial.

We now place  i.i.d.\ random variables $\omega_{a}$ at the set of vertices $a\in D_{n}$ and consider statistical mechanical quantities defined through sums over the set of directed paths, $\Gamma_{n}$.  Since  $D_{n}$ is canonically embedded in $D_{n+1}$, we can  view the random variables $\omega_{a}$, $a\in \cup_{n=1}^{\infty}D_{n} $ as residing in a single probability space.   We  assume that the random variables   have mean zero, variance one, and  finite exponential moments
\begin{align*}
\lambda(\beta) := \log \mathbb{E}\big[ e^{\beta \omega_{a}}   \big]\,<\, \infty
\end{align*}
for sufficiently small $\beta>0$.    The  partition function $Z_{n}(\beta)$ is the random positive quantity  defined by
$$Z_{n}(\beta) \, := \,\frac{1}{|\Gamma_{n}|  }\sum_{p\in \Gamma_{n}  }e^{\beta  H_{n}^{\omega}(p)    } \, , \hspace{1cm}\text{where}\hspace{1cm}  H_{n}^{\omega}(p) \, :=\,  \sum_{a\in p}  \omega_{a} \,  \quad \text{and} \quad |\Gamma_{n}| = \bbf^{\frac{\sbf^{n}-1}{\sbf-1} }\,.    $$
The partition function normalizes random probability measures $\mu^{(\omega)}_{\beta, n}$ on the path space $p\in \Gamma_{n}$:
$$\mu^{(\omega)}_{\beta, n}\big(p) \,   = \, \frac{  e^{\beta  H_{n}^{\omega}(p) }    }{  Z_{n}(\beta)   } \,  .    $$

Disorder for directed polymer models is commonly interpreted through a closely related quantity, $W_{n}(\beta)$, defined by
\begin{align}\label{ExpW}
W_{n}(\beta)\, := \, \frac{  Z_{n}(\beta)  }{  \mathbb{E}\big[  Z_{n}(\beta)   \big] }\,, \quad \text{which can be written as }\quad \frac{1}{|\Gamma_{n}|  }\sum_{p\in \Gamma_{n}  } \prod_{a\in p} E_{a}(\beta   )
\end{align}
for $\ds E_{a}(\beta   ):= e^{\beta  \omega_{a} - \lambda(\beta) }$.  The inductive nature of the construction of the diamond graphs $D_{n}$ suggests that there should be a recursive relation between the distributions of $W_{n}(\beta)$ and $W_{n+1}(\beta)$;  if $W_{n}^{(i,j)}(\beta)$  are independent copies of  $W_{n}(\beta)$ and $E^{(i,j)}(\beta   )$ are independent copies of  $E_{a}(\beta   ) $, then there is the following equality in distribution
\begin{align}\label{Induct}
W_{n+1}(\beta)  \, \stackrel{ d }{ =}  \,  \frac{1}{b}\sum_{i=1}^{b} \prod_{1\leq j\leq s}  W_{n}^{(i,j)}(\beta) \prod_{1\leq j\leq s-1}  E^{(i,j)}(\beta   )\, .
\end{align}
Moreover, the sequence of random variables $W_{n}$ forms a mean one martingale with respect to the filtration $\mathcal{F}_{n}:=\sigma\big\{  \omega_{a} \text{ for } a\in  D_{m}, m \leq n    \big\}$.  The martingale convergence theorem guarantees the almost sure existence of a nonnegative limit as $n\rightarrow \infty$:
$$  W_{\infty}(\beta) \, = \,\lim_{n\rightarrow \infty}W_{n}(\beta)\, .  $$
The literature~\cite{Bolth,Sepp,lacoin} on random polymer models commonly interprets those $\beta$ for which  $W_{\infty}(\beta)>0$ as being in the \textit{weak disorder} regime and those $\beta$ for which  $W_{\infty}(\beta)=0$  as being in the \textit{strong disorder} regime.

In Lemma 2.4 and Proposition 2.7 of~\cite{lacoin},  Lacoin and Moreno establish the following results for the diamond hierarchical lattice (amongst others):
\begin{enumerate}[i)]
\item   For any fixed $b$, $s$, and $\beta$, the random variable $W_{\infty}(\beta)$ satisfies the zero-one law
$$\mathbb{P}\big[  W_{\infty}(\beta)>0 \big] \, =\, 0 \text{ or } 1\,.   $$

\item When $b\leq s$, then for all $\beta>0$
$$\mathbb{P}\big[  W_{\infty}(\beta)=0 \big]=1\,. $$

\item When $s < b$ there is a $\beta_c > 0$ such that $\mathbb{P}[W_{\infty}(\beta) > 0] = 1$ for $\beta \leq \beta_c$ and $\mathbb{P}[W_{\infty}(\beta) = 0] = 1$ for $\beta > \beta_c$.
\end{enumerate}
Thus by (ii), when $b\leq s$  the question of weak disorder versus strong disorder is merely that of infinite temperature ($\beta=0$) versus finite temperature ($\beta>0$), whereas by (iii), when $s < b$ there is a region of finite temperatures $(\beta \leq \beta_c)$ in which weak disorder still holds.

Our goal in this article is to magnify the high temperature regime of the model, $\beta\ll 1$, in the cases $b<s$ and $b=s$
 through an examination of the limiting behavior of the random variables $W_{n}( \beta_{n})$ for  $n\gg 1$ and some choice of $\beta_{n}\ll 1$.  In other words, the inverse temperature $\beta_{n}$ is scaled towards the critical value $\beta_c = 0$ as the size $n$ of the system goes to infinity. The following theorem states our main result in the $b<s$ case.

\begin{theorem}[$b<s$]\label{ThmMainIII} Pick $\widehat{\beta}>0 $ and define $\beta_{n}:= \widehat{\beta}(b/s)^{\frac{n}{2}}  $.  As $n\rightarrow \infty$, there is convergence in distribution
\begin{align*}
W_{n}(\beta_{n}) \quad \stackrel{\mathcal{L}}{\Longrightarrow} \quad    L_{\widehat{\beta}^2\frac{s-1}{s-b}}^{b,s}\,
\end{align*}
for a family of  probability distributions $(L_{r}^{b,s})_{r > 0}$ supported on $\R^+$ and satisfying the following:

\begin{enumerate}[I)]
\item  $L_{r}^{b,s}$ has mean $1$ and variance $\frak{v}_{b,s}(r)$ for a function $\frak{v}_{b,s}:\R^{+}\rightarrow \R^+$ satisfying
$$    \frak{v}_{b,s}\big(\frac{s}{b}r\big)\,  =\,\frac{1}{b}\big[ (1+r)^{s}\,-\,1    \big]            \hspace{1cm}\text{and}\hspace{1cm} \lim_{r\searrow 0}  \frac{ \frak{v}_{b,s}(r)}{r}\,=\,1\,.   $$

\item  If $X^{(i,j)}_{r}$ are independent random variables with  distribution $L_{r}^{b,s}$, then
$$   L_{\frac{s}{b}r}^{b,s} \quad \stackrel{d}{=}\quad   \frac{1}{b}\sum_{i=1 }^{b} \prod_{1\leq  j\leq s} X^{(i,j)}_{r}\,.    $$

\item  If $X_{r}$ has distribution $L_{r}^{b,s}$, then as $r\searrow  0$ there is convergence in law
$$   \frac{X_{r}-1}{\sqrt{r}}  \hspace{.5cm}  \stackrel{\mathcal{L}}{\Longrightarrow} \hspace{.5cm}     \,\mathcal{N}(0,1)\,   .  $$

\end{enumerate}

\end{theorem}

\begin{remark}
Given $x >0$ the family of  probability distributions $(L_{r}^{b,s})_{r > 0}$ can be used to define a consistent family of random measures $\mu_{x}^{(n)}$ on paths $\Gamma_{n}$ for  $n\in \mathbb{N}$.   To do this   we attach i.i.d.\ random variables $X_{n}(a)$ having distribution $  L_{x(b/s)^{n}  }^{b,s}$ to edges $a\in E_{n}$ of the diamond graph $D_{n}$.   The random probability assigned to a given path $p\in \Gamma_{n}$ is then
\begin{align*}
\mu_{x}^{(n)}(p) \, = \, \frac{\prod_{a\Try p} X_{n}(a)}{\sum_{p\in \Gamma_{n}} \prod_{a\Try  p} X_{n}(a)    }\, ,
\end{align*}
where the products are over the set of edges $a$ lying along the path $p$.  We show that this definition is self-consistent in the proof of Corollary~\ref{CorCon}.

\end{remark}

\begin{remark}
In Lemma~\ref{LemTransition} we show that $L_{r}^{b,s}$ converges weakly to a $\delta$-distribution at $0$ as $r\rightarrow \infty$.  This is indicative of a continuous transition to strong disorder as $r\rightarrow \infty$. It would be interesting to determine if there are exponential fluctuations for these laws as $r \rightarrow \infty$, that is if by taking a logarithm and then centering and scaling in an appropriate way a non-trivial limiting law appears. This is the situation that was encountered for the limiting law of the partition function in \cite{alberts}, built upon earlier work in \cite{ACQ}. In that case the limiting law was the ubiquitous Tracy-Widom GUE distribution. 
\end{remark}

\begin{remark}
We emphasize that our methods for showing the existence of the $L_r^{b,s}$ laws and establishing convergence towards them is not based on polynomial chaos expansions, as in \cite{alberts,Carav1,Carav2}. In the polynomial chaos method the partition function is expanded into a sum of polynomials in the disorder variables and collected according to the degree. Convergence follows by Central Limit Theorem type arguments and limit laws for the transition kernels of the underlying random walk, which act as coefficients of the polynomials. The proofs in this paper are entirely different and instead use the recursive structure of the lattice to eliminate the use of the chaos expansion.
\end{remark}

Next we state our main result in the $b=s$ case, for which we scale the inverse temperature  as $\beta=\widehat{\beta}/n$.  We feel that this is an interesting regime, because there is a critical point in the behavior of the normalized partition function depending on whether the parameter $\widehat{\beta}$ is less than or greater than a cut-off value
\begin{align}\label{KappaB}
\kappa_{b}:=\frac{\pi \sqrt{b}}{\sqrt{2}(b-1)}.
\end{align}

\begin{theorem}[$b=s$]\label{ThmMain}
As $n\rightarrow\infty $, we have the following behavior for the variance of $W_{n}\big( \widehat{\beta}/n \big)$:
\begin{align}\label{WConvInProb}
\text{}\hspace{1cm}\textup{Var}\bigg(W_{n}\Big(\frac{\widehat{\beta}}{n}\Big)\bigg)  \quad \longrightarrow \quad \begin{cases} 0, & \quad   0\leq \widehat{\beta}  \leq   \kappa_{b} \, ,  \vspace{.2cm}    \\  \infty,  & \quad   \,\,\,\,\,\,\widehat{\beta} > \kappa_{b} \, . \end{cases}
\end{align}
When  $\widehat{\beta}  <  \kappa_{b}$ the fluctuations of $W_{n}\big( \widehat{\beta}/n\big) $ around $1$ obey the limit theorem:
\begin{align}\label{WConv}
  n^{\frac{1}{2}}\bigg(W_{n}\Big(\frac{\widehat{\beta}}{n}\Big) -1\bigg)    \quad  \stackrel{\mathcal{L}}{\Longrightarrow} \quad  \mathcal{N}\Big( 0,\, \upsilon_{b}(\widehat{\beta}) \Big) \,,
\end{align}
where the variance  is  $\ds \upsilon_{b}(\widehat{\beta}):=\widehat{\beta} \frac{\sqrt{2}}{\sqrt{b}}\tan\big( \frac{b-1}{\sqrt{2b}}\widehat{\beta}  \big)$.  At the critical value $\widehat{\beta}=\kappa_{b}$ there is a different scaling:
\begin{align}\label{WConvII}
  \sqrt{\log n}\bigg(W_{n}\Big(\frac{\kappa_{b} }{n}\Big) -1\bigg)    \quad  \stackrel{\mathcal{L}}{\Longrightarrow} \quad  \mathcal{N}\Big( 0,\, \frac{6}{b+1}\Big) \,.
\end{align}

\end{theorem}

\begin{remark}
We will see in Section \ref{SecWeakDisorder} that the cut-off $\kappa_{b}$ is the blow-up point of the solution $\phi_{b}:\R^{+}\rightarrow \R^{+}$ to the  differential equation
\begin{align*}
\text{ } \hspace{3cm}\frac{d}{dr}\phi_{b}(r) \, = \, \frac{b-1}{2} \big(\phi_{b}(r)\big)^{2}\, +\,\frac{b-1}{b} \,, \hspace{1cm} \text{with}\quad \phi_{b}(0)=0\,.
\end{align*}
The solution is explicitly given by  $\phi_{b}(r)= \frac{\sqrt{2}}{\sqrt{b}}\tan\big( \frac{b-1}{\sqrt{2b}}r  \big)$.  Note that $\kappa_{2}=\pi$.
\end{remark}

\begin{remark}For $\BA<\kappa_b$ and large $n$, a simple Taylor expansion shows that the free energy-type quantity $\ds\log\big(Z_{n}(\widehat{\beta}/n)\big) $ is close to $\ds W_{n}\big( \widehat{\beta}/n \big) -1$ and the convergence~(\ref{WConv}) is equivalent to
\begin{align*}
n^{\frac{1}{2}}\log\Big(Z_{n}\Big(\frac{\widehat{\beta}}{n}\Big) \Big)   \quad  \stackrel{\mathcal{L}}{\Longrightarrow} \quad  \mathcal{N}\Big( 0,\, \upsilon_{b}(\widehat{\beta}) \Big) \,.
\end{align*}

\end{remark}

To give a more complete picture of the behavior of $W_{n}(\beta_{n})$ with $\beta_{n}\searrow 0$ for large $n$    in the $b\neq s$ cases, we state the following theorem.

\begin{theorem}\label{ThmMainII}  Let $\BA>0$ and $\beta_{n}\searrow 0$.

\begin{enumerate}[i)]

\item When $b<s$ we have the following convergences in  law for  $W_{n}(\beta_{n})  $ depending on  $\beta_{n}$:
 $$\text{}\hspace{1cm}W_{n}(\beta_{n})  \quad  \stackrel{\mathcal{L}}{\Longrightarrow} \quad \begin{cases} 1, & \quad  \beta_{n}\big(\frac{s}{b}\big)^{n}\,\rightarrow \,0 \, ,  \vspace{.2cm}  \\
  L_{\widehat{\beta}^2\frac{s-1}{s-b}}^{b,s},   &  \quad   \beta_{n}\big(\frac{s}{b}\big)^{n}\,\rightarrow \,\BA\,,
  \\  0 ,  & \quad   \beta_{n}\big(\frac{s}{b}\big)^{n}\,\rightarrow \,\infty  \, . \end{cases}  $$

\item  When $b>s$ then $W_{n}(\beta_{n})$ converges to $1$.  A more  refined characterization of $W_{n}\big( \beta_{n}) $ is
\begin{align}\label{WConvII}
  \frac{1}{\beta_{n}}\big(W_{n} (\beta_{n}) -1\big)    \quad  \stackrel{\mathcal{P}}{\Longrightarrow} \quad   \sum_{m=1}^{\infty} \frac{1}{b^{m}}\sum_{ a\in V_{m}    }\omega_{a}     \,,
\end{align}
where $V_{m}$ is the set of $m^{th}$ generation vertices.

\end{enumerate}

\end{theorem}

\begin{remark}
In (i) of the above theorem,  the convergence of  $W_{n}( \beta_{n})$ to zero in the $b<s$ case when $\beta_{n}(b/s)^{n/2}\rightarrow \infty$  follows from the homogeneous environment tilting method in the proofs from Sections 5.2-5.6 of~\cite{lacoin}.
\end{remark}

\begin{remark}\label{Remarkb>s}
In the $s < b$ case, since there is a region of finite temperature in which weak disorder holds, it would be interesting to find an intermediate disorder regime by scaling towards the critical value of $\beta_c$ as the system size grows, that is, to consider
\begin{align*}
W_n(\beta_c + \beta_{n})
\end{align*}
for an appropriate sequence $\beta_{n}\searrow 0$ . Perturbations of this type are difficult to study, however, since they typically require fine properties of the Gibbs measure $\mu_{\beta_c, n}^{(\omega)}$ at criticality. For directed polymers on trees results of this type were obtained in \cite{Alberts_Ortgiese}, based on an earlier analysis of the Gibbs measure at criticality in \cite{Hu}.
\end{remark}

\section{Basic constructions, notation, and heuristics}\label{SecDef}

\subsection{More on diamond graphs}

Now we will give a more in-depth description of the diamond graph structure. The set of copies of $D_{n}$ found in $ D_{m+n}$ has a canonical one-to-one correspondence with the edges of $D_{m}$.  We will use the following notation:
\begin{itemize}

\item The symbol  $D_{n}$ for the $n^{th}$ diamond graph also denotes the set of vertices on the graph, excluding the roots  $A$ and $B$.

\item For $k\leq n$, let $V_{k}\subset D_{n}$ denote the set of $k^{th}$ generation vertices in $D_{n}$, in other words, those vertices that appear in $D_k$ but not in $D_{k-1}$.

\item  $E_{n}$ denotes the set of edges on the graph $D_{n}$.

\item For $k\leq n$, $G_{k,n}$ denotes the set of copies of $D_{k}$ found in $D_{n}$. An element $g\in G_{k,n}$ is also regarded as a subset of the set of vertices $D_{n}$. 

\end{itemize}
By construction of $D_n$ the elements in $G_{k,n}$ are in one-to-one correspondence with elements in $E_{n-k}$.

\begin{center}
\includegraphics[scale=.5]{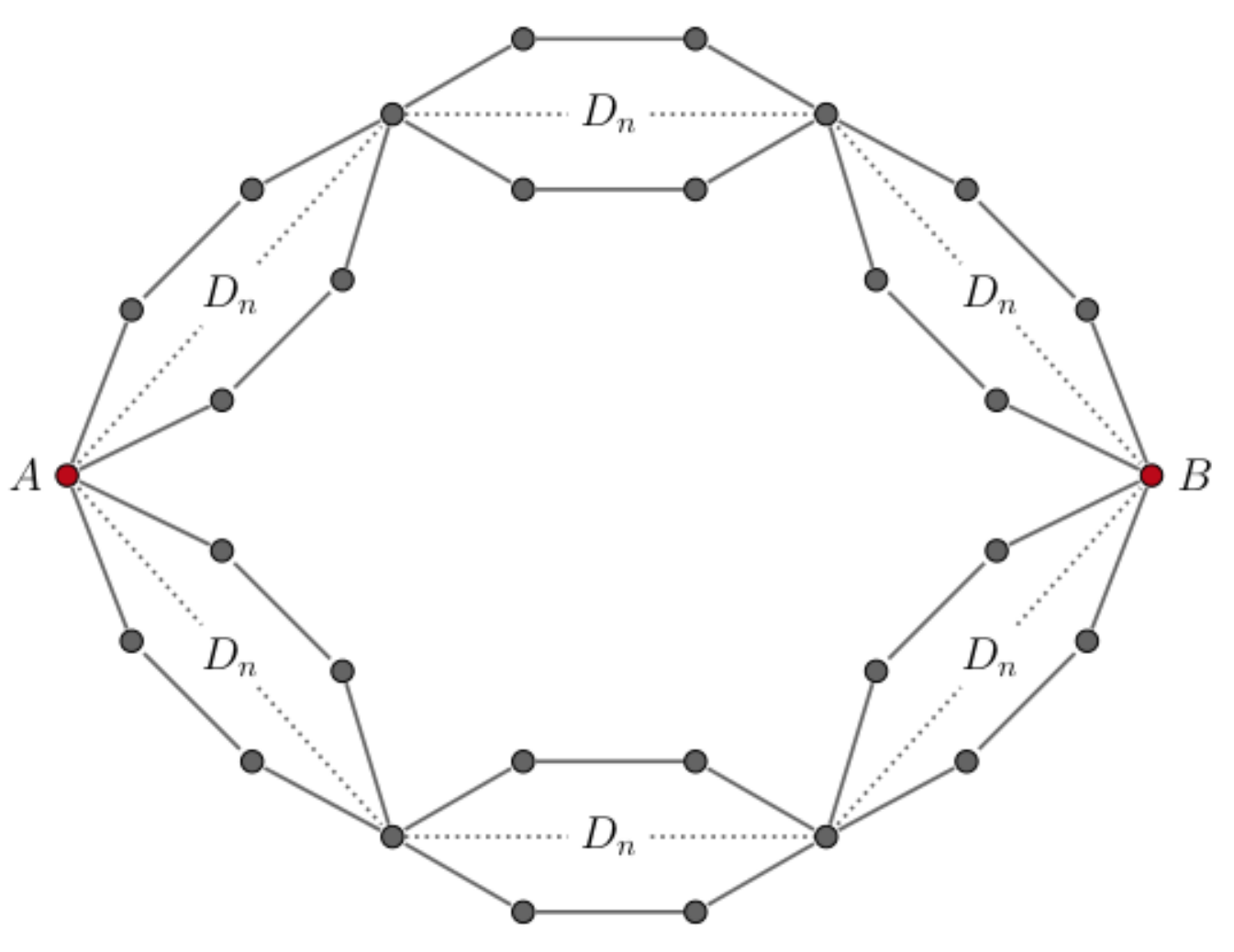}\\
Copies of $D_{n}$ embedded in $D_{n+m}$. \\
\end{center}

The edges of $D_{1}$ can be labeled by $(i,j) \in \{1,\cdots,\bbf   \}\times \{1,\cdots, \sbf \}$, where $(i,j)$ refers to the $j^{th}$ segment of the $i^{th}$ branch.  More generally, the set of edges  $E_{n}$ in $D_{n}$ can be identified with
$$ E_{n}\equiv \big(\{1,\cdots,\bbf   \}\times \{1,\cdots, \sbf\}\big)^{n}.$$
The correspondence follows from the recursive procedure for the construction of $D_{n}$:
\begin{align*}
&\hspace{4.8cm}\textbf{(Encodes an element of } E_{1} \textbf{ in the local copy of } D_{1} \textbf{)}  \\
&\hspace{4cm}\underbrace{(i_{1},j_{1})\times \cdots \times (i_{n-1},j_{n-1})} \times  \overbrace{(i_{n}, j_{n})} \\
&\textbf{(Encodes  an element of }  G_{1,n}\equiv E_{n-1}\textbf{, the set of copies of } D_{1} \textbf{ in  } D_{n}\textbf{)}
\end{align*}
Similarly, elements in $G_{k,n}$ can be labeled by $\big(\{1,\cdots,\bbf   \}\times \{1,\cdots, \sbf\}\big)^{n-k}$.  The elements of $G_{k-1,n}$ that are contained in $g \in G_{k,n}$ will be labeled by
$$  g\TS (i,j) \, ,  $$
where $i\in \{1,\cdots,\bbf   \}$, and $j\in \{1,\cdots,\sbf   \}$.

The $k^{th}$  generation vertices $V_{k}$ are identified with the set
\begin{align*}
G_{n-k+1,n}\times \big(  \{ 1,\cdots ,\bbf    \}\times \{1,\cdots,\sbf -1    \}     \big) \, .
\end{align*}
For $g\in G_{k,n}$,  we will label the generation $n-k+1$ vertices that are contained in $g$ by
 $$g\DD (i,j)\, ,$$
where  $i\in \{1,\cdots,\bbf   \}$ and $j\in \{1,\cdots,\sbf-1   \}$.\vspace{.3cm}

We have the  following basic combinatorial observations regarding these graphs:
\begin{enumerate}[i)]

\item The graph $D_{n}$ has $\bbf^{n}\sbf^{n}$ edges, i.e., $|E_{n}|=\bbf^{n}\sbf^{n} $.  Since the subgraphs $G_{k,n} $ and the edges $E_{n-k}$ are in one-to-one correspondence, this implies that  $|G_{k,n}|=\bbf^{n-k}\sbf^{n-k}$.

\item The graph $D_{n}$ has $\bbf^{k}\sbf^{k-1}(\sbf -1)$ vertices of generation $k$ when $k\leq n$, i.e., $|V_{k}|= \bbf^{k}\sbf^{k-1}(\sbf -1)$.

\item  An element $g\in G_{k,n}$ contains $\bbf^{j-n+k}\sbf^{j-n+k-1}(\sbf -1)$ vertices of generation $  j\in (  n-k, n]$.

\item There are $\bbf^{\frac{\sbf^{n}-1}{\sbf -1}}$ paths between the root vertices $A$ and $B$ in $D_{n}$, i.e.,  $|\Gamma_{n}|=\bbf^{\frac{\sbf^{n}-1}{\sbf -1}}$.

\end{enumerate}

\subsection{Definitions and recursive relations for statistical quantities}\label{SubSecRecur}

Recall that $G_{k,n}$ refers to the collection of copies of $D_{k}$ found as subgraphs of  $D_{n}$.
  For $g\in G_{k,n}$, we define $W_{n}(\beta; g)$ in analogy to  $W_{n}(\beta)$ except restricted to the subgraph $g$.   To be clear, $W_{n}(\beta; g)$ does not depend on the random variables at the roots of the subgraph $g$.  The following lemma is  essentially a  restatement of~(\ref{Induct}) in the subgraph notation:

\begin{lemma}  For $g\in G_{k,n}$, the random variables $W_{n}(\beta; g)$ satisfy the recursive relation
\begin{align}\label{WRecur}
W_{n}(\beta; g)\, =\,  \frac{1}{b}\sum_{i=1}^{b} \prod_{1\leq j\leq s}W_{n}\big(\beta; g\TS (i,j)\big)\prod_{1\leq j\leq s-1}  E_{g\diamond (i,j)}(\beta   )\,
\end{align}
with  $W_{n}(\beta; g):=1$ for $g\in G_{0,n}$.

\end{lemma}

From this lemma it is straightforward to derive a recursive relation for the variance of the $W_n$ variables.

\begin{corollary}\label{Sig}
For $g\in G_{k,n}$, the variance $\sigma_{k}(\beta):=\textup{Var}\big[  W_{n}(\beta; g)\big]$ satisfies the recursive relation in $k\in \mathbb{N}$
$$
\sigma_{k+1}(\beta)\, =\,  \frac{1}{b}\Big[ \big(1+\sigma_{k}(\beta)\big)^{s}e^{(s-1)[\lambda(2\beta)-2\lambda(\beta)]} \,- \, 1\Big]\, .
$$
\end{corollary}

\subsection{Heuristic motivation }\label{SecHeuristic}

We now offer a short heuristic analysis motivating the  scalings $\beta=\BA (b/s)^{n/2}$ and $\beta=\BA/n$ for the respective cases  $b<s$ and $b=s$.  We begin by  naively linearizing  the expression~(\ref{ExpW}) for the normalized partition function to obtain
\begin{align}\label{Boondog}
W_{n}(\beta_{n})\,\approx \, &\frac{1}{|\Gamma_{n}|}\sum_{p\in \Gamma_{n}}\prod_{a\in p} \big(1+\beta_{n} \omega_{a} \big)\nonumber \,,
\intertext{which is intuitively justified for small $\beta_{n}$.  The right side above can be expanded into the form}
\,=\, &1+ \sum_{m \geq 1} \sum_{\substack{a_{1},\cdots, a_{m} \in D_{n}\\ a_{i}\neq a_{j}}  }\beta_{n}^{m}P(a_{1},\cdots, a_{m})\omega_{a_1}\cdots \omega_{a_m}\, ,
\end{align}
where  $P(a_{1},\cdots,a_m)$ is the probability that the vertices $a_{1},\cdots, a_{m} \in D_{n}$ lie on a random path $p\in \Gamma_{n}$.    We seek a choice of  $\beta_{n}$ such that $W_{n}(\beta_{n})$ converges in law to a nontrivial limit with large $n$.  A reasonable method for finding a candidate for $\beta_{n}$ is to determine the appropriate scaling for the first-order term (i.e. the $m=1$ term) of~(\ref{Boondog}), which is merely a sum of i.i.d.\ random variables, and thus characterized by the central limit theorem.  This method of choosing $\beta_{n}$ turns out to work for $b<s$ but to be misleading in the case of $b=s$.
   The variance of the first-order term in~(\ref{Boondog}) has the form
\begin{align}\label{Blahh}
\textup{Var}\bigg(\beta_{n}\sum_{a\in D_{n} }P(a)\omega_{a}\bigg)  \,=\, \beta_{n}^2 \sum_{a\in D_{n}}\big(P(a)\big)^{2} \,=\,& \beta_{n}^2\sum_{k=1}^{n}(\text{$\#$  of $k^{th}$ generation vertices}) b^{-2k}\nonumber
\\  \,=\, &\beta_{n}^2\sum_{k=1}^{n}\frac{s-1}{s}\big(\frac{s}{b}\big)^{k} \nonumber
\\  \,\approx \, & \beta_{n}^2\frac{s-1}{s}\cdot \begin{cases}  \big(\frac{s}{b}\big)^{n}\big(1-\frac{b}{s}\big)^{-1}   \,, &  b<s \,, \\   \,n\,,  &  b=s\, . \end{cases}
\end{align}
Thus, by the central limit theorem, the first-order term from~(\ref{Boondog}) converges in law to a Gaussian when  $\beta_{n}=(b/s)^{n/2}$ and $\beta_{n}=n^{-1/2}$ in  the $b<s$ and $b=s$ cases, respectively.

 Now we perform the same exercise with the second-order term in~(\ref{Boondog}) to check whether the above scalings are still appropriate.  For $a_{1},a_{2}\in D_{n}$, the probability $P(a_{1},a_{2})$ has the form
$$P(a_{1},a_{2})=\epsilon(a_{1},a_{2})\exp\big\{ \frak{l}(a_{1},a_{2})  -\frak{g}(a_{1})-\frak{g}(a_{2}) \big\}\, , $$
where $\epsilon(a_{1},a_{2})$ is $0$ when no directed path passes through both of the vertices $a_{1}$ and $a_{2}$,  $ \frak{g}(a)$ is the generation of the vertex, and  $ \frak{l}(a_{1},a_{2}) $ is the smallest $k$ such that $a_{1},a_{2}\in g$ for some $g\in G_{n-k,n}$.    The variance of the second-order term has the form
\begin{align*}
\textup{Var}\Bigg(&\beta_{n}^2 \sum_{\substack{a_{1},a_{2}\in D_{n}\\ a_{1}\neq a_{2}} } P(a_{1},a_{2})\omega_{a_{1}}\omega_{a_{2}}\Bigg)   \, =\, \beta_{n}^4 \sum_{\substack{a_{1},a_{2}\in D_{n}\\ a_{1}\neq a_{2}}}\big(P(a_{1},a_{2})\big)^2
 \, =\, \beta_{n}^4  \sum_{k=1}^{n}\sum_{\substack{a_{1},a_{2}\in D_{n}\\  k= \frak{l}(a_{1},a_{2})}}   \big(P(a_{1},a_{2})\big)^2 \,.
\end{align*}
The self-similar structure of the diamond graphs allows the above to be written as
\begin{align*}
 \beta_{n}^4 \sum_{k=1}^{n} b^{-2j}|G_{n-k,n}|  \left( \frac{b(s-1)(s-2)   }{2}+  bs(s-1)\sum_{a\in D_{n-k}}\big(P(a)\big)^{2}+\frac{bs(s+1)}{2}\bigg(\sum_{a\in D_{n-k}}\big(P(a)\big)^{2}\bigg)^2  \right)  .
\end{align*}
The three terms above correspond to the respective cases  $\frak{g}(a_{1})=\frak{g}(a_{2})=k$,  $\frak{g}(a_{1})<\frak{g}(a_{2})=k$ (or vice versa), and   $\frak{g}(a_{1}) ,\frak{g}(a_{2}) <k$.  The third term is dominant, and we can apply that $|G_{n-k,n}|=(bs)^{k}$ and  our analysis in~(\ref{Blahh}) to get that the above is approximately
\begin{align*}
\, \approx \,& \beta_{n}^4 \sum_{k=1}^{n}\big(\frac{s}{b}\big)^{k}  \frac{bs(s+1)}{2}\Big( \frac{s-1}{s}\Big)^2 \cdot\begin{cases}  \big(\frac{s}{b}\big)^{2(n-k)}  \big(1-\frac{b}{s}\big)^{-2}\,,    &  b<s \,  , \vspace{.1cm} \\ (n-k)^2\,,  &  b=s\, . \end{cases}
\intertext{Finally, by standard summation formulas, this is equal to  }
  \,= \, & \beta_{n}^4\frac{ (s+1)(s-1)^2}{2} \cdot \begin{cases} \frac{(\frac{b}{s})^2}{(1-\frac{b}{s})^3} \big(\frac{s}{b}\big)^{2n} \,+\,\mathit{O}\big((\frac{s}{b})^{n}\big)  \,,  & \quad  b<s\, , \vspace{.1cm} \\  \frac{1}{3}n^{3}\,+\,\mathit{O}(n^2)\,,  & \quad  b=s \, .  \end{cases}
\end{align*}
For $b<s$ the choice $\beta_{n}=(b/s)^{n/2}$  again yields a variance  that converges  for the second-order term, as would be expected.  In contrast,  for $b=s$  the previous choice of $\beta_{n}=n^{-1/2}$ causes the variance of the second-order term to blow up with large $n$.   Notice, however, that if $\beta_{n}$ is scaled in proportion to $ n^{-1} $ in the  $b=s$ case, then the first- and second-order terms both have variances of the same order,  $\propto n^{-1}$.  This suggests that
$  \sqrt{n} \big( W_{n}\big(\BA /n\big)\,-\,1\big)    $
is an interesting rescaled version of the normalized partition function to study when $b=s$.   It is possible that there is a nontrivial limit for $W_{n}\big(\BA /\sqrt{n}\big)$ in the $b=s$ case, but the limiting distribution will not have finite variance and we do not explore this possibility here.

\section{A limit theorem in the $b<s$ case}\label{Secb<s}

In this section we prove Theorem~\ref{ThmMainIII}.  We first work on a slightly simpler problem in which random variables are placed on the edges of the diamond graph rather than the vertices. We use a coupling of Gaussian random variables across the different generations of the lattice to prove that a certain linearized version of the partition function is a Cauchy sequence, thereby establishing the existence of the $L_{r}^{b,s}$ laws. We then build on this result to show that the limit laws are universal, and finally translate these results from the edge case to the vertex case.

\subsection{The edge problem   }\label{Secb<sEdge}

The proof of the following lemma is below.

\begin{lemma}\label{LemExist} There exists a family of probability distributions $(L_{r}^{b,s})_{r >0}$  satisfying the properties (I)-(III) listed in Theorem~\ref{ThmMainIII}.
\end{lemma}

 The family of measures $(L_{r}^{b,s})_{r\geq 0}$ from Lemma~\ref{LemExist}   can be used to construct a  family of random probability measures on  the set of directed paths $\Gamma_{n}$ whose laws are consistent for all $n\in \mathbb{N}$.   For $k<n $ the path set $\Gamma_{k}$ on $D_{k}$ is canonically identified with a partition of  $\Gamma_{n}$ by relating paths that pass through the same collection of  $k^{th}$ generation vertices.   This consistency property means that $(L_{r}^{b,s})_{r\geq 0}$ can be used to construct random measures on the set of directed paths, $\Gamma_{\infty}$, on the continuum diamond lattice $D_{\infty}$, which we will not discuss further in this article.   Given $a\in E_{n}$ and $p\in \Gamma_{n}$, we will write $a\Try p$ to mean that the path $p$ passes through the edge $a$.

\begin{corollary}\label{CorCon}  Pick $x>0$.  Let $\mu^{b,s}_{k}$ denote a random probability  measure on $\Gamma_{k}$ such that for an independent family of random variables $\{ X( a)\}_{a\in E_{k}} $  with distribution $L_{x(b/s)^{k}}^{b,s}$ a path $p\in \Gamma_{k}$ is assigned probability
$$\mu^{b,s}_{k}(p)\,  =  \, \frac{ \prod_{a\Try p} X( a)     }{ \sum_{q\in \Gamma_{k} } \prod_{a\Try q} X( a)      } \,, $$
where the products are over the set of edges, $a$, lying along the path $p\in \Gamma_{k}$.   For $n>k$, let $S_{p}$ be the set of paths on $D_{n}$ that pass through the same set of $k^{th}$ generation vertices as $p$.  Then the measures $\mu^{b,s}_{k}$ are consistent for all $k\in \mathbb{N}$ in the sense that the family of random variables $\{  \mu^{b,s}_{n}(S_{p}) \}_{p \in \Gamma_{k}}$ has the same law as $\{  \mu^{b,s}_{k}(p)  \}_{p \in \Gamma_{k}}$.

\end{corollary}

\begin{proof}
 For $g\in G_{0,n}\equiv E_{n}$, let $X(g)$ be independent\ random variables with distribution $ L_{x(b/s)^n}^{b,s}  $.  We can define independent random variables  $X(g)$ labeled by $g\in G_{n-k,n}\equiv E_{k}$ through inductive use of the recursive relation
\begin{align*}
X(g)\, =\,  \frac{1}{b}\sum_{i=1}^{b} \prod_{1\leq j\leq s}X\big( g\TS (i,j)\big)\,.
\end{align*}
The $X(g)$, $g\in  E_{k}$ have distribution $ L_{x(b/s)^k}^{b,s}  $ as a consequence of property (II) listed in Theorem~\ref{ThmMainIII}.   Notice that
\begin{align*}
\mu^{b,s}_{n}\big(S_{p}\big) \, \stackrel{d}{=} \, \frac{   \sum_{q\in S_{p}}\prod_{a\Try q} X( a)     }{ \sum_{q\in \Gamma_{n} } \prod_{a\Try q} X( a)      }\, =\,  \frac{   \prod_{g\Try p} X( g)     }{ \sum_{q\in \Gamma_{k} } \prod_{g\Try q} X( g)      }\, \stackrel{d}{=} \,\mu_{k}^{b,s}(p)\,.
\end{align*}
Thus the laws of the measures are consistent.  This equality in law generalizes to the families $\{  \mu^{b,s}_{n}(S_{p}) \}_{p \in \Gamma_{k}}$, $\{  \mu^{b,s}_{k}(p)  \}_{p \in \Gamma_{k}}$.

\end{proof}

\begin{definition}\label{DefBFW}
Given numbers $X_{g}\in \R$ labeled by $g\in E_{n}\equiv G_{0,n}$, we define

\begin{itemize}
\item $\mathbf{W}_{n}\big(\{X_{g}\}_{g\in E_{n}}\big)$  to be equal to $\widehat{W}_{n}( D_{n})$ for the array of real numbers $\widehat{W}_{n}(g)$, $g\in G_{k,n}$ determined by the recursive relation
\begin{align}\label{Horses}
\widehat{W}_{n}(g)\, =\,  \frac{1}{b}\sum_{i=1}^{b} \prod_{1\leq j\leq s}\widehat{W}_{n}\big(g\TS (i,j)\big)
\end{align}
with initial condition $\widehat{W}(g)=X_{g}$ for $g\in  G_{0,n}$.

\item $\mathbf{\overline{W}}_{n}\big(\{X_{g}\}_{g\in E_{n}}\big)$   analogously to the above through the linearized recursive relation
\begin{align}\label{Horses}
\widehat{W}_{n}(g)\, =\,1\,+\, \frac{1}{b}\sum_{i=1}^{b}\sum_{1\leq j\leq s}\Big( \widehat{W}_{n}\big(g\TS (i,j)\big)\,-\,1\Big)\,.
\end{align}

\end{itemize}

\end{definition}

\begin{remark}\label{Bub} If $ \{ X_{a}\}_{a\in E_{n}}$ is an array of real numbers  and $k<n$, then
$$\mathbf{W}_{n}\big(\{X_{a}\}_{a\in E_{n}}\big) \, =\, \mathbf{W}_{k}\Big( \big\{\mathbf{W}_{n-k}\big(  \{X_{a}\}_{a\Try g}\big) \big\}_{g\in E_{k}}   \Big)\,,  $$
where we have abused notation by using the one-to-one correspondence between $E_{k}$ and $G_{n-k,n}$ and that $G_{n-k,n}$ defines a partition of $E_{n}$  to classify each edge $a\in E_{n}$ as a member of an edge $g$ in $E_{k}$.

\end{remark}

\begin{remark}\label{Triv} Once we have proved Theorem \ref{ThmMainIII} we will have the following: if $ \{\mathbf{x}_{g}\}_{g\in E_{k}}$ is an array of independent random variables  having distribution $L_{x}^{b,s}$, then
$\mathbf{W}_{k}\big( \{ \mathbf{x}_{g}   \}_{g\in E_{k}}  \big) $ has law $L_{ x(s/b)^{k}  }^{b,s}$.

\end{remark}

We will work towards a proof of the following  limit theorem.

\begin{theorem}[Edge limit theorem]\label{ThmEdge} For $n\in \mathbb{N}$, let
$X_{g}^{(n)}$ be an array of positive i.i.d.\  random variables labeled by $g\in  E_{n}$ with mean $1$ and variance $x_{n}\in \R^{+}$.   In addition, we assume that
$$  \lim_{n\rightarrow \infty}(s/b)^{n}   x_{n}\,=\,x  \hspace{1cm}\text{and}\hspace{1cm}  \lim_{\lambda\nearrow \infty} \sup_{n\in \mathbb{N}} \, (sb )^{n} \mathbb{P}\Big[  \big|X_{g}^{(n)}-1 \big|^2  > \lambda (b/s  )^{n} \Big] \,=\,0\,    $$
for some $x>0$.    Then there is weak convergence as $n\rightarrow \infty$ given by
$$  \mathbf{W}_{n}\big( \{ X_{g}^{(n)}  \}_{g\in E_{n}} \big) \hspace{.7cm}  \stackrel{\mathcal{L}}{\Longrightarrow}    \hspace{.7cm}   L_{x}^{b,s}\, ,     $$
where the family of probability distributions $(L_{r}^{b,s})_{r\geq 0}$ satisfies the properties listed in Theorem~\ref{ThmMainIII}.

\end{theorem}

\begin{definition}
Define $\widehat{M}:\R^+\rightarrow \R^+$ by
$$ \widehat{M}(x)\, :=\, \frac{1}{b}\big[ (1+x)^{s}\,-\,1     \big]  \,.    $$

\end{definition}

\begin{remark}\label{LilRemark}
 Let  $X_{g}$ be an array of positive i.i.d.\ random variables labeled by $g\in E_{k}$ having mean $1$ and variance $ x\in \R^{+}$.
The $k$-fold composition of the map $\widehat{M}$ yields
$$ \textup{Var}\Big( \mathbf{W}_{k}\big( \{ X_{g}  \}_{g\in E_{k}} \big) \Big) \, = \,   \widehat{M}^{k}(x)   \, .    $$
We will also make repeated use of the simple formula
\begin{align*}
\textup{Var} \Big( \mathbf{\overline{W}}_{k}\big(\{X_{g}\}_{g\in E_{k}}\big) \Big) = \left( \frac{s}{b} \right)^k x.
\end{align*}
\end{remark}

Next we collect some estimates on $\widehat{M}$.  The proof of Lemma~\ref{LemMMap} is placed in Section~\ref{SecProofsOne}.

\begin{lemma}\label{LemMMap} For any $\lambda>0$,  there is a $C>0$ such that the following inequalities hold for all $0\leq  x\leq \lambda$,   $N\in \mathbb{N}$, and $N\geq n$:

\begin{enumerate}[i)]

\item  $\widehat{M}^{N-n}\big(x(b/s)^{N}   \big) \,\leq \,  Cx  ( b/s)^{n}  $,

\item  $\widehat{M}^{N-n}\big(x(b/s)^{N}   \big) \, - \, x  ( b/s)^{n}  \, \leq \,   C x^2  ( b/s)^{2n} $, and

\item   $    \frac{d}{dx} \Big[ \widehat{M}^{n}\big( x(  b/s)^{n}\big) \Big] \, \leq \, C   $.

\item For unrestricted $x\geq 0$, we still  have a bound of the form
$$  \frac{d}{dx} \Big[ \widehat{M}^{n}\big( x(  b/s)^{n}\big) \Big]\, \leq \frac{s}{b}\exp\bigg\{ \frac{s-1}{1-\frac{b}{s}} \widehat{M}^{n}\big( x(  b/s)^{n}\big) \bigg\}\,. $$

\end{enumerate}

\end{lemma}

\begin{corollary}[Limiting variance function]\label{LemVarFun}

There is an increasing function $\frak{v}_{b,s}:\R^{+}\rightarrow \R^{+}$ such that for all $x\in \R^{+}$
\begin{enumerate}[I)]
\item  As $n\rightarrow \infty$,
$$\widehat{M}^{n}\big( x (b/s)^{n}   \big) \hspace{.5cm} \nearrow \hspace{.5cm} \frak{v}_{b,s}(x)\, .$$

\item $\frak{v}_{b,s}$  satisfies the relation
\begin{align*}
\frak{v}_{b,s}\big(\frac{s}{b}x\big)    \, = \,  \widehat{M}\big(  \frak{v}_{s,b}(x)  \big) \,.
\end{align*}

\item The derivative  of $\frak{v}_{b,s}$  satisfies $\lim_{x\searrow 0}\frak{v}_{b,s}'(x)=1$ and
\begin{align*}
  \frak{v}_{b,s}'(x)   \, =   \, & \prod_{j=1}^{\infty}  \Big(1+  \frak{v}_{b,s}\big(x(b/s)^{j}    \big)  \Big)^{s-1} \,  =\, \prod_{j=1}^{\infty}  \Big(1+  \widehat{M}^{-j}\big(\frak{v}_{b,s}(x) \big)  \Big)^{s-1}    \,.
\end{align*}

\end{enumerate}

\end{corollary}
\begin{proof} Properties (II) and (III) follow from the existence of the limit (I).   To see that the sequence $\widehat{M}^{n}\big( x (b/s)^{n}   \big)$ is Cauchy, notice that
for $N>n\gg 1$
\begin{align*}
\widehat{M}^{N}\big( x (b/s)^{N}   \big)  \,-\,  \widehat{M}^{n}\big( x (b/s)^{n}   \big)\, =& \, \widehat{M}^{n}\Big(\widehat{M}^{N-n}\big( x (b/s)^{N}\big)   \Big)  \,-\, \widehat{M}^{n}\big( x (b/s)^{n}   \big)\\
\leq & \, \Big[    \widehat{M}^{N-n}\big( x (b/s)^{N}\big)   \,-\, x (b/s)^{n} \Big] \frac{d}{dy}  \widehat{M}^{n}(y )\Big|_{y=\widehat{M}^{N-n}( x (b/s)^{N}) }\\
\leq & \, C x^{2}  \big( b/s\big)^{2n} \frac{d}{dy}  \widehat{M}^{n}(y )\Big|_{y=c (b/s)^{n} } \\
\leq & \, C^{2}x^2 \big( b/s\big)^{2n}\,.
\end{align*}
In the first inequality above, we have used that the derivative of $\widehat{M}$ is an increasing function.  The second inequality uses (ii) and (i)  of Lemma~\ref{LemMMap} for the first and second factors, respectively.   The third inequality uses (iii) of the same lemma.

 \vspace{.3cm}

\end{proof}

\begin{lemma}\label{LemBFW}
For each $n\in \mathbb{N}$, let $Y_{g}^{(n)}$ and $Z_{g}^{(n)}$ be arrays of i.i.d.\ random variables labeled by $g\in E_{n}$ such that $Y_{g}^{(n)}$ and $Z_{g}^{(n)}$ are uncorrelated, $\mathbb{E}\big[Y_{g}^{(n)}\big]=1$, $\mathbb{E}\big[Z_{g}^{(n)}\big]=0$, and for large $n>0$
$$\overline{y}\,:=\,  \sup_{n\in \mathbb{N} } ( s/b )^{n}\textup{Var}\big(Y_{g}^{(n)}  \big) \,  <  \, \infty\,\hspace{.7cm}\text{and}\hspace{.7cm}  \textup{Var}\big(Z_{g}^{(n)}  \big)  \, = \, \mathit{o}\big( (b/s)^{n}\big) \,. $$
 Then, the following holds:
\begin{enumerate}[i)]
\item  There is convergence to zero in probability as $n\rightarrow \infty$ given by
$$ \left|  \mathbf{W}_{n}\big(\{Y_{g}^{(n)}+Z_{g}^{(n)}\}_{g\in E_{n}}\big)\,-\,\mathbf{W}_{n}\big(\{Y_{g}^{(n)}\}_{g\in E_{n}}\big)    \right|   \hspace{.5cm}\stackrel{\mathcal{P}}{\Longrightarrow} \hspace{.5cm} 0\, .    $$

\item  For $g\in E_{n}$ and  $N>n $, the random variables  $\mathbf{W}_{N-n}\big(\{Y_{a}^{(N)}\}_{a\Try g}\big) -\mathbf{\overline{W}}_{N-n}\big(\{Y_{a}^{(N)}\}_{a\Try g }\big)$ and $\mathbf{\overline{W}}_{N-n}\big(\{Y_{a}^{(N)}\}_{a\Try g }\big)$ are uncorrelated, and there is a $C>0$ such that for all $N,n\in \mathbb{N}$ and $Y_{a}^{(N)}$
$$ \textup{Var}\Big(  \mathbf{W}_{N-n}\big(\{Y_{g}^{(N)}\}_{a\Try g}\big)  \, - \,    \mathbf{\overline{W}}_{N-n}\big(\{Y_{g}^{(N)}\}_{a \Try g}\big)\Big)  \, \leq  \, C\overline{y}^2( b/s )^{2n}   \,  .    $$

\end{enumerate}

\end{lemma}

\begin{proof}Part (i): Define $y_{n}:= (s/b)^{n}\textup{Var}\big( Y_{g}^{(n)} \big)      $  and $x_{n}:=  (s/b)^{n}\textup{Var}\big( Y_{g}^{(n)}+Z_{g}^{(n)}   \big)  $.   The random variables  $\mathbf{W}_{n}\big(\{Y_{g}^{(n)} +Z_{g}^{(n)} \}_{g\in E_{n}}\big)-\mathbf{W}_{n}\big(\{Y_{g}^{(n)} \}_{g\in E_{n}}\big)$ and $\mathbf{W}_{n}\big(\{Y_{g}^{(n)} \}_{g\in E_{n}}\big)$ are uncorrelated and thus we have first equality below.
\begin{align*}
\mathbb{E}\bigg[  \Big|  \mathbf{W}_{n}\big(\{Y_{g}^{(n)} +Z_{g}^{(n)} \}_{g\in E_{n}}\big)\,&-\,\mathbf{W}_{n}\big(\{Y_{g}^{(n)} \}_{g\in E_{n}}\big)    \Big|^2    \bigg]\\  &\, = \ \textup{Var}  \Big( \mathbf{W}_{n}\big(\{Y_{g}^{(n)} +Z_{g}^{(n)} \}_{g\in E_{n}}\big)  \Big)\,-\,\textup{Var}\Big(  \mathbf{W}_{n}\big(\{Y_{g}^{(n)} \}_{g\in E_{n}}\big)       \Big)  \\
&\, = \, \widehat{M}^{n}\big( x_{n}(b/s )^n   \big)\,-\,\widehat{M}^{n}\big( y_{n}(b/s  )^n   \big)
\\ &\, \leq \,       \frac{d}{dx}\Big[ \widehat{M}^{n}\big( x(b/s  )^n  \big)        \Big]\Big|_{x=x_{n}}   \big( x_{n}\,-\,y_{n}   \big)    \, .
\end{align*}
The second equality follows from Remark~\ref{LilRemark} and the  inequality uses that  $\widehat{M}$ has increasing derivative.  By our assumptions, $x_{n}$ is bounded for all $n$ and $|x_{n}-y_{n}|\rightarrow 0$.  Hence, we can apply (i) and (iii) of Lemma~\ref{LemMMap} to show that the above converges to zero with large $n$. \vspace{.3cm}

\noindent Part (ii): That the stated variables are uncorrelated follows from the definition of $\mathbf{W}$ and $\mathbf{\overline{W}}$. Since they are uncorrelated we have
\begin{align*}
 \textup{Var}\Big(  \mathbf{W}_{N-n}\big(\{Y_{a}^{(N)} \}_{a\Try g}\big)  \, - \,   \mathbf{\overline{W}}_{N-n}&\big(\{Y_{a}^{(N)} \}_{a\Try g}\big)\Big)\\  \, = \, & \textup{Var}\Big(  \mathbf{W}_{N-n}\big(\{Y_{a}^{(N)}\}_{a\Try g}\big) \Big) \, - \,   \textup{Var}\Big( \mathbf{\overline{W}}_{N-n}\big(\{Y_{a}^{(N)} \}_{a\Try g}\big)\Big)\\  \, = \, & \widehat{M}^{N-n}\big(y_{N}(b/s)^{N}   \big) \, -\, y_{N}(b/s)^{n} \\
\leq \,& C y_{N}^2(b/s)^{2n} \,\leq \,   C \overline{y}^2(b/s)^{2n}     \, ,
\end{align*}
where the first inequality holds by part (ii) of Lemma~\ref{LemMMap}.

\end{proof}


The proof of Lemma~\ref{LemExist} is a warm-up for the proof of  Theorem~\ref{ThmMainIII}, and we will refer back to some of the constructions in the proof.

\begin{proof}[Proof of Lemma~\ref{LemExist}]
Fix $r>0$.  For $n\in \mathbb{N}$ let $\{ \widehat{\omega}_{g}\}_{g\in E_n}$ be arrays of independent random variables with distribution
 $\mathcal{N}\big(0,\,  r(b/s)^{n}   \big) $ and satisfying a hierarchical consistency relation:
$$\widehat{\omega}_{g} \, =\,\frac{1}{b}\sum_{i=1}^{b}\sum_{1\leq j\leq s} \widehat{\omega}_{g\times (i,j)}\, .     $$
The above can be easily constructed using, for instance, two independent standard Brownian motions $\mathbf{B}^{(1)}$ and $\mathbf{B}^{(2)}$ and the fact that $ E_n$ has a canonical one-to-one correspondence with $(\{1,\cdots, b \}\times\{1,\cdots,s  \})^{n}$ to identify  the $\widehat{\omega}_{g}$, $g\in E_n$ with variables of the form
$$ b^{n}\Big( \mathbf{B}^{(1)}_{\frac{i+1}{b^n} } \,-\,\mathbf{B}^{(1)}_{\frac{i}{b^n} }   \Big) \Big( \mathbf{B}^{(2)}_{\frac{j+1}{s^n} } \,-\,\mathbf{B}^{(2)}_{\frac{j}{s^n} }   \Big)\,,$$
where $0\leq i\leq b^{n}-1$ and $0\leq j\leq s^{n}-1$.   This construction has the useful property that for $N>n$ and $g\in G_{n,N}$,
\begin{align}\label{Narbitsch}
1\,+\,\widehat{\omega}_{g}\, =\,\mathbf{\overline{W}}_{N-n}\big(\{ 1\,+\,\widehat{\omega}_{a}  \}_{a\Try g }\big)\, .
\end{align}

We will show that there is convergence in probability
\begin{align}\label{Babadu}
\mathbf{W}_{n}\big( \{ 1+ \widehat{\omega}_{g}   \}_{g \in E_{n}}   \big)\hspace{.5cm} \stackrel{\mathcal{P}}{\Longrightarrow} \hspace{.5cm}\big(\textbf{Some limit}\big)\, .
\end{align}
 It is enough to show that $\mathbf{W}_{n}\big( \{  \widehat{\omega}_{g}   \}_{g \in E_{n}}   \big)$ is Cauchy in probability. Notice that for $N>n$ the equality~(\ref{Narbitsch}) implies the first equality below
\begin{align*}
\mathbf{W}_{N}\big( \{ 1+ \widehat{\omega}_{a}   \}_{a \in E_{N}}   \big)\,-\, &\mathbf{W}_{n}\big( \{  1+\widehat{\omega}_{g}   \}_{g \in E_{n}}   \big)\ \\
\,& = \,\mathbf{W}_{N}\big( \{ 1+ \widehat{\omega}_{a}   \}_{a \in E_{N}}   \big)\,-\, \mathbf{W}_{n}\Big( \big\{ \mathbf{\overline{W}}_{N-n}\big(\{ 1\,+\,\omega_{a}  \}_{a\Try g }\big)  \big\}_{g \in E_{n}}   \Big)\, \\
\, & = \, \mathbf{W}_{n}\big(\{  Y_{g}^{(n)}+Z_{g}^{(n)}\}_{g\in E_{n}  } \big)\,-\,\mathbf{W}_{n}\big( \{   Y_{g}^{(n)} \}_{g\in E_{n}  } \big)\, ,
\end{align*}
where for $g\in  E_{n}\equiv G_{N-n,N}$
$$  Y_{g}^{(N)}\, :=\, \mathbf{\overline{W}}_{N-n}\big(\{1+\widehat{\omega}_{a}   \}_{a \Try g}\big)     \hspace{.5cm}\text{and} \hspace{.5cm}  Z_{g}^{(N)} \, :=\,\mathbf{W}_{N-n}\big(\{1+\widehat{\omega}_{a}   \}_{a\Try g}\big)\, -\, \mathbf{\overline{W}}_{N-n}\big(\{1+\widehat{\omega}_{a}   \}_{a\Try g}\big)   \,.  $$
In the above, we have used the  one-to-one correspondence between $E_{n}$ and  $G_{N-n,N}$ to classify edges  $a\in E_{N}$ as elements $g$ in $E_{n}$ (since $G_{N-n,N}$ defines a partition of $E_{N}$).  Notice that  $Y_{g}^{(n)}$ and $Z_{g}^{(n)}$ are uncorrelated and
$$       \textup{Var}\big( Y_{g}^{(n)}\big) \,=\, r(b/s )^{n}  \hspace{.6cm}\text{and}\hspace{.6cm}   \textup{Var}\big( Z_{g}^{(n)}\big)  \, =\,    \mathit{o}\big(  (b/s )^{n} \big)     \, ,     $$
where the order equality follows from part (ii) of Lemma~\ref{LemBFW}.  Hence we can apply part (i) of Lemma~\ref{LemBFW} to conclude that the difference between $\mathbf{W}_{n}\big(\{  Y_{g}^{(n)}+Z_{g}^{(n)}\}_{g\in E_{n}  } \big)$ and
$\mathbf{W}_{n}\big( \{   Y_{g}^{(n)} \}_{g\in E_{n}  } \big)$ converges in probability to zero with large $n$.\vspace{.3cm}

The properties (I)-(III) of Theorem \ref{ThmMainIII} follow closely from the above limit construction.  For  (II),  notice that as $n\rightarrow \infty$
 for each $(i,j)\in \{1,\cdots, b\}\times \{1,\cdots, s\}$ our result~(\ref{Babadu})  implies
$$ \mathbf{W}_{n-1}\big( \{ 1+ \widehat{\omega}_{g\TS(i,j)}   \}_{g \in E_{n}}   \big) \quad  \stackrel{\mathcal{L}}{\Longrightarrow}\quad  \, X_{(i,j)}\,, $$
where the $X_{(i,j)}$ are independent random variables with law $ L_{rb/s}^{b,s}$.  Moreover,
\begin{align}
\frac{1}{b}\sum_{i=1}^{b}\prod_{1\leq j\leq s} \mathbf{W}_{n-1}\big( \{ 1+ \widehat{\omega}_{g\TS(i,j)}   \}_{g \in E_{n}}\big)\,=\, \mathbf{W}_{n}\big( \{ 1+ \widehat{\omega}_{g}   \}_{g \in E_{n}}   \big) \quad   \stackrel{\mathcal{L}}{\Longrightarrow}\quad L_{r}^{b,s}\,,
\end{align}
where the convergence in law holds by another application of the limit result above. Thus (II) holds.  For (III), we write
\begin{align*}
&\frac{\mathbf{W}_{n}\big( \{ 1+ \widehat{\omega}_{g}   \}_{g \in E_{n}}   \big)\,-\,1}{\sqrt{r}}\\ &\hspace{1cm} \,=  \, \underbrace{\frac{\mathbf{W}_{n}\big( \{ 1+ \widehat{\omega}_{g}   \}_{g \in E_{n}}   \big)\,-\,\mathbf{\overline{W}}_{n}\big( \{ 1+ \widehat{\omega}_{g}   \}_{g \in E_{n}}   \big)}{\sqrt{r}}} \,+\, \underbrace{\frac{\mathbf{\overline{W}}_{n}\big( \{ 1+ \widehat{\omega}_{g}   \}_{g \in E_{n}}   \big)-1}{\sqrt{r}}} \,,\\
&\hspace{3.9cm} \stackrel{\mathcal{P}}{\Longrightarrow } 0 \text{ as } r\searrow 0  \hspace{4.1cm} \stackrel{d}{=}\,\mathcal{N}(0,1)
\end{align*}
where the convergence as $r\searrow 0$ of the first term is uniform in large $n$ by part (ii) of Lemma~\ref{LemBFW}. The normal distribution of the second term follows from \eqref{Horses} being a linear recursion and the initial variables being normal.

\end{proof}

\begin{proof}[Proof of Theorem~\ref{ThmEdge}]  Fix some $0<\epsilon\ll 1 $ and $n\in \mathbb{N}$ such that $x(b/s)^{n}< \epsilon $.  For $N>n$,  let $\{\widehat{\omega}_{a}\}_{a\in E_{N}} $ be an array of normal random variables defined as in the proof of Lemma~\ref{LemExist} for $r$ replaced by $x$.
By the method in the proof of Lemma~\ref{LemExist}, we can construct an array $\{\mathbf{x}_{g}\}_{g\in E_{n}} $ of independent random variables  with distribution   $L_{  x(b/s)^{n}  }^{b,s}$ through limits in probability as $N\rightarrow \infty$ of the form
\begin{align}\label{Crinkle}
 \mathbf{W}_{N-n}\big(   \{ 1+\widehat{\omega}_{a}   \}_{a\Try g }   \big) \hspace{.7cm} \stackrel{\mathcal{P}  }{\Longrightarrow}     \hspace{.7cm} \mathbf{x}_{g}   \, ,
\end{align}
where $\{ 1+\widehat{\omega}_{a}   \}_{a\Try g } $ refers to the subarray of $\{1 + \widehat{\omega}_{a}\}_{a\in E_{N}} $ comprised of all edges $a\in E_{N}$ lying on the subgraph $g\in  G_{N-n,N}\equiv E_{n} $.  Throughout this proof we will freely identify the sets $G_{N-n,N}$ and  $E_{n}$, allowing us to treat elements $g$ of $E_{n}$ as containing sets of edges from $ E_{N}$.   Notice that  $\mathbf{W}_{n}\big( \{ \mathbf{x}_{g}   \}_{g\in E_{n}}  \big) $ has law $L_{  x }^{b,s}$ by Remark~\ref{Triv}.   We will show that for arbitrary $F:\R\rightarrow \R$ with bounded derivative that in the limit $N\rightarrow \infty$
\begin{align}\label{GoGoGirls}
\mathbb{E}\Big[ F\Big(\mathbf{W}_{N}\big( \{  X_ {a}^{(N)}  \}_{a \in E_{N}  }  \big) \Big)   \Big]\hspace{.7cm}\longrightarrow  \hspace{.7cm}\mathbb{E}\Big[F\Big(\mathbf{W}_{n}\big( \{  \mathbf{x}_ {g}  \}_{g \in E_{n }  }  \big) \Big)   \Big] \, .
\end{align}
The above guarantees the weak convergence of $\mathbf{W}_{n}\big( \{  X_ {a}^{(n)}  \}_{a \in E_{n}  }  \big)$ to  $L_{  x }^{b,s}$.

The difference between the terms in~(\ref{GoGoGirls}) is bounded through the triangle inequality as follows:
\begin{align}\label{HerpADerp}
\bigg|\mathbb{E}\Big[  F\Big(\mathbf{W}_{N}\big( &\{  X_ {a}^{(N)}  \}_{a \in E_{N}  }  \big) \Big)   \Big] \,- \,\mathbb{E}\Big[F\Big(\mathbf{W}_{n}\big( \{  \mathbf{x}_ {g}  \}_{g \in E_{n }  }  \big) \Big)   \Big]\bigg| \nonumber  \\ \leq \, & \bigg|\mathbb{E}\Big[  F\Big(\mathbf{W}_{N}\big( \{  X_ {a}^{(N)}  \}_{a \in E_{N}  }  \big) \Big)   \Big] \,- \,\mathbb{E}\Big[F\Big(\mathbf{W}_{n}\Big( \big\{ \mathbf{\overline{W}}_{N-n}\big( \{ X_{a}^{(N)}\}_{a\Try g} \big)  \big\}_{g\in E_{n}}   \Big) \Big)   \Big]\bigg|\nonumber  \\
&+\,\bigg|\mathbb{E}\Big[F\Big(\mathbf{W}_{n}\Big( \big\{ \mathbf{\overline{W}}_{N-n}\big( \{ X_{a}^{(N)}\}_{a\Try g} \big)  \big\}_{g\in E_{n}}   \Big) \Big)   \Big]\, - \,  \mathbb{E}\Big[F\Big(\mathbf{W}_{n}\big( \{  1+\widehat{\omega}_ {g}  \}_{g \in E_{n }  }  \big) \Big)   \Big]  \bigg| \nonumber  \\
&+\, \bigg|   \mathbb{E}\Big[F\Big(\mathbf{W}_{n}\big( \{  1+\widehat{\omega}_ {g}  \}_{g \in E_{n }  }  \big) \Big)   \Big]\, -\,     \mathbb{E}\Big[F\Big(\mathbf{W}_{n}\Big(\big\{ \mathbf{W}_{N-n}\big(   \{ 1+\widehat{\omega}_{a}    \}_{a\Try g }   \big)\big\}_{g\in E_{n} } \Big) \Big)   \Big] \bigg| \nonumber \\
&+\, \bigg|    \mathbb{E}\Big[F\Big(\mathbf{W}_{n}\Big(\big\{ \mathbf{W}_{N-n}\big(   \{ 1+\widehat{\omega}_{a}  \}_{a\Try g }   \big)\big\}_{g\in E_{n} } \Big) \Big)   \Big] \, -\,     \mathbb{E}\Big[F\Big(\mathbf{W}_{n}\big( \{  \mathbf{x}_ {g}  \}_{g \in E_{n }  }  \big) \Big)   \Big] \bigg| \,.
\end{align}
It is sufficient to show that the terms above are each bounded by a constant multiple of $\epsilon >0$ in the limit $N\rightarrow \infty$.
 The last term on the right side of~(\ref{HerpADerp}) converges to zero as $N\rightarrow \infty$ by~(\ref{Crinkle}) and  since $\mathbf{W}_{n}\big(\{ a_{g} \}_{g\in E_{n}}\big)$ is a continuous function (a multivariable polynomial) of the array $ \{ a_{g} \}_{g\in E_{n}}$.  We will treat the first three terms on the right side of~(\ref{HerpADerp}) by proving the following statements:
\begin{enumerate}[(I)]

\item  There is a $C>0$  such that for all $N,n\in \mathbb{N}$ with $N>n$
$$ \textup{Var}\bigg( \mathbf{W}_{N}\big( \{ X_{g}^{(N)}\}_{g\in E_{N}}\big)\, -\, \mathbf{W}_{n}\Big( \big\{ \mathbf{\overline{W}}_{N-n}\big( \{ X_{a}^{(N)}\}_{a\Try g} \big)  \big\}_{g\in E_{n}}   \Big)    \bigg)\, \leq  \, Cx^2( b/s  )^{n}   \, .  $$
Consequently, the first term on the right side of~(\ref{HerpADerp})   is bounded by $\epsilon \sqrt{C} \sup_{x\in \R}| F'(x) | $.

\item   For any fixed $n$,  as $N\rightarrow \infty$
$$\mathbb{E}\Big[F\Big(\mathbf{W}_{n}\Big( \big\{ \mathbf{\overline{W}}_{N-n}\big( \{ X_{a}^{(N)}\}_{a\Try g} \big)  \big\}_{g\in E_{n}}   \Big) \Big)   \Big]\hspace{.5cm}\longrightarrow  \hspace{.5cm} \mathbb{E}\Big[F\Big(\mathbf{W}_{n}\big( \{  1+\widehat{\omega}_ {g}  \}_{g \in E_{n }  }  \big) \Big)   \Big]\, . $$

\item  There is a $C>0$ such that for all $n\in \mathbb{N}$
$$ \textup{Var}\bigg(\mathbf{W}_{n}\big( \{  1+\widehat{\omega}_ {g}  \}_{g \in E_{n }  }  \big) \,-\,\mathbf{W}_{n}\Big(\big\{ \mathbf{W}_{N-n}\big(   \{ 1+\widehat{\omega}_{a}    \}_{a\Try g }   \big)\big\}_{g\in E_{n} } \Big) \bigg)\, \leq \,  Cx^2( b/s  )^{n}   \, .  $$
Hence the third term on the right side of~(\ref{HerpADerp})   is bounded by $\epsilon \sqrt{C}\sup_{x\in \R}| F'(x) |   $.

\end{enumerate}

\noindent (I)  For $g\in E_{n}$, define
$$ Y_{g}^{(N)}\,:=\, \mathbf{\overline{W}}_{N-n}\big( \{ X_{a}^{(N)}\}_{a\Try g} \big) \hspace{.7cm}\text{and}\hspace{.7cm}Z_{g}^{(N)}\,:=\, \mathbf{W}_{N-n}\big( \{ X_{a}^{(N)}\}_{a\Try g} \big)\,-\, \mathbf{\overline{W}}_{N-n}\big( \{ X_{a}^{(N)}\}_{a\Try g} \big) \, .   $$
Notice that
$$ \textup{Var}\big(  Y_{g}^{(N)}  \big)\, = \, x_{N} (b/s)^{n}  \hspace{.7cm}\text{and}\hspace{.7cm}    \textup{Var}\big(  Z_{g}^{(N)}  \big)\, \leq  \  c x^2(b/s)^{2n} \,,   $$
where the inequality holds for some $c>0$ and all $N$, $n$ by part (ii) of Lemma~\ref{LemBFW}.   The variance of the  difference between $\mathbf{W}_{N}\big( \{ X_{a}^{(N)}\}_{a\in E_{N}}\big) $
and $\mathbf{W}_{n}\big( \big\{ \mathbf{\overline{W}}_{N-n}\big( \{ X_{g}^{(N)}\}_{a\Try g} )  \big\}_{g\in E_{n}}   \big)$ can be written as
\begin{align*}
 \,\textup{Var}\bigg( \mathbf{W}_{n}\Big( \big\{ Y_{g}^{(N)} +Z_{g}^{(N)}  \big\}_{g\in E_{n}}   \Big)\,-\,\mathbf{W}_{n}\Big( \big\{ Y_{g}^{(N)} \big\}_{g\in E_{n}}   \Big)\bigg) \, ,
\end{align*}
which is bounded by a constant multiple of $x^{2}(b/s)^{n}$ for all $N$ and $n$ by the bound in the proof of part (i)
 of Lemma~\ref{LemBFW}.

\vspace{.3cm}

\noindent (II):  For $a\in E_{N}$ define $\mathbf{x}_{a}^{(N)} := (s/b)^{ N/2} \big(X_{a}^{(N)}-1\big)$.  Then, by our assumptions on $X_{a}^{(N)}$, the random variables $\mathbf{x}_{a}^{(N)}$ are i.i.d.\ with mean zero, variance $(s/b)^{ N} x_{N}$, and satisfy a Lindeberg condition.   We can write $\mathbf{\overline{W}}_{N-n}\big( \{ X_{a}^{(N)}\}_{a\Try g} )$ in the form
\begin{align*}
&\hspace{1.5cm}\mathbf{\overline{W}}_{N-n}\big( \{ X_{a}^{(N)}\}_{a\Try g} \big)\, =\, 1\,+\,\big(\frac{b}{s}\big)^{\frac{n}{2}}  \Bigg[\underbrace{\frac{1}{(bs)^{\frac{N-n}{2}}}\sum_{a\Try g}\mathbf{x}_{a}^{(N)}}  \Bigg]\, ,\\
&\hspace{6.3cm}\big(\textbf{Central limit-type sum}\big)
\end{align*}
where the sum in square brackets includes $|g|=(sb)^{N-n}$ elements.  By the central limit theorem,  there is convergence in law as $N\rightarrow \infty$
\begin{align*}
\mathbf{\overline{W}}_{N-n}\big( \{ X_{a}^{(N)}\}_{a\Try g} \big)\, -\, 1\hspace{.5cm}\stackrel{\mathcal{L}}{\Longrightarrow} \hspace{.5cm} \widehat{\omega}_{g}\,
\end{align*}
since $ \widehat{\omega}_{g}$ has distribution $\mathcal{N}\big(0,x(b/s)^{n}\big)$. The result then follows since the copies of $\mathbf{\overline{W}}_{N-n}\big( \{ X_{a}^{(N)}\}_{a\Try g} \big)$  for  different $g\in E_{n}$ are independent and $h(\mathbf{y}):=\mathbb{E}\big[F\big( \mathbf{W}_{n}( \{\mathbf{y}_{g}\}_{g\in E_{n}  }\big)    \big]$
is a continuous function of arrays $\mathbf{y}\in \R^{(bs)^{n}}$ (recall that $|E_{n}|=b^n s^n$).

\vspace{.3cm}

\noindent (III):  The argument is similar to (I).   Define
\begin{align*}
Y_{g}\,:=\, 1+\widehat{\omega}_ {g}  \hspace{.7cm}\text{and}\hspace{.7cm} Z_{g}^{(N)}\,:=\, \mathbf{W}_{N-n}\big( \{  1+\widehat{\omega}_ {a} \}_{a\Try g} )\,-\, \mathbf{\overline{W}}_{N-n}\big( \{ 1+\widehat{\omega}_ {a}\}_{a\Try g} ) \, .
\end{align*}
 Notice that for $N>n$
$$Y_{g}=\mathbf{\overline{W}}_{N-n}\big( \{ 1+\widehat{\omega}_{a}\}_{a\Try g} )\, ,\hspace{.7cm} \textup{Var}\big(  Y_{g}  \big)\, = \, x (b/s)^{n} \,, \hspace{.7cm}\text{and}\hspace{.7cm}    \textup{Var}\big(  Z_{g}^{(N)}  \big)\, \leq  \  c x^2(b/s)^{2n} \,,   $$
where the inequality holds for some $c>0$ and all $N$, $n$ by part (ii) of Lemma~\ref{LemBFW}.  Thus we can apply the inequalities  in the proof of part (i) of Lemma~\ref{LemBFW} to bound (III) since
\begin{align*}
\textup{Var}\bigg(\mathbf{W}_{n}\big( \{  1+&\widehat{\omega}_ {g}  \}_{g \in E_{n }  }  \big) \,-\,\mathbf{W}_{n}\Big(\big\{ \mathbf{W}_{N-n}\big(   \{ 1+\widehat{\omega}_{a}    \}_{a\Try g }   \big)\big\}_{g\in E_{n} } \Big) \bigg)\\
&= \,\textup{Var}\bigg( \mathbf{W}_{n}\Big( \big\{ Y_{g} +Z_{g}^{(N)}  \big\}_{g\in E_{n}}   \Big)\,-\,\mathbf{W}_{n}\Big( \big\{ Y_{g} \big\}_{g\in E_{n}}   \Big)\bigg) \, .
\end{align*}

\end{proof}

\subsection{Reducing the vertex problem to the edge problem}

 Now we work towards a proof of Theorem~\ref{ThmMainIII} by approximating the original partition function by the partition function for a model in which the randomness is placed on the edges and is thus  applicable to  Theorem~\ref{ThmEdge}.

\begin{definition}
Define $M_{n}:\R^{+}\rightarrow \R^{+}$ by
$$ M_{n}(x):=\frac{1}{b}\Big[\big( 1+x    \big)^s e^{(s-1)[  \lambda(2\beta_{n})-2\lambda(\beta_{n})   ] }   \,-\,1\Big] \,. $$
\end{definition}

\begin{remark}
Note that for $g\in G_{n-k,n}$ we have  $\textup{Var}\big( W_{n}(\beta_{n}; g)  \big) =M_{n}^{k}(0)$.
\end{remark}
The proof of the following technical lemma is in Section~\ref{SecProofsOne}

\begin{lemma}\label{Initial} Let $\beta_{n}=\BA (b/s)^{n/2}$ and $g\in G_{m,n}$.   There is a $C>0$ such that for all $n\in \mathbb{N}$ and $m\leq n/2$
\begin{enumerate}[i)]
\item  $
\Big|\textup{Var}\big( W_{n}(\beta_{n}; g)  \big)   \,-\,\BA^2\frac{s-1}{s-b}\big( 1-(b/s)^{m}\big)(b/s)^{n-m} \Big| \, \leq \,  C (b/s)^{n}\,  $

\item   $\mathbb{E}\Big[ \big|W_{n}(\beta_{n}; g)-1\big|^4    \Big]\leq C(b/s)^{ 2(n-m)}\, .$

\end{enumerate}

\end{lemma}

\begin{theorem}[Variance convergence]\label{LemVarIII}Define $\ds \alpha_{n}:=\BA^2\frac{s-1}{s-b}(b/s)^{\lfloor n/2\rfloor}  $.   As $n\rightarrow \infty$
\begin{align*}
\big| M_{n}^{n}(  0) \,- \,  \widehat{M}^{\lceil\frac{n}{2}\rceil }(  \alpha_{n} ) \big| \,\longrightarrow \,0\, .
\end{align*}
In particular, by combining the above with part (i) of  Lemma~\ref{LemVarFun}, we have that  $M_{n}^{n}(  0) $ converges to $ \frak{v}_{s,b}\big(\BA^2\frac{s-1}{s-b}\big)$ with large $n$.

\end{theorem}

\begin{proof}For $\gamma_{n} :=M_{n}^{\lfloor n/2\rfloor }(0)$, we have the relations
\begin{align}
\big| M_{n}^{n}(  0)\,-\,\widehat{M}^{\lceil  \frac{n}{2}\rceil }(  \alpha_{n} ) \big| \,= &\,\big| M_{n}^{\lceil \frac{n}{2}\rceil }( \gamma_{n} )\,-\,\widehat{M}^{\lceil \frac{n}{2}\rceil }(  \alpha_{n} ) \big| \nonumber \\
\leq &   \,\big| M_{n}^{\lceil \frac{n}{2}\rceil }( \gamma_{n} )\,-\,\widehat{M}^{\lceil \frac{n}{2}\rceil }(  \gamma_{n} ) \big|\,+\,  \big| \widehat{M}^{\lceil \frac{n}{2}\rceil }( \gamma_{n} )\,-\,\widehat{M}^{\lceil\frac{n}{2}\rceil }(  \alpha_{n} ) \big|\, .\label{Forbes}
\end{align}
Thus we just need to bound the two terms on right side above.  For the second term above,
\begin{align*}
  \big| \widehat{M}^{\lceil \frac{n}{2}\rceil }( \gamma_{n} )\,-\,\widehat{M}^{\lceil \frac{n}{2}\rceil}(  \alpha_{n} ) \big|\, \leq \, |\gamma_{n}-\alpha_{n}| \frac{d}{dy}\Big[ \widehat{M}^{\lceil \frac{n}{2}\rceil }( y)   \Big]\Big|_{y=\max(\alpha_{n},\gamma_{n})} \, \leq \,& C (b/s )^{n}\,,
\end{align*}
where we have applied part (i) of  Lemma~\ref{Initial} to control the difference $\gamma_{n} -\alpha_{n}$  and part (iii) of Lemma~\ref{LemMMap} to control the derivative of $ \widehat{M}^{\lceil n/2\rceil }$.

Now we will bound the first term on the right side of~(\ref{Forbes}).  For $x$ in a bounded interval $[0,L]$,  there is a $c>0$ such that
 \begin{align}
  \widehat{M}(x)\, \leq \,   M_{n}(x) \, \leq \,  \widehat{M}(x)+c(b/s)^{n}   \,. \label{SonsOfHarpy}
\end{align}
By using a telescoping sum, we can write the difference between $ M_{n}^{m }( \gamma_{n} )$ and $\widehat{M}^{m }(  \gamma_{n} ) $ as
\begin{align*}
 M_{n}^{m }( \gamma_{n} )\,-\,\widehat{M}^{m }(  \gamma_{n} ) \,  =\,& \sum_{k=0}^{m-1}\Big(\widehat{M}^{m-k-1}\big(M_{n}^{k+1 }( \gamma_{n} )\big)\,-\,\widehat{M}^{m-k }\big(M_{n}^{k }( \gamma_{n} )\big)\Big) \\
 \leq \,&  \sum_{k=0}^{m-1}\bigg[\frac{d}{dy}\widehat{M}^{m-k-1}(y)\Big|_{y=M_{n}^{k+1 }( \gamma_{n} )}        \bigg]\Big( M_{n}\big(M_{n}^{k }( \gamma_{n} )\big)\,-\,\widehat{M}\big(M_{n}^{k }( \gamma_{n} )\big)\Big)\,.
 \intertext{Let $\widetilde{n}\in \mathbb{N}$ be the smallest $m$ such that $M_{n}^{m }( \gamma_{n} ) >\widehat{M}^{m+1 }(  \gamma_{n} )$ or  $\widehat{M}^{m }(  \gamma_{n} )>L$.  For $m< \widetilde{n} $  the above is }
 \leq \,&  cm\big(\frac{b}{s}\big)^{n}\max_{0\leq k< m}   \frac{d}{dy}\widehat{M}^{m-k-1}(y)\Big|_{y=\widehat{M}^{k+2 }( \gamma_{n} )}     \\
\leq \,&  cm\big(\frac{b}{s}\big)^{n} \max_{0\leq k< m} \big(\frac{s}{b}\big)^{m- k}\exp\bigg\{\frac{s-1}{1-\frac{b}{s}}    \widehat{M}^{m}(\gamma_{n} )   \bigg\}\,,
\intertext{where we have applied (iv) of Lemma~\ref{LemMMap} to bound the derivative of $\widehat{M}$.  For $c':=c\exp\big\{ \frac{s-1}{1-\frac{b}{s} }L \big\}  $, the last expression is}
 \leq \,&  c'm\big(\frac{b}{s}\big)^{\lfloor \frac{n}{2}\rfloor}        \,.
\end{align*}

Thus we have that
$$  M_{n}^{m }( \gamma_{n} )\,\leq \,\widehat{M}^{m }(  \gamma_{n} )  \,+\,c'm (b/s)^{\lfloor \frac{n}{2}\rfloor} \,.  $$
for $m< \widetilde{n}$.  The value $\widetilde{n}$ must be greater than $n$ for  $n\gg 1$ since
$\widehat{M}^{m+1 }(  \gamma_{n} )$  has the the lower bound
\begin{align}
\widehat{M}^{m+1 }(  \gamma_{n} ) \, = \,\widehat{M}\big(\widehat{M}^{m}(  \gamma_{n} )\big) \,  > \,\frac{s}{b}\widehat{M}^{m}(  \gamma_{n} )  \,> &\, \widehat{M}^{m}(  \gamma_{n} )\,+\,\frac{s-1 }{b} \gamma_{n}\,,
\end{align}
and  $\gamma_{n}$ has order $(b/s)^{\lfloor n/2\rfloor}$ by Lemma~\ref{Initial}.  Therefore
the first term on the right side of~(\ref{Forbes}) decays with order $n(b/s)^{^{\lfloor n/2\rfloor}} $ for large $n$.

\end{proof}

It is convenient to define a model for which the disorder at ``low" generation vertices is neglected.  We define ``lower" generation to be $ \leq  \lceil n/2\rceil$.  This effectively means that we can work on a problem of size $\lceil n/2\rceil $   in which randomness is placed  on edges.

\begin{proof}[Proof of Theorem~\ref{ThmMainIII}]
Recall that $\beta_n = \BA(b/s)^{n/2}$. For $g\in G_{ \lfloor n/2\rfloor ,n}\equiv E_{ \lceil n/2 \rceil}$, define $X_{g}^{(n)}:=W_{n}(\beta_{n}; g)$.   Notice that the array of random variables $\{X_{g}^{(n)}\}_{g\in E_{ \lceil n/2\rceil}}$ satisfies the conditions of Theorem~\ref{ThmEdge} as a consequence of Lemma~\ref{Initial}.   We can write $W_{n}(\beta_{n})=W_{n}(\beta_{n}; D_{n}) $ as
$$  W_{n}(\beta_{n})\, = \,  \mathbf{W}_{ \lceil \frac{n}{2}\rceil }\big(\{ X_{g}^{(n)}  \}_{g\in E_{ \lceil \frac{n}{2}\rceil}}   \big)\,+ \, \Big( W_{n}(\beta_{n}; D_{n})  \,-\,\mathbf{W}_{ \lceil \frac{n}{2}\rceil }\big(\{ X_{g}^{(n)}  \}_{g\in E_{ \lceil \frac{n}{2}\rceil}}   \big)\Big) \,.  $$
We will show below that the difference between $W_{n}(\beta_{n}; D_{n}) $ and $\mathbf{W}_{ \lceil  n/2\rceil }\big(\{ X_{g}^{(n)}  \}_{g\in E_{ \lceil  n /2\rceil}}   \big) $ converges in probability to zero. Thus the convergence in law of $W_{n}(\beta_{n})$ to $L_{\BA^2\frac{s-1}{s-b}}$ follows from Theorem~\ref{ThmEdge}.

The random variables   $W_{n}(\beta_{n}; D_{n})- \mathbf{W}_{  \lceil n/2\rceil }\big(\{ X_{g}^{(n)}  \}_{g\in E_{ \lceil n/2\rceil}}   \big)    $ and  $ \mathbf{W}_{  \lceil  n/2\rceil }\big(\{ X_{g}^{(n)}  \}_{g\in E_{ \lceil  n/2\rceil}}   \big)    $ are uncorrelated, so
\begin{align*}
 \mathbb{E}\bigg[ \Big(W_{n}(\beta_{n}; D_{n}) - \mathbf{W}_{  \lceil \frac{n}{2}\rceil }\big(\{ X_{g}^{(n)}  \}_{g\in E_{ \lceil \frac{n}{2}\rceil}}   \big)  \Big)^2 \bigg] \,=& \, \textup{Var}\Big( W_{n}(\beta_{n}; D_{n})   \Big) \, - \, \textup{Var}\Big( \mathbf{W}_{  \lceil \frac{n}{2}\rceil }\big(\{ X_{g}^{(n)}  \}_{g\in E_{ \lceil \frac{n}{2}\rceil}}   \big)   \Big)\\  \, = &\, M_{n}^{n }(0 )\,-\,\widehat{M}^{ \lceil \frac{n}{2}\rceil }(  \alpha_{n} ) \, .
\end{align*}
  The above converges to zero by Theorem~\ref{LemVarIII}.

\end{proof}

\subsection{Transition to strong disorder in the distributions $L_{r}^{b,s}$ }

Let $(L_{r}^{b,s})_{r\geq 0}$  be the family of distributions satisfying the properties (I)-(III) listed in Theorem~\ref{ThmMainIII}.  It is clear from the definition that $L_{r}^{b,s}$ converges to the delta mass at $1$ as $r\searrow 0$.  The limiting behavior as $r \rightarrow \infty$ is less obvious, although it is reasonable to expect that the distributions  transition to strong disorder behavior, in other terms, $L_{r}^{b,s}$ converges to the delta mass at $0$. This is indeed the case as the next lemma shows.    Our proof will follow the homogeneous environment tilting method used by Lacoin and Moreno in Section 5.2 of \cite{lacoin}, applied to random weights on  bonds rather than sites.

\begin{lemma}\label{LemTransition}
 $L_{r}^{b,s}$ converges weakly to a $\delta$-distribution at zero as $r\rightarrow \infty$.

\end{lemma}

\begin{proof}
Let $W_{r}$ have distribution $L_{r}^{b,s}$.  It suffices to show that the fractional moment $\mathbb{E}\big[  W_{r}^{1/2}  \big]$ converges to zero as $r\rightarrow \infty$.    By the proof of Lemma~\ref{LemExist}, $W_{r}$ can be constructed through a weak  limit   $N\rightarrow \infty$ of
\begin{align}
\label{Dino}
\mathbf{W}_{N}\bigg(  \Big\{ 1+ \sqrt{r}\big( \frac{b}{s} \big)^{\frac{N}{2}}\widehat{\omega}_{a}  \Big\}_{a\in E_{N}}  \bigg)\,,
\end{align}
where $\mathbf{W}_{N}$ is defined as in Definition~\ref{DefBFW} and $ \{\widehat{\omega}_{g}\}_{g\in E_{N}}$ is an array of independent random variables with distribution $\mathcal{N}\big(0,1\big) $.
The proof of Lemma~\ref{LemExist} also holds with the random variables
$$  W_{r}^{(N)}\,:=\,\mathbf{W}_{N}\Bigg(  \bigg\{ \exp\Big\{\sqrt{r}\big( \frac{b}{s} \big)^{\frac{N}{2}}\widehat{\omega}_{a} -\frac{r}{2}\big( \frac{b}{s} \big)^{N}\Big\}   \bigg\}_{a\in E_{N}} \Bigg)\, $$
in place of~(\ref{Dino}) without any serious changes, and we will do our computation with this expression.  Allow us to  abuse notation by identifying measures with their corresponding  expectation symbols.   Define a new measure $\widetilde{\mathbb{E}}$ on the system with density
$$ \frac{  d\widetilde{\mathbb{E}}}{d \mathbb{E} } \, := \, \exp\Bigg\{-\sum_{a\in E_{N} }\Big( (bs)^{-\frac{N}{2}}\widehat{\omega}_{a}  +\frac{1}{2}(bs)^{-N}\Big) \Bigg\}  \, .$$
By Cauchy-Schwarz, we can write
 \begin{align}\label{Gimmy}
\mathbb{E}\big[ W_{r,N}^{\frac{1}{2}}  \big] \,=\,\widetilde{\mathbb{E}}\bigg[ \frac{d \mathbb{E} }{  d\widetilde{\mathbb{E}}} W_{r,N}^{\frac{1}{2}}  \bigg]  \, \leq \,  \widetilde{\mathbb{E}}\bigg[ \Big(\frac{d \mathbb{E} }{  d\widetilde{\mathbb{E}}} \Big)^2 \bigg]^{\frac{1}{2}}   \widetilde{\mathbb{E}}\big[ W_{r,N}  \big]^{\frac{1}{2}} \,.
\end{align}
However, computations with Gaussian integrals yield that
\begin{align*}
 \widetilde{\mathbb{E}}\bigg[ \Big(\frac{d \mathbb{E} }{  d\widetilde{\mathbb{E}}} \Big)^2 \bigg]^{\frac{1}{2}}\,=\,e^{ \frac{1}{2} } \hspace{1cm}\text{and}\hspace{1cm} \widetilde{\mathbb{E}}\big[ W_{r,N}  \big]^{\frac{1}{2}} \,=\,e^{ -\frac{1}{2}\sqrt{r} } \,.
\end{align*}
The first equation above uses that $|E_{N}|=(bs)^{N}$ and the second uses that each path in $\Gamma_{N}$ contains $s^{N}$ edges.  By applying these inequalities in~(\ref{Gimmy}), we get that $\mathbb{E}\big[ W_{r,N}^{1/2}  \big] $ is smaller than $\exp\{ (1-\sqrt{r})/2  \}$.  For any fixed $r\in \R^{+}$,
$$\mathbb{E}\big[ W_{r,N}^{\frac{1}{2}}  \big] \quad \longrightarrow \quad \mathbb{E}\big[  W_{r}^{\frac{1}{2}}  \big]  $$
as $N\rightarrow \infty$ since $W_{r,N}\stackrel{\mathcal{L}}{\Rightarrow} W_{r}$ and the second moment of $W_{r,N}$ is uniformly bounded for all $N>0$. Therefore, $\mathbb{E}\big[ W_{r,N}^{1/2}  \big] $ is smaller than $\exp\{ (1-\sqrt{r})/2  \}$, which converges to zero as $r\rightarrow \infty$.

\end{proof}

\section{Results for the $b=s$ case}\label{SecWeakDisorder}

In this section we prove Theorem~\ref{ThmMain}.

\subsection{Definitions relevant for the $s=b$ case}

\begin{remark} When considering $\beta = \BA /n\ll 1$, it will be convenient to define
$$ R_{n}(\BA; g)\, : =\,   \sqrt{n}\Big(W_{n}\Big(\frac{\BA}{n}; g\Big)-1 \Big) \, .$$
An explicit recursive relation for $R_{n}(\BA; g)$ can be deduced from~(\ref{WRecur}) starting from $R_{n}(\BA; g)=0$ for $g\in G_{0,n}$:
\begin{align}\label{Full}
R_{n}(\BA; g)\, =\,\,   \sqrt{n}\Bigg[\frac{1}{b}\sum_{i=1}^{b} \prod_{1\leq j\leq b}\bigg(1+\frac{1}{\sqrt{n}}R_{n}\big(\BA; g\TS (i,j)\big)\bigg)\prod_{1\leq j\leq b-1} E_{g\diamond (i,j)}\Big(\frac{\BA}{n}   \Big) \,-\, 1\Bigg]\,.
\end{align}
Notice that for $n\gg 1$ the random variables satisfy the asymptotic relation
\begin{align}\label{Partial}
 R_{n}(\BA; g)\, =& \,\, \frac{1}{b}\sum_{i=1}^{b} \sum_{1\leq j\leq b}R_{n}\big(\BA; g\TS (i,j)\big)  \, +\, \frac{1}{b\sqrt{n}}\sum_{i=1}^{b} \sum_{1\leq j_{1}<j_{2} \leq b}R_{n}\big(\BA; g\TS (i,j_{1})\big) R_{n}\big(\BA; g\TS (i,j_{2})\big)\nonumber  \\  &\, +\,  \frac{\BA}{b\sqrt{n}}\sum_{i=1}^{b} \sum_{1\leq j\leq b-1}\omega_{g\diamond (i,j)} \, +\,\textbf{(Lower-order terms)}\, .
\end{align}
The above uses that $ \omega_{ g\diamond (i,j)}$ has mean zero through the approximation
$$E_{g\diamond (i,j)}\Big(\frac{\BA}{n}\Big)\, = \, \frac{ e^{\frac{\BA}{n} \omega_{ g\diamond (i,j)} }   }{ \mathbb{E}\big[   e^{\frac{\BA}{n} \omega_{ g\diamond (i,j)} } \big]   }\, \approx \,   1\, + \, \frac{\BA}{n}\omega_{ g\diamond (i,j)} \,+\,\mathit{O}\Big(\frac{1}{n^2}\Big)   \, .  $$
The third term on the  right-hand side of~(\ref{Partial})  introduces independent noise  at each step.
The variance of $R_{n}(\BA; g)$ is given by $\varrho_{k}^{(n)}(\BA)=n\sigma_{k}(\BA/n)$, where $\sigma_{k}(\beta)$ is defined as in Corollary~\ref{Sig}.

\end{remark}

\begin{definition}
We define $\widehat{R}_{n}(\BA; g)$ as the solution to the following recursive relation, which corresponds to \eqref{Partial} without the lower-order terms:
\begin{align}\label{PartialII}
 \widehat{R}_{n}(\BA; g)\, :=& \,\, \frac{1}{b}\sum_{i=1}^{b} \sum_{1\leq j \leq b}\widehat{R}_{n}\big(\BA; g\TS (i,j)\big)  \, +\, \frac{1}{b\sqrt{n}}\sum_{i=1}^{b} \sum_{1\leq j_{1}<j_{2}\leq b }\widehat{R}_{n}\big(\BA; g\TS (i,j_{1})\big) \widehat{R}_{n}\big(\BA; g\TS (i,j_{2})\big)\nonumber  \\  &\, +\,  \frac{\BA}{b\sqrt{n}}\sum_{i=1}^{b} \sum_{1\leq j\leq b-1}\omega_{g\diamond (i,j)} \,.
\end{align}
The initial condition is $\widehat{R}_{n}(\BA; g)=0$ for $g\in G_{0,n}$. In addition we define
$$
\widehat{R}_{k,n}(\BA)\, :=\,\frac{1}{b^{n-k}}\sum_{ g\in G_{k,n}}\widehat{R}_{n}(\BA; g)\,.
$$

\end{definition}

The lemma below follows from the recursive relation~(\ref{PartialII}) that defines $\widehat{R}_{n}(\BA; g)$. We define
\begin{align*}
\widehat{\varrho}_k^{(n)}(\BA) := \mathbb{E} \big[ \widehat{R}_n(\BA; g)^2 \big].
\end{align*}
for $g \in G_{k,n}$. Note that the expectation on the right is independent of $g$ by the iid assumption.

\begin{lemma}\label{This}
The moments $\widehat{\varrho}_k^{(n)}$ satisfy the recursive relation
\begin{align*}
       \widehat{\varrho}_{k+1}^{(n)}(\BA)\, = \, \widehat{\varrho}_{k}^{(n)}(\BA) \, + \,\frac{1}{n}\frac{b-1}{2b } \big[\widehat{\varrho}_{k}^{(n)}(\BA)\big]^{2}\,+\,\frac{1}{n}\frac{\BA^2(b-1) }{b  } \,.
\end{align*}
\end{lemma}

The following lemma is also an immediate consequence of the recursive relation~(\ref{PartialII}).

\begin{lemma}\label{LemStochRecur}
The random variables $ \widehat{R}_{k,n}(\BA ) $ satisfy the recursive equation in $k\in \mathbb{N}$ given by:
\begin{align*}
\widehat{R}_{k+1,n}(\BA ) \,= & \,\,\widehat{R}_{k,n}(\BA ) \, +\, \frac{1}{b\sqrt{n}}\sum_{g\in G_{k,n }}\sum_{i=1}^{b} \sum_{1\leq j_{1}<j_{2}\leq b }\widehat{R}_{n}\big(\BA; g\TS (i,j_{1})\big) \widehat{R}_{n}\big(\BA; g\TS (i,j_{2})\big) \\  &\, +\,  \sum_{g\in G_{k,n }}\frac{\BA}{b\sqrt{n}}\sum_{i=1}^{b} \sum_{1\leq j\leq b-1}\omega_{g\diamond (i,j)} \, .
\end{align*}

\end{lemma}

\subsection{The case of $b=s$ and $\BA< \kappa_{b}$}

We will now reset the notation for the map $ M_{n}$ that was assigned in Section~\ref{Secb<s} to take in consideration the new scaling $\beta_{n}=\BA/n$ that is relevant  when $b=s$.

\begin{definition}\label{DefMaps} Define the maps $M_{n},\widehat{M}_{n}:\R^{+}\rightarrow \R^{+}$ for $n\in \mathbb{N}$ as
\begin{itemize}

\item $\ds M_{n}(x)\, :=\, \frac{n}{b}\Big[\big(1+ \frac{x}{n} \big)^b e^{(b-1)[\lambda (2\frac{\widehat{\beta}}{n})-2\lambda(\frac{\widehat{\beta}}{n})  ] }   \,-\,1   \Big]$

\item $\displaystyle \widehat{M}_{n}(x)\, :=\, x+\frac{1}{n}\frac{b-1}{2}x^{2}+\frac{1}{n}\frac{\BA^{2} (b-1)}{b} $

\end{itemize}

\end{definition}

\begin{remark}\label{RemarkMaps}
The maps $M_{n}$ and $\widehat{M}_{n}$ are defined so that
$$\varrho_{k+1}^{(n)}(\BA)\, =\,  M_{n}\big( \varrho_{k}^{(n)}(\BA) \big) \quad \text{and}\quad  \widehat{\varrho}_{k+1}^{(n)}(\BA)\, =\,  \widehat{M}_{n}\big( \widehat{\varrho}_{k}^{(n)}(\BA) \big)\, . $$
Thus the sequence $k \mapsto \varrho_k^{(n)}$ is determined by repeated compositions of the maps $M_n$. 
\end{remark}

The proof of Lemma~\ref{LemTan} is calculus-based, and we have put it in Section~\ref{SecProofsTwo}

\begin{lemma}\label{LemTan}
Let $\BA\in (0,\kappa_{b})$.  Define $\gamma(x,y):=\tan^{-1}\Big(\frac{\sqrt{b}}{\BA\sqrt{2}}x\Big)+ \frac{\BA(b-1)}{\sqrt{2b}}y $, and for $\epsilon > 0$ let $$A_{\epsilon} = \left \{ (x, m, n) \in \R^+ \times \N \times \N : \gamma \left( x, \frac{m}{n} \right) < \frac{\pi}{2} - \epsilon \right \}.$$
Then for any $\epsilon > 0$ there is a constant $C > 0$ such that

\begin{enumerate}[i)]
\item  $ \displaystyle
\sup_{ (x, m, n) \in A_{\epsilon}  }
 \frac{d}{dx} \widehat{M}_{n}^{m}(x)   \, \leq \,  C\, ,
$

\item $ \displaystyle
\sup_{ (x, m, n) \in A_{\epsilon}  }\Bigg| \widehat{M}_{n}^{m}(x)\,- \, \frac{\BA\sqrt{2}}{\sqrt{b}}\tan\bigg(\tan^{-1}\bigg( \frac{\sqrt{b}}{\BA\sqrt{2}}x \bigg)+  \BA\frac{b-1}{\sqrt{2b}} \frac{m}{n}   \bigg)\Bigg| \, \leq \,  \frac{C}{ n}\, ,
$

\item $ \displaystyle
\sup_{ (x,m,n) \in A_{\epsilon}   }\Big| M_{n}^{m}(x)\,-\, \widehat{M}_{n}^{m}(x)\Big| \, \leq \,  \frac{C}{ n}\, .
$

\end{enumerate}

\end{lemma}

\begin{remark} The function $g(r):= \frac{\BA\sqrt{2}}{\sqrt{b}}\tan\big(\tan^{-1}\big( \frac{\sqrt{b}}{\BA\sqrt{2}}x \big)+  \BA\frac{b-1}{\sqrt{2b}} r   \big)$ is the solution to the differential equation
$$  \frac{d}{dr}g(r) \, = \, \frac{b-1}{2} g^{2}(r) \, +\,\frac{b-1}{b} \widehat{\beta}^{2}      $$
with initial condition $g(0)=  x$.
\end{remark}

\begin{lemma}\text{ }\label{LemVar}
\begin{enumerate}[i)]

\item For any $\BA\in [0,\kappa_{b})$ and $r\in [0, 1  ]$, there is convergence as $n\rightarrow \infty$ given by
$$  \widehat{\varrho}_ {\lfloor nr \rfloor}^{(n)}(\BA) \quad \longrightarrow \quad \frac{\BA\sqrt{2}}{\sqrt{b}}\tan\Big(\BA\frac{b-1}{\sqrt{2b}} r  \Big)\, . $$

\item   For $m \in \mathbb{N}$, define  $\ds \widehat{\varrho}_{k,m}^{(n)}(\BA) := \mathbb{E}\big[\big(\widehat{R}_{n}(\BA; g)\big)^m \big]$. Then for any $\BA\in [0,\kappa_{b})$ and $m\in \mathbb{N}$, there is a $C>0$ such that for all  $r\in [0, 1  ]$ and $n\in \mathbb{N}$
$$\widehat{\varrho}_ {\lfloor nr \rfloor, 2m}^{(n)}(\BA)\, \leq  \, C \Big(\widehat{\varrho}_ {\lfloor nr \rfloor, 2}^{(n)}(\BA)\Big)^m     $$
and, in particular, $\widehat{\varrho}_ {\lfloor nr \rfloor, 2m}^{(n)}(\BA)$ is uniformly bounded over the stated range of variables.

\item For $0\leq j \leq k\leq n$, the random variables $\widehat{R}_{j,n}(\BA )$ and  $\widehat{R}_{k,n}(\BA )- \widehat{R}_{j,n}(\BA )$ are uncorrelated.

\end{enumerate}

\end{lemma}
\begin{proof}
\noindent Part (\I):  This is a consequence of part (\I\I) of Lemma~\ref{LemTan} and the recursive relation
$$ \widehat{\varrho}_ {k+1}^{(n)}(\BA) \, = \,   M_{n}\Big(\widehat{\varrho}_ {k}^{(n)}(\BA)\Big)\, . \vspace{.2cm}  $$

\noindent Part (\I\I):   As a consequence of part (\I), $\widehat{\varrho}_ {k, 2}^{(n)}(\BA)$ is uniformly bounded for all $n>0$ and $ 0\leq k\leq n r$.   For $m=3$ it is straightforward to show that $\widehat{\varrho}_{k, 3}^{(n)}$ satisfies the recursive relation
\begin{align*}
       \widehat{\varrho}_{k+1,3}^{(n)}(\BA)\, = &\, \frac{1}{b}  \widehat{\varrho}_{k,3}^{(n)}(\BA) \, + \,\frac{1}{\sqrt{n}}\frac{3(b-1)}{b} \big[\widehat{\varrho}_{k,2}^{(n)}(\BA)\big]^{2}\,+\,\frac{1}{n} \frac{b-1}{b} \widehat{\varrho}_{k,2}^{(n)}(\BA) \widehat{\varrho}_{k,3}^{(n)}(\BA)\,,+\,   \frac{1}{n^{\frac{3}{2}}} \frac{(b-1)(b-2)}{b^2} \big[\widehat{\varrho}_{k,2}^{(n)}(\BA) \big]^{3} \\  & \,+\, \frac{1}{n^{\frac{3}{2}}} \frac{b-1}{2b^2} \big[\widehat{\varrho}_{k,3}^{(n)}(\BA) \big]^{2}+\frac{1}{n^{\frac{3}{2}} } \frac{\BA^3 (b-1)}{b^2}\mathbb{E}[\omega^3] \,.
\end{align*}
From this we see that the absolute value of $ \widehat{\varrho}_ {k, 3}^{(n)}(\BA)$ satisfies a recursive inequality of the form
\begin{align*}
  \big| \widehat{\varrho}_ {k+1, 3}^{(n)}(\BA) \big| \, \leq & \,   \frac{1}{b}\big| \widehat{\varrho}_ {k, 3}^{(n)}(\BA) \big|  +\,\frac{3}{\sqrt{n}}\frac{b-1}{b} \big[\widehat{\varrho}_{k,2}^{(n)}(\BA)\big]^{2}\,+\,    \frac{1}{n}\frac{b-1}{b } \widehat{\varrho}_{k,2}^{(n)}(\BA)\big| \widehat{\varrho}_ {k, 3}^{(n)}(\BA) \big|   \,+\,\frac{1}{n^{\frac{3}{2}} }\frac{\BA^3(b-1) }{b^2  }\mathbb{E}[\omega^3]\\
&+\,   \frac{1}{n^{\frac{3}{2}}}\frac{(b-1)(b-2)}{b^2 }\big[\widehat{\varrho}_{k,2}^{(n)}(\BA) \big]^{3}\,+\, \frac{1}{n^{\frac{3}{2}}}\frac{b-1}{2b^2 }\big[\widehat{\varrho}_{k,3}^{(n)}(\BA) \big]^{2}
\intertext{with $|\widehat{\varrho}_ {0, 3}^{(n)}(\BA)|=0      $.    Let $\widehat{k}\in \mathbb{N}$ be the first value such that $|\widehat{\varrho}_{\widehat{k},3}^{(n)}(\BA)|>1$.     There exists an $\epsilon\in (0,1)$ and  a $c>0$ such that for all $k\leq \min(\widehat{k}-1,nr) $ the above is   }
\leq & \, \epsilon   \big| \widehat{\varrho}_ {k, 3}^{(n)}(\BA) \big| \, +\, \frac{c}{\sqrt{n}}\, .
\end{align*}
It follows that   $\displaystyle \sup_{0\leq k\leq \lfloor rn   \rfloor   }\big|\widehat{\varrho}_ {k+1, 3}^{(n)}(\BA)\big|$ is $\mathit{O}(\frac{1}{\sqrt{n}})$ and, in particular, bounded for all $n\gg 1$.

Similar reasoning holds for $m=4$. From the recursive equality for $\widehat{\varrho}_{k, 4}^{(n)}$, which is straightforward to derive but lengthy to write out, we have an inequality for  $\widehat{\varrho}_ {k, 4}^{(n)}(\BA) $ of the form
\begin{align*}
  \widehat{\varrho}_ {k+1, 4}^{(n)}(\BA) \, \leq  \, \epsilon  \widehat{\varrho}_ {k, 4}^{(n)}(\BA)\,+\,\frac{6(b^2-1)}{b^2} \big[\widehat{\varrho}_{k,2}^{(n)}(\BA)\big]^{2}\,+\,\frac{c}{n}
\end{align*}
for some   $\epsilon\in (0,1)$ and $c>0$.   It follows that $\displaystyle \big|\widehat{\varrho}_ {\lfloor rn   \rfloor, 4}^{(n)}(\BA)\big|$ is bounded by a multiple of $\big[\widehat{\varrho}_{ \lfloor rn   \rfloor,2}^{(n)}(\BA)\big]^{2}$.  The reasoning can be extended inductively to arbitrary moments.

\vspace{.2cm}

\noindent Part (\I\I\I): This follows from the recursive formula in Lemma \ref{LemStochRecur}.

\end{proof}

\begin{theorem}\label{CleanProcess}
For $\beta\in (0,\kappa_{b})$ and $n\in \mathbb{N}$, define the stochastic process
\begin{align*}
\hspace{2cm}Y_{r}^{(n)}(\BA)   \, := \,     \widehat{R}_{\lfloor r n \rfloor ,n}( \BA )\,, \hspace{1cm}r\in [0,1]\, .
\end{align*}
As  $n\rightarrow\infty$,  $\big(Y_{r}^{(n)}(\BA)\big)_{r\in [0,1]} $ converges in law to a continuous Gaussian  process $\big(\mathbf{y}_{r}(\BA)\big)_{r\in [0,1]} $ with independent, mean zero increments and
$$  \mathbf{y}_{r}(\BA)   \, \stackrel{d}{=} \,  \mathcal{N}\bigg( 0,\, \frac{\BA\sqrt{2}}{\sqrt{b}}\tan\Big(\BA\frac{b-1}{\sqrt{2b}} r  \Big) \bigg)\,.       $$
\end{theorem}

\begin{remark}
The limit process $\big(\mathbf{y}_{r}(\BA)\big)_{r\in [0,1]} $  is simply a deterministic time-change of a standard Brownian motion $(\mathbf{B}_{t})_{t\geq 0}$
$$ \mathbf{y}_{r}(\BA)   \, \stackrel{d}{=} \, \mathbf{B}_{\tau_{r}},  $$
where $\ds \tau_{r}:= \frac{\BA\sqrt{2}}{\sqrt{b}}\tan\big(\BA\frac{b-1}{\sqrt{2b}} r  \big)   $.

\end{remark}

\begin{proof}[Proof of Theorem \ref{CleanProcess}] First we will focus on finite-dimensional distributional convergence.  For all $n\in \mathbb{N}$ and $0\leq  r\leq s \leq 1$,
\begin{align}    \mathbb{E}\big[ Y_{r}^{(n)}(\BA) \big]\,=\,0  \hspace{1cm}\text{and}\hspace{1cm}  \mathbb{E}\Big[ Y_{r}^{(n)}(\BA)\Big(Y_{s}^{(n)}(\BA)\,-\, Y_{r}^{(n)}(\BA)    \Big)\Big] \, =\,0  \,,
\end{align}
 where the second equality follows from part (\I\I\I) of Lemma~\ref{LemVar}.   Thus finite-dimensional convergence is implied by the convergence of the variances $ \mathbb{E}\big[|Y_{r}^{(n)}(\BA)|^{2}\big]$.

Next we show that $Y_{r}^{(n)}(\BA)$ converges to a normal distribution  for $r\in [0,1)$. By definition of $ \widehat{R}_{k,n}(\BA ) $, we can write $ Y_{r}^{(n)}(\BA) $ as
\begin{align}\label{Gabel}
\text{}\hspace{2cm} Y_{r}^{(n)}(\BA)\, =& \,\widehat{R}_{\lfloor r n \rfloor ,n}(\BA )  \, = \,\underbrace{ \frac{ 1}{b^{n-\lfloor r n \rfloor }} \sum_{ g\in G_{\lfloor r n \rfloor ,\,n}   }  \widehat{R}_{n}\big(\BA ; g\big)   } \,. \\   &\hspace{.7cm} \textbf{Lindeberg-Feller central limit-type sum} \nonumber
\end{align}
However, the random variables $\widehat{R}_{n}\big(\BA ; g\big)  $ are i.i.d.\ with mean zero and variance  $ \widehat{\varrho}_ {\lfloor nr \rfloor}^{(n)}(\BA)$.  Recall that $ \widehat{\varrho}_ {\lfloor nr \rfloor}^{(n)}(\BA)$ converges to  $\frac{\BA\sqrt{2}}{\sqrt{b}}\tan\big(\BA\frac{b-1}{\sqrt{2b}} r  \big)$ as $n\rightarrow \infty$ by Lemma~\ref{LemVar}.      Since~ $|G_{\lfloor r n \rfloor ,\,n}|=b^{2(n-\lfloor r n \rfloor) } $ and by the above observation,~(\ref{Gabel}) converges to a normal distribution with mean zero and variance $\frac{\BA\sqrt{2}}{\sqrt{b}}\tan\big(\BA\frac{b-1}{\sqrt{2b}} r  \big)$ provided that the following Lindeberg condition is satisfied: as $n\rightarrow \infty$ for any fixed $\epsilon>0$
\begin{align*}
\frac{ 1}{b^{2n-2\lfloor r n \rfloor }}\Bigg[ \sum_{ g\in G_{\lfloor r n \rfloor ,n}   } \mathbb{E}\bigg[  \big|  \widehat{R}_{n}\big(\BA ; g\big)\big|^2 \chi\Big( \big|\widehat{R}_{n}\big(\BA ; g\big)\big| > \epsilon b^{n-\lfloor r n \rfloor } \Big)   \bigg] \Bigg]\quad \longrightarrow \quad 0\, .
\end{align*}
However, by Chebyshev, the expression above is bounded by
\begin{align}\label{Pimpernel}
\frac{ 1}{\epsilon^2 b^{4n-4\lfloor r n \rfloor }} \sum_{ g\in G_{\lfloor r n \rfloor ,n}   } \mathbb{E}\Big[  \big|  \widehat{R}_{n}\big(\BA ; g\big)\big|^4  \Big] \,  = \,   \frac{ 1}{\epsilon^2 b^{2n-2\lfloor r n \rfloor }}\widehat{\varrho}_ {\lfloor nr \rfloor, 4}^{(n)}(\BA)\, \leq \,\frac{ C}{\epsilon^2 b^{2n-2\lfloor r n \rfloor }} \, ,
\end{align}
where the second inequality holds for some $C>0$ and  all $n>0$ by part (\I\I) of Lemma~\ref{LemVar}.  The right side of~(\ref{Pimpernel}) converges to zero, so the Lindeberg condition holds.

When $r=1$, the sum~(\ref{Gabel}) contains only one term and thus does not immediately fall within the purview of Lindeberg-Feller. Let $(u_{n})_{n\geq 0}$ be a non-decreasing sequence of integers with $0\leq  u_{n}\leq n$, $u_{n}\rightarrow \infty$, and $\frac{u_{n}}{n}\rightarrow 0$.  Define  $\widehat{Y}_{1}^{(n)}(\BA):=\widehat{R}_{ n-u_{n} ,n}(\BA ) $.   By writing
\begin{align}\label{ReSplit}
 Y_{1}^{(n)}(\BA) \,   = \, \widehat{Y}_{1}^{(n)}(\BA)\,+\,\Big( Y_{1}^{(n)}(\BA) \,  -\, \widehat{Y}_{1}^{(n)}(\BA)   \Big),
\end{align}
it is sufficient for us to show that
\begin{enumerate}[I)]
\item  $\widehat{Y}_{1}^{(n)}(\BA)$ converges to the normal distribution $\ds\mathcal{N}\Big(0,\frac{\BA\sqrt{2}}{\sqrt{b}}\tan\big(\BA\frac{b-1}{\sqrt{2b}}   \big)\Big)$.

\item  $ Y_{1}^{(n)}(\BA)   -\widehat{Y}_{1}^{(n)}(\BA)  $ converges in probability to zero.
\end{enumerate}
The proof of  I) follows analogously to the case of $r<1$ above with $\lfloor r n\rfloor $ replaced by $n-u_{n}$.  For II), we observe that
the expectation of the square of the right side of (\ref{ReSplit}) is
\begin{align*}
\mathbb{E}\bigg[\Big|   Y_{1}^{(n)}(\BA) \,  -\, \widehat{Y}_{1}^{(n)}(\BA)   \Big|^{2}\bigg] \,= &\,n  \mathbb{E}\bigg[\Big|   \widehat{R}_{ n ,n}(\BA )  \,  -\, \widehat{R}_{ n-u_{n} ,n}(\BA )     \Big|^{2}\bigg]\, \nonumber  \\  =& \,\widehat{\varrho}_{n}^{(n)}(\BA)\,  - \, \widehat{\varrho}_{n-u_{n}}^{(n)}(\BA) \, \nonumber  \\  =& \,  \mathit{O}\Big(\frac{u_{n} }{n}\Big)\, .
\end{align*}
The second equality holds because $\widehat{R}_{ n ,n}(\BA )  \,  -\, \widehat{R}_{ n-u_{n} ,n}(\BA )   $ and $\widehat{R}_{ n-u_{n} ,n}(\BA )  $ are uncorrelated, and   the order equality follows from part (\I\I) of Lemma \ref{LemVar} since $\widehat{\varrho}_{n}^{(m)}(\BA)=\widehat{M}_{n}^{m}(0)$.

\end{proof}

\begin{lemma}\label{LemZ2Z}
There is a $C>0$ such that for all $n\geq 1$
\begin{align*}
\mathbb{E}\Big[ \big|R_{n}(\BA )\,-\,\widehat{R}_{n}(\BA )\big|^{2}\Big] \, \leq \, \, \frac{C}{n }\, .
\end{align*}

\end{lemma}

\begin{proof}The random variables  $R_{n}(\BA)-\widehat{R}_{n}(\BA )$ and
$\widehat{R}_{n}(\BA )$ are uncorrelated.  This can be seen, for instance, by  defining a series of intermediary random variables $\BR_{n}^{(\ell)}(\BA ) $ between $R_{n}(\BA)$ and $\widehat{R}_{n}(\BA )$ indexed by $\ell\in [0,n]$ such that $\BR_{n}^{(\ell)}(\BA ) :=\BR_{n}^{(\ell)}(\BA; g )$ for $g=D_{n}$, where the family of random variables $\{\BR_{n}^{(\ell)}(\BA; g )\}_{g\in G_{k,n}}$   satisfies the recursive relation~(\ref{Full}) of $R_{n}(\BA; g)$ for $g\in G_{k,n}$ with  $k<\ell$ and the recursive relation~(\ref{Partial})  of $\widehat{R}_{n}(\BA; g)$ for $g\in G_{k,n}$ with $k\geq \ell $.  By this construction, $\widehat{R}_{n}(\BA)=\BR_{n}^{(0)}(\BA ) $ and  $R_{n}(\BA)=\BR_{n}^{(n)}(\BA ) $ and the increments $\BR_{n}^{(\ell+1)}(\BA ) -\BR_{n}^{(\ell)}(\BA )$ are pairwise uncorrelated.

Since   $R_{n}(\BA)-\widehat{R}_{n}(\BA )$  and
$\widehat{R}_{n}(\BA )$ are uncorrelated, we have the first equality below:
\begin{align}
\mathbb{E}\bigg[ \Big(R_{n}(\BA )\,-\,\widehat{R}_{n}(\BA)\Big)^{2}\bigg]   \, = \,& \mathbb{E}\big[ R_{n}^2(\BA )\big]  \,-\,\mathbb{E}\big[ \widehat{R}_{n}^2(\BA )\big]  \nonumber \\
\, := \,& \varrho_{n}(\BA)\,-\,\widehat{\varrho}_{n}(\BA) \nonumber \\
\, = \,&M_{n}^{n}(0)\,-\,\widehat{M}_{n}^{n}(0)\, =\,\mathit{O}\big(1/n\big)\,.\label{Blah}
\end{align}
The order equality follows from part (\I\I\I) of  Lemma~\ref{LemTan}.

\end{proof}

\begin{proof}[Proof of  Theorem~\ref{ThmMain} in the $b=s$ and  $\beta<\kappa_{b} $ case]
We can write
\begin{align}\label{Split}
R_{n}(\BA) \, = \, \widehat{R}_{n}(\BA)\, + \, \big( R_{n}(\BA) \,-\,\widehat{R}_{n}(\BA)\big),
\end{align}
where the rightmost term converges to zero in probability  as a consequence of Lemma~\ref{LemZ2Z}.  Moreover, $\widehat{R}_{n}(\BA)$  converges in distribution to $\mathcal{N}\big(0, \upsilon_{b}(\beta)  \big)$ for $ \upsilon_{b}(\beta) :=\frac{\BA\sqrt{2}}{\sqrt{b}}\tan\big(\BA\frac{b-1}{\sqrt{2b}}   \big)$  by Theorem~\ref{CleanProcess}, which completes the proof.

\end{proof}

\subsection{The case of $b=s$ and $\BA= \kappa_{b}$}

We will need slightly more refined estimates to treat the critical point case,  $\BA= \kappa_{b}$,  than was required  for the analysis of $\BA<\kappa_{b}$.  Recall that we defined the family of random variables $\widehat{R}_{n}(\BA; g)$ through a recursive formula~(\ref{PartialII})  that was essentially a quadratic approximation of the recursive formula~(\ref{Full}) satisfied by the family of random variables $R_{n}(\BA; g):=\sqrt{n}(W_{n}(\BA; g)-1)$. At the critical point, a third-order term in the recursive formula for $R_{n}(\BA; g)$  becomes nonnegligible, making it convenient to formulate the random variables  $\widetilde{R}_{n}(\BA; g)$ in the definition below that serve the same purpose as $\widehat{R}_{n}(\BA; g)$ last section.

\begin{definition}
For $g\in G_{k,n}  $   we define $\widetilde{R}_{n}(\BA; g)$ to be the solution to the  recursive relation
\begin{align}\label{PartialIII}
\widetilde{R}_{n}&(\BA; g)\, := \,\, \frac{1}{b}\sum_{i=1}^{b}\sum_{1\leq  j\leq b} \widetilde{R}_{n}\big(\BA; g\TS (i,j)\big)  \, +\, \frac{1}{b\sqrt{n}}\sum_{i=1}^{b} \sum_{1\leq j_{1}<j_{2}\leq b }\prod_{k=1,2} \widetilde{R}_{n}\big(\BA; g\TS (i,j_{k})\big)\nonumber  \\  & \, +\, \frac{1}{bn}\sum_{i=1}^{b} \sum_{1\leq j_{1}<j_{2}<j_{3}\leq b }\prod_{k=1,2,3} \widetilde{R}_{n}\big(\BA; g\TS (i,j_{k})\big) \, +\,  \frac{\BA}{b\sqrt{n}}\sum_{i=1}^{b} \sum_{1\leq j \leq b-1}\omega_{g\diamond (i,j)}
\end{align}
with initial condition is $\widetilde{R}_{n}(\BA; g)=0$ for $g\in G_{0,n}$.

\end{definition}

 In addition, we define $\widetilde{W}_{n}(\BA; g)$,  $  \widetilde{R}_{k,n}(\BA)$, $\widetilde{M}_{n}(x) $, $\widetilde{\varrho}_{k,p}^{(n)}(\BA) $ in analogy to  $\widehat{W}_{n}(\BA; g)$,  $  \widehat{R}_{k,n}(\BA)$, $\widehat{M}_{n}(x) $, $\widehat{\varrho}_{k,p}^{(n)}(\BA) $ from the last section.

\begin{definition} Let $0\leq k\leq n$ and $g\in G_{k,n}  $.

\begin{itemize}
\item $\widetilde{W}_{n}(\BA; g)\,:=\,1\,+\,\frac{1}{\sqrt{n}}\widetilde{R}_{n}(\BA; g) $

\item $  \widetilde{R}_{k,n}(\BA)\, :=\,\frac{1}{b^{n-k}}\sum_{ g\in G_{k,n}}\widetilde{R}_{n}(\BA; g)$

\item  $\widetilde{M}_{n}(x)\,:=\,x\,+\,\frac{b-1}{2n}x^2\,+\,\frac{(b-1)(b-2)}{6n^2}x^3 \,+\,\frac{\BA^2(b-1)}{bn}$

\item $ \widetilde{\varrho}_{k,p}^{(n)}(\BA) := \mathbb{E}\big[\big(\widetilde{R}_{n}(\BA; g)\big)^p \big] $

\end{itemize}
When $p=2$ we identify $\widetilde{\varrho}_{k,2}^{(n)}(\BA)\equiv  \widetilde{\varrho}_{k}^{(n)}(\BA) $.

\end{definition}

\begin{remark}In analogy with~(\ref{RemarkMaps}), the maps $\widetilde{M}_{n}$ satisfy
\begin{align}\label{Zag}
 \widetilde{\varrho}_{k+1}^{(n)}(\BA)\, =\,  \widetilde{M}_{n}\big( \widetilde{\varrho}_{k}^{(n)}(\BA) \big)\, .
\end{align}
The random variables  $\widetilde{R}_{k,n}(\BA)$  have uncorrelated increments in the index $k\in \mathbb{N}$  and satisfy $\widetilde{R}_{n,n}(\BA)=\widetilde{R}_{n}(\BA; D_{n})$.

\end{remark}

Since the variance of $\widetilde{W}_{n}(  \kappa_{b}/ n; D_{n}\big) $ is $\widetilde{M}_{n}^{n}(0)$, the following lemma requires an analysis of $n$-fold compositions of the map $\widetilde{M}$.   The proof is in Section~\ref{SecProofsTwo}.

\begin{lemma} \label{LemRT} As $n\rightarrow \infty$ the variance of the random variable $ \widetilde{W}_{n}\big(\kappa_{b}/n; D_{n}\big)  $ has the convergence
$$ \log(n)\textup{Var}\bigg( \widetilde{W}_{n}\Big( \frac{ \kappa_{b}}{n};D_{n}\Big)  \bigg)\,\quad  \longrightarrow \, \quad
 \frac{6}{b+1} \, .  $$

\end{lemma}

The following lemma contains some analogous results to Lemma~\ref{LemTan}, and its proof is placed in Section~\ref{SecProofsTwo}.  The main point is (iii), which implies that the variance, $ M_{n}^{n}(0)$, of $W_{n}\big( \BA/n\big)$ is close to the variance,  $\widetilde{M}_{n}^{n}(0) $, of $\widetilde{W}_{n}\big(\BA/n\big)$.

\begin{lemma}\label{LemTanII}
Parts (i) and (ii) below hold for small enough $\epsilon>0$.

\begin{enumerate}[i)]

\item There are constants, $c,C>0$ such that for all $n>1$
$$  \,c \frac{n}{\log n} \,  \leq  \, \widetilde{M}^{\lceil n+\epsilon\log n \rceil}_{n}(0)    \,  \leq  \,C \frac{n}{\log n} \, .$$

\item    Define $\ell_{n}:=\lceil n+\epsilon \log n\rceil$.  There is $C>0$ such that for all $n>1$ all $m\leq \ell_{n}$
$$ \displaystyle
 \frac{d}{dx} \widetilde{M}_{n}^{\ell_{n}-m}(x )\Big|_{x=\widetilde{M}_{n}^{m}(0) }   \, \leq \,  C\Big(\frac{n}{\log n }\Big)^2\frac{1}{1+\big(\frac{\pi}{b-1} \widehat{M}_{n}^m(0)  \big)^2 } \, . $$

\item  As $n\rightarrow \infty$,
  $$ \frac{\log n }{n}\big|  M_{n}^{n}(0) - \widetilde{M}_{n}^{n}(0) \big|\, =\, \mathit{O}\Big(\frac{1}{\log n}   \Big) \, .  $$

\end{enumerate}

\end{lemma}

\begin{proof}[Proof of  Theorem~\ref{ThmMain} in the case $s=b$ and $\beta=\kappa_{b}$]  The reasoning follows along the same basic lines as for the case of $\beta<\kappa_{b}$.  First we show that $R_{n}(  \kappa_{b})= \sqrt{n}(W_{n}(  \kappa_{b}/n) -1)$ can be approximated by  $\widetilde{R}_{n}\big(  \kappa_{b}; D_{n} \big)=\sqrt{n}( \widetilde{W}_{n}\big(  \kappa_{b}/n; D_{n}\big)-1)$, on the scale of $\sqrt{\log n}/\sqrt{n}$.  Since $R_{n}(  \kappa_{b})-\widetilde{R}_{n}( \kappa_{b}; D_{n})$ and $\widetilde{R}_{n}( \kappa_{b}; D_{n})$ are uncorrelated
\begin{align*}
\mathbb{E}\Bigg[\bigg(  \Big(\frac{\log n}{n}  \Big)^{\frac{1}{2}}R_{n}(  \kappa_{b})\,-\,\Big(\frac{\log n}{n}  \Big)^{\frac{1}{2}}\widetilde{R}_{n}( \kappa_{b}; D_{n})  \bigg)^{2}\Bigg]\, &=\,\frac{\log n}{n} \bigg[   \textup{Var}\Big( R_{n}(\kappa_{b})\Big)\, - \,  \textup{Var}\Big( \widetilde{R}_{n}\big( \kappa_{b}; D_{n}\big)   \Big)\bigg] \, \\ &=\,\frac{\log n }{n}\Big(  M_{n}^{n}(0) - \widetilde{M}_{n}^{n}(0) \Big) \,\longrightarrow \, 0 \, ,
\end{align*}
where the convergence follows from part (\I\I\I) of Lemma~\ref{LemTanII}.   Thus the difference  $$\Big(\frac{\log n}{n}  \Big)^{\frac{1}{2}}R_{n}\Big( \frac{ \kappa_{b}}{n}\Big)\,-\,\Big(\frac{\log n}{n}  \Big)^{\frac{1}{2}}\widetilde{R}_{n}\Big( \frac{ \kappa_{b}}{n}; D_{n}\Big)$$ converges to zero in probability.

It remains to be shown that
$$Y_{n}\,:=\,\Big(\frac{\log n}{n}  \Big)^{\frac{1}{2}} \widetilde{R}_{n}( \kappa_{b}; D_{n})\, = \, \Big(\frac{\log n}{n}  \Big)^{\frac{1}{2}}\widetilde{R}_{n,n}(\kappa_{b})$$
 converges in law to a Gaussian. The proof is similar to the proof of Theorem \ref{CleanProcess} so we only provide a sketch.   For an increasing sequence $u_{n}\in \mathbb{N}$  with $u_{n}=\mathit{o}\big((\log n)/n   \big)$, define $\widetilde{Y}_{n}:=(\frac{\log n}{n})^{1/2}\widetilde{R}_{n-u_{n},n}(\kappa_{b})$.  It suffices to show that
$$I)\,\,\,\,  Y_{n}\,-\,  \widetilde{Y}_{n}\,\, \stackrel{\mathcal{P}}{\Longrightarrow}\,\,  0     \hspace{1.5cm}\text{and}\hspace{1.5cm}   II)\,\,\,\,\widetilde{Y}_{n}\,\, \stackrel{\mathcal{L}}{\Longrightarrow}\,\, \mathcal{N}\Big(0,\frac{6}{b+1}\Big) \, . $$
For (I) $Y_{n}\,-\,  \widetilde{Y}_{n}$ and $ \widetilde{Y}_{n}$  are uncorrelated,  and thus by similar reasoning as above we get the equality
\begin{align}\label{Morbal}
\mathbb{E}\Big[\big(Y_{n}\,-\,  \widetilde{Y}_{n}\big)^{2}\Big]\, =\,&\frac{\log n }{n}\Big( \widetilde{M}_{n}^{n}(0)\, -\,\widetilde{M}_{n}^{n-u_{n}}(0) \Big)\,.\nonumber \intertext{Using that the derivative of $\widetilde{M}_{n}$ is increasing, the above is smaller than   }  \,\leq \,&   \frac{\log n }{n}\widetilde{M}_{n}^{u_{n}}(0)\frac{d}{dx}\widetilde{M}_{n}^{n-u_{n}}\Big|_{x=\widetilde{M}_{n}^{u_{n}}(0)}\nonumber \\  \, \leq  \,&  C u_{n}\frac{n}{\log n }\,\longrightarrow \,0\,.
\end{align}
 The second inequality holds for some $C>0$ because  with large $n$ $$\widetilde{M}_{n}^{u_{n}}(0)\,=\, u_{n}\frac{\BA^2}{n} \frac{b-1}{b}\,+\, \mathit{O}\Big( \frac{u_{n}}{n^2}   \Big)\,, $$ and the derivative of $\widetilde{M}_{n}^{n-u_{n}}$ is bounded by a constant multiple of $n^2/\log^2n$ by part (ii) of Lemma~\ref{LemTanII}.  The last line of~(\ref{Morbal}) converges to zero by our assumption on  $u_{n}$.

For (II) recall that $ \widetilde{Y}_{n}$ is a normalized sum of i.i.d.\ random variables
$$   \widetilde{Y}_{n}\,:= \, \frac{1}{b^{u_{n}} }\sum_{g\in G_{n-u_{n},n}} \Big(\frac{\log n}{n}  \Big)^{\frac{1}{2}} \widetilde{R}_{n}(\kappa_{b};g)\,. $$
The random variables $ (\frac{\log n}{n})^{1/2}\widetilde{R}_{n}(\kappa_{b};g)$ have variance
$$\textup{Var}\bigg( \Big(\frac{\log n}{n}\Big)^{1/2}\widetilde{R}_{n}(\kappa_{b};g)\bigg)\,=\,  \textup{Var}\big(\widetilde{Y}_{n}\big)\,\approx\,  \textup{Var}\big(\widetilde{Y}_{n}\big)\,\longrightarrow  \,\frac{6}{b+1}\,, $$
where the  approximation holds by~(\ref{Morbal}) and the convergence is  by Lemma~\ref{LemRT}.  Thus the convergence  in law of $ \widetilde{Y}_{n}$ to $\mathcal{N}\big(0,\frac{6}{b+1}\big)$ follows from the Lindeberg-Feller central limit theorem if the random variables $ (\frac{\log n}{n})^{1/2}\widetilde{R}_{n}(\kappa_{b};g)$ satisfy a Lindeberg condition.  Uniform bounds on the fourth moments can be obtained by an analogous argument to that for part (ii) of Lemma~\ref{LemVar}.

\end{proof}

\subsection{Variance explosion}\label{SecStrongDisorder}

We now focus on the variance explosion stated in Theorem~\ref{ThmMain} for the case $s=b$ with  $\BA> \kappa_{b}$.    In fact, Lemma~\ref{LemVarExplode} provides a stronger result, which narrows down where the variance begins to blow-up in terms of the system size.  Before moving to the proof of Lemma~\ref{LemVarExplode}, we prove the following lemma, which is merely a translation of previous results.

\begin{lemma}\label{LemAnalog} The variance of
$ \sqrt{\log n}\big(W_{\lfloor n \kappa_{b}/\BA \rfloor }( \BA/n ) -1\big) $ converges to $ \frac{6}{b+1} $.
\end{lemma}
\begin{proof}
In the build-up  (see Lemmas~\ref{LemRT} and~\ref{LemTanII}) to the proof of~(\ref{WConvII}) in Theorem~\ref{ThmMain} we showed that the variance of $ \sqrt{\log n}\big(W_{n }( \kappa_{b}/n\big) -1\big)$  converges to $ \frac{6}{b+1} $.  We can simply apply this result by replacing the system size $n$  by $\lfloor n \kappa /\BA\rfloor$ and the inverse temperature $\kappa_{b}/n$ by $\kappa_{b}/\lfloor n\kappa /\BA\rfloor \approx \BA/n$.

\end{proof}

\begin{lemma}\label{LemVarExplode}
Suppose $s=b$ with  $\BA >  \kappa_{b}$.   Then as  $n$ goes to infinity the variance of $  W_{n}(\BA/n)   $ tends to infinity.  Moreover, there is are constants $0<c_{ \downarrow  }< c_{\uparrow} $ such that for $\ell_{\downarrow}(n):=\lfloor n \kappa_{b}/\BA +c_{\downarrow}\log n \rfloor $ and $\ell_{\uparrow}(n):=\lfloor n \kappa_{b}/\BA +c_{\uparrow}\log n \rfloor $
$$                    \textup{Var}\Big(  W_{\ell_{\downarrow}(n)}(\BA/n)   \Big) \,\, \longrightarrow \,\, 0 \hspace{1cm} \text{and}\hspace{1cm} \textup{Var}\Big(  W_{\ell_{\uparrow}(n)}(\BA/n)   \Big) \,\, \longrightarrow \,\, \infty \,.  $$

\end{lemma}

\begin{remark}
In Theorem 2.11 of~\cite{lacoin}, Lacoin and Moreno state that there exist $c,C>0$ such that for all $\beta>0$
\begin{align}\label{Tintin}
  \exp\Big\{  -\frac{c}{\beta^2}  \Big\}   \,\leq \,\lambda (\beta) -p(\beta) \,\leq \,   \exp\Big\{  -\frac{C}{\beta}  \Big\}\,,
\end{align}
where $p(\beta)$ is the quenched free energy
$$  p(\beta)\,:=\,\lim_{n\rightarrow \infty}\mathbb{E}\big[ \log\big(Z_{n}(\beta) \big)  \Big]\,. $$
By following their argument, which relies on~\cite[Proposition 4.3]{lacoin}, we can use the convergence  of   $ \textup{Var}\big(  W_{\ell_{\downarrow}(n)}(\BA/n)   \big) $ to zero to get can an upper bound of the form
\begin{align}\label{Koojo}
   \lambda (\beta) -p(\beta) \,\leq \, \exp\Big\{   -\frac{\kappa_{b}\log b }{\beta}   +C\log \beta    \Big\}     \,
\end{align}
for some $C>0$ and all $\beta>0$.  However, since the inequality~(\ref{Tintin}) involves qualitatively different  bounds from above and below, it is not clear to us that~(\ref{Koojo}) is truely a step in the right direction.

\end{remark}

\begin{proof}[Proof of Lemma~\ref{LemVarExplode}]

   The variance of  $W_{n}( \BA /n ) $ is equal to $\frac{1}{n}\varrho_ {n }^{(n)}(\BA) $, so it is sufficient to show that $\varrho_ {n }^{(n)}(\BA)$ eventually grows at a faster than linear rate in $n$.  Recall that $\varrho_ {m }^{(n)}(\BA) \, = M_{n}^{m}(0) $ and  $M_{n}(x) \geq \widehat{M}_{n}(x)  $ for all $x\geq 0$, where $M_{n}$ and $\widehat{M}_{n}$ are defined as in  Definition~\ref{DefMaps}.  As a consequence of Lemma~\ref{LemAnalog},   we can choose $n$ large enough so that $\widehat{M}_{n}^{k}(0)    $ with $k:=\lfloor n \kappa_{b}/\BA \rfloor  $ is larger than any $L>0$.  Pick some $L$ greater than $\frac{16\BA}{(\BA-\kappa_{b})(b-1)}$.   Notice that we have the  lower bound $\varrho_ {n }^{(n)}(\BA)\, >\,  \widehat{M}_{n}^{n-k}(L)$.
The first inequality below holds for any $K>0$ by the form of the maps $\widehat{M}_{n}$:
\begin{align}\label{MKay}
\widehat{M}_{n}^{r}(K)\,  > \, K\bigg(1+\frac{1}{n}\frac{b-1}{2}K   \bigg)^r \,
    >\,   K\exp\left\{  \frac{r}{n}\frac{b-1}{4} K\right\}\,.
\end{align}
The second inequality holds  as long as $\frac{K(b-1)}{2n}$ is smaller than the solution $x>1$ to the equation    $1+x =\exp\{ x/2   \}   $.
Observe that the function $g(r):=\widehat{M}_{n}^{r}(K)$ doubles before $r$ reaches the value $\lceil \frac{4n}{(b-1)K}  \rceil $.  Define the sequence $r_{j}\in \mathbb{N}$ such that
$$ r_{j}\, :=\, \sum_{i=1}^{j} \left\lceil \frac{4n}{2^{i-1}(b-1)L}  \right\rceil  \, .  $$
Through repeated use of~(\ref{MKay}), we have that
\begin{align}
 \widehat{M}_{n}^{r_{j}}(L)\,>\, L 2^{j}\, \quad \text{ and thus }\quad \widehat{M}_{n}^{n-k}(L)\, > \, L2^{|S_{L,n}  |}
\end{align}
for the set $S_{L,n}:=\{ j\in \mathbb{N}\,|\,   r_{j}<n-k     \}$.  However,  the number of terms in $S_{L,n}$ will grow linearly with large $n$:
$$  \frac{ | S_{L,n}| }{n} \, \approx \, \frac{n-k-  \frac{8n}{(b-1)L}  }{n} \,>\,\frac{1}{2}\Big(1-\frac{\kappa_{b}}{\BA}\Big)\,,$$
where the second inequality uses  that $L$ has a lower bound of  $\frac{16\BA}{(\BA-\kappa_{b})(b-1)}$.   It follows that  $\widehat{M}_{n}^{n}(0) $ grows exponentially in $n$ and thus the variance of $W_{n}\big( \BA/n \big) $  does also. \vspace{.2cm}

In fact the variance of  $W_{k}\big( \BA /n \big) $ already blows up as $k$ reaches the value $\ell_{\uparrow}(n):=\lfloor n \kappa_{b}/\BA +c_{\uparrow }\log n \rfloor $  for  $c_{\uparrow}>2\log_{2}(e)$ since with large $n$
$$\frac{|S_{L,\ell_{n}} |}{  c_{\uparrow}\log n   } \quad \longrightarrow \quad  1\,.$$
Thus,
$$  \textup{Var}\Big(  W_{\ell(n)}\big( \BA /n \big)  \Big)\, =\,\frac{1}{n}  \varrho_ {\ell(n) }^{(n)}(\BA)\, >\,  \frac{1}{n}\widehat{M}_{n}^{\ell(n)-k}(L)   \, > \, \frac{L}{n}2^{|S_{L,\ell(n)}  |} \,\gg \sqrt{n} \,\,   .  $$

Finally, the existence of $c_{\downarrow}$ such that   $\textup{Var}\Big(  W_{\ell_{\downarrow}(n)}(\BA/n)   \Big) $ converges to zero for  $\ell_{\downarrow}(n):=\lfloor n \kappa_{b}/\BA +c_{\downarrow}\log n \rfloor $ follows from Lemma~\ref{LemAnalog} and part (i) of Lemma~\ref{LemTanII}.  For the application of Lemma~\ref{LemTanII}, the system size ``$n$" in the statement of the lemma is replaced by our value  $\lfloor n \kappa_{b}/\BA \rfloor$.

\end{proof}

\section{The case of $s<b$}\label{SecNormalWeakDisorder}

In this section we prove Theorem \ref{ThmMainII} in the case $s < b$.      The analysis here is trivial, but it is mathematically  instructive to see what happens differently between the $b=s$ and $b>s$ cases.  Recall that we are generically scaling the  inverse temperature $\beta\equiv\beta_n$ to zero as the size, n, of the system grows.

  It will be convenient to reset some of the notations from previous sections. For $g\in G_{k,n}$ define
$$ R_{n}(g) \, := \,   \frac{1}{\beta_{n}}\Big(W_{n}\big(\beta_{n}; g\big)-1 \Big)    \, ,    $$
which satisfies the recursive relation
\begin{align}\label{ReRec}
R_{n}( g)\, =\,\,    \frac{1}{\beta_{n}}\Bigg[\frac{1}{b}\sum_{i=1}^{b} \prod_{1\leq j\leq s}\bigg(1+\beta_{n}R_{n}\big( g\TS (i,j)\big)\bigg)\prod_{1\leq j \leq s-1}  E_{g\diamond (i,j)}(\beta_{n}   ) \,-\, 1\Bigg]\,.
\end{align}

\begin{definition}
 For $g\in G_{k,n}$ we define $\widehat{R}_{n}(g)$ as the solution to the  recursive relation
\begin{align}\label{PartialIII}
 \widehat{R}_{n}( g)\, =& \,\, \frac{1}{b}\sum_{i=1}^{b} \sum_{1\leq j\leq s}\widehat{R}_{n}\big( g\TS (i,j)\big) \, +\,  \frac{1}{b }\sum_{i=1}^{b} \sum_{1\leq j \leq s-1}\omega_{g\diamond (i,j)}
\end{align}
 with $\widehat{R}_{n}( g)=0$ for $g\in G_{0,n}$.
\end{definition}
\begin{remark}\label{RemarkEasy}
The recursive relation~(\ref{PartialIII}) does not contain any quadratic terms in the random variables $\widehat{R}_{n}\big( g\TS (i,j)\big) $  as it did in the last section.  Notice that when $g=D_{n}$ the recursive relation~(\ref{PartialIII}) simply gives us
$$ \widehat{R}_{n}( D_{n})\, = \, \sum_{m=1}^{n} \frac{1}{b^{m}}\sum_{ a\in V_{m}    }\omega_{a}  \, .$$
\end{remark}

\begin{definition} Define the maps $M_{n},\widehat{M}_{n}:\R^{+}\rightarrow \R^{+}$ for $n\in \mathbb{N}$ as
\begin{itemize}

\item $\ds M_{n}(x)\, :=\, \frac{1}{b\beta_{n}^2}\Big[\big(1+ x\beta_{n}^2 \big)^{s} e^{(s-1)[\lambda (2\beta_{n})-2\lambda(\beta_{n})  ]}   \,-\,1   \Big]$

\item $\displaystyle \widehat{M}_{n}(x)\, :=\, \frac{s}{b}x+\frac{s-1}{b} $

\end{itemize}

\end{definition}
\begin{remark}
As usual the maps $M_{n}$ and $\widehat{M}_{n}$ are defined so that
$$\mathbb{E}\Big[\big( R_{n}( D_{n})   \big)^2\Big]\, =\,  M_{n}^{n}( 0) \hspace{.7cm} \text{and} \hspace{.7cm}  \mathbb{E}\Big[\big( \widehat{R}_{n}( D_{n})   \big)^2\Big]\, =\,  \widehat{M}_{n}^{n}(0 )\, . $$

\end{remark}

The following lemma has an easy proof, which we do not include.

\begin{lemma}\label{LemMS}  There is a $C>0$ such that the inequality below holds for all $n\geq 1$
$$
\big| M_{n}^{n}(0)\,-\, \widehat{M}_{n}^{n}(0)\big| \, \leq \,  C\beta_{n}\, .
$$

\end{lemma}

\begin{proof}[Proof of Theorem~\ref{ThmMainII} in the case of $s<b$] By the recursive relation~(\ref{RemarkEasy}), it is clear that
$$ \widehat{R}_{n}( D_{n})\quad \,\stackrel{\mathcal{P}}{\Longrightarrow} \quad  \,\sum_{m=1}^{\infty} \frac{1}{b^{m}}\sum_{ a\in V_{m}    }\omega_{a} \, . $$
Thus it is sufficient to show that the variance of the difference $R_{n}( D_{n})  -   \widehat{R}_{n}(D_{n})$
converges to zero with large $n$.  The difference $R_{n}( D_{n})  -   \widehat{R}_{n}(D_{n})$ is uncorrelated with $ \widehat{R}_{n}(D_{n})$, so
\begin{align}
\mathbb{E}\Big[ \big(R_{n}( D_{n})  \, - \,   \widehat{R}_{n}( D_{n})\big)^{2}\Big]   \, = \,& \mathbb{E}\Big[ \big( R_{n}( D_{n})   \big)^{2}\Big] \,-\, \mathbb{E}\Big[ \big( \widehat{R}_{n}( D_{n})   \big)^{2}\Big]  \nonumber \\
\, = \,&   M_{n}^{n}(0)\,-\,\widehat{M}_{n}^{n}(0)  \, =\,\mathit{O}\big(\frac{1}{n}\big)\,,
\end{align}
where the order equality holds by Lemma~\ref{LemMS}.

\end{proof}

\section{The edge model}\label{SecEdgeModel}

A closely related polymer model can be defined by placing random i.i.d.\ weights $\omega_{a}$ on the edges $a\in E_{n}$ rather than the vertices of the diamond lattice $D_{n}$.  The normalized partition function takes the form
\begin{align}\label{ExpWII}
W_{n}(\beta)\, := \,  \frac{1}{|\Gamma_{n}|}\sum_{p\in \Gamma_{n}  }\prod_{a\Try p} E(\beta; a   )
\end{align}
for $\ds E(\beta; a   ):= \frac{e^{\beta  \omega_{a}  }  }{ \mathbb{E}[e^{\beta  \omega_{a}  }]  } $, where $a\Try p$ means that the edge $a$ lies along the path $p$.  In analogy with~(\ref{Induct}) there is an obvious recursive relation between the distributions of $W_{n}(\beta)$ and $W_{n+1}(\beta)$:
\begin{align*}
W_{n+1}(\beta)  \, \stackrel{ d }{ =}  \,  \frac{1}{b}\sum_{i=1}^{b} \prod_{1\leq j\leq s}  W_{n}^{(i,j)}(\beta)  \, ,
\end{align*}
 where the $W_{n}^{(i,j)}(\beta)$  are independent copies of  $W_{n}(\beta)$.  In the case of $b=s$, there is an interesting dichotomy between the edge and vertex models that is revealed by considering the expected number of edges/vertices shared between two randomly chosen paths:  two random paths picked from $\Gamma_{n}$ have an expected number of shared edges equal to $1$ and shared vertices equal to $\frac{b-1}{b}n$.   This is a consequence of the hierarchical   positions of the vertices on the graph in which  the lower generation vertices are shared by exponentially more paths.  Indeed, the edges  play an equivalent role with respect to paths to only the $n^{th}$ generation vertices.  Since the elements $\prod_{a\Try p} E(\beta; a   )$  in the sum~(\ref{ExpWII}) share fewer of the variables $ E(\beta; a   )$ on average in the edge model, this suggests that the edge model will exhibit ``less disorder" and perhaps even have a finite range of small $\beta$ in which the model is weakly disordered.  On the other hand the  second moment of $W_{n}(\beta) $ tends to infinity with large $n$ since it satisfies the recursive equation
$$  \mathbb{E}\big[W_{n+1}^2(\beta) \big]  \, =\,  \frac{1}{b}  \mathbb{E}\big[  W_{n}^2(\beta)  \big]\,+\,\frac{b-1}{b}\,.  $$
with $\mathbb{E}\big[W_{0}^2(\beta) \big]=\mathbb{E}\big[E^2(\beta; a   )\big] >1$.   The following theorem follows in analogy to Theorem~\ref{ThmMain}.

\begin{theorem}[$b=s$, edge model]\label{ThmOldII}  With large $n\in \mathbb{N}$ the variance of $W_{n}\big( \widehat{\beta}/\sqrt{n}\big)$ converges to zero for $\BA\leq \kappa_{b}$ and tends to infinity for $\BA > \kappa_{b}$.   When  $\widehat{\beta}  <  \kappa_{b}$  we have the following weak convergence:
\begin{align*}
 \sqrt{n}\bigg(W_{n}\Big(\frac{\widehat{\beta}}{\sqrt{n}}\Big) -1\bigg)    \quad  \stackrel{\mathcal{L}}{\Longrightarrow} \quad  \mathcal{N}\big( 0,\, \upsilon_{b}(\widehat{\beta}) \big) \,,
\end{align*}
where the variance of the limit  is  $\ds \upsilon_{b}(\widehat{\beta}):=\big(1/\widehat{\beta}^{2} - 1/\kappa_{b}^{2}  \big)^{-1}$.    At the critical value $\widehat{\beta}=\kappa_{b}$, the limit result becomes:
\begin{align*}
  \sqrt{\log n}\left(W_{n}\Big(\frac{\widehat{\kappa_{b} }}{\sqrt{n}}\Big) -1\right)    \quad  \stackrel{\mathcal{L}}{\Longrightarrow} \quad  \mathcal{N}\Big( 0,\, \frac{6}{b+1}\Big) \,.
\end{align*}

\end{theorem}

\section{Miscellaneous proofs}\label{SecMiscProofs}

\subsection{Proofs from Section~\ref{Secb<s}}\label{SecProofsOne}

\begin{proof}[Proof of Lemma~\ref{LemMMap}]
Note that there is a $c>0 $ such that for all $x\in \R^{+}$
\begin{align}\label{Timoy}
  \frac{s}{b}x \,< \, \widehat{M}(x)\, \leq \, \frac{s}{b}xe^{cx}\, .
\end{align}
It is sufficient to show that the inequalities (i)-(iii) hold for $\lambda>0$ small enough.   We will assume that $\lambda $ satisfies
\begin{align}\label{Thingy}
\exp\bigg\{\frac{c\frac{s}{b}\lambda }{1-\frac{b}{s}}   \bigg\}\leq \frac{s}{b}\,.
\end{align}
\vspace{.2cm}

\noindent Part  (i):
Let $\widetilde{n}\equiv \widetilde{n}(x,N) \in \mathbb{N}$ be the first value of $m$ such that
$$     \widehat{M}^{m}\big( x (b/s)^{N} \big)  \, >\,   x\frac{s}{b}\,. $$
When $N$ is large, the  value of $\widetilde{n}$ must be greater than $N$ by the following induction argument:   For $m< \widetilde{n}\wedge N$, we can apply~(\ref{Timoy}) to get the first two inequalities below.
\begin{align}\label{Goring}
  \widehat{M}^{m+1}\big( x (b/s)^{N}   \big) \,\leq \,&  x  (b/s)^{N-m-1}   \exp\bigg\{c\sum_{r=0}^{m}\widehat{M}^{r}\big( x (b/s)^{N}   \big)   \bigg\} \nonumber \\
\,\leq \,&  x  (b/s)^{N-m-1}  \exp\bigg\{c\sum_{r=0}^{m}\big(\frac{b}{s}\big)^{m-r}\widehat{M}^{m}\big( x (b/s)^{N}   \big)   \bigg\}\nonumber \\
\,= \,&  x  (b/s)^{N-m-1}  \exp\bigg\{ \frac{c}{1-\frac{b}{s}}\widehat{M}^{m}\big( x (b/s)^{N}   \big)   \bigg\}\nonumber \\
   \,\leq \,&    x  (b/s)^{N-m-1}\exp\bigg\{\frac{c\frac{s}{b}x}{1-\frac{b}{s}}   \bigg\} \, \leq \,   x  \big(\frac{b}{s}\big)^{N-m}\,.
\end{align}
The third inequality uses the assumption that $m< \widetilde{n}$, and the last inequality uses that $x\leq \lambda$.  The above implies that $N+1\leq  \widetilde{n}\wedge N$.  We can apply analogous inequalities as above to get the desired bound for $ \widehat{M}^{N-n}\big( x (b/s)^{N}   \big)$.

 \vspace{.3cm}

\noindent Part (ii): By similar reasoning as in~(\ref{Goring}),
\begin{align*}
  \widehat{M}^{N-n}\big( x (b/s)^{N}   \big) \,-\, x  (b/s)^{n} \,\leq \,& x  (b/s)^{n} \bigg(\exp\bigg\{\frac{c}{1-\frac{b}{s}}\widehat{M}^{N-n}\big( x (b/s)^{N}   \big)    \bigg\}-1   \bigg)\,.
\intertext{Moreover, applying the inequality $e^{r}-1\leq re^{r}$ yields     }
\,\leq \,&  x(b/s)^{n} \frac{c}{1-\frac{b}{s}}   \widehat{M}^{N-n}\big( x (b/s)^{N}   \big)    \exp\bigg\{\frac{c}{1-\frac{b}{s}}\widehat{M}^{N-n}\big( x (b/s)^{N}   \big)    \bigg\}\,
\\ \,\leq \,& Cx^2(b/s)^{2n} \,,
\end{align*}
where the last inequality holds for some $C>0$ as a consequence of (i).\vspace{.3cm}

\noindent Parts (iii) and (iv): Let  $y:= x (b/s)^{n}$.  To bound the derivative, we can apply the chain rule to get
\begin{align}
 \frac{d}{dx}  \widehat{M}^{n}\big(x(b/s)^{n} \big)\, \leq \, & \frac{s}{b}\prod_{r=0}^{n-1}\Big(  1+\widehat{M}^{r}(y ) \Big)^{s-1} \nonumber \\ \,=\, & \frac{s}{b} \prod_{r=1}^{n}\Big(  1+\widehat{M}^{r-n}\big( \widehat{M}^{n-1}(y )  \big) \Big)^{s-1}\nonumber  \\
\, \leq \, & \frac{s}{b}\exp\left\{ (s-1)\sum_{r=1}^{n} \widehat{M}^{r-n}\big( \widehat{M}^{n-1}(y )  \big) \right\} . \nonumber  \\
\intertext{Moreover, since $\widehat{M}^{-1}(y)\leq \frac{b}{s}y$,     }
\, \leq \, & \frac{s}{b}\exp\left\{ \frac{ s-1}{ 1-\frac{b}{s}   }\widehat{M}^{n-1}(y ) \right\} \,.\nonumber  \\
\intertext{Note that the above gives us (iv).   We can apply that $\widehat{M}^{n-1}\big( x (b/s)^{n} \big)\leq x(s/b)$, which holds    by the reasoning in the proof of  (i), to bound the above by  }
\, \leq \, & \frac{s}{b}\exp\left\{    x \frac{s}{b}  \frac{ s-1  }{ 1-\frac{b}{s}   } \right\} \,\leq \,   \frac{s}{b}\exp\left\{    \lambda\frac{s}{b}  \frac{ s-1  }{ 1-\frac{b}{s}   } \right\} \, .
\end{align}
The above gives us a bound independent of   $n$.

\end{proof}

\begin{proof}[Proof of Lemma~\ref{Initial}] Part (i):  Recall that $\textup{Var}\big( W_{n}(\beta_{n}; g)  \big) =M_{n}^{m}(0)$.  Pick $c>0$ to be large.   For $x< c(b/s)^{ n/2}   $, there is a $c'>0$ such that
\begin{align}\label{Godoe}
0\leq  M_{n}(x)-  \frac{s}{b}x -\BA^2\frac{s-1}{b}\big(\frac{b}{s}\big)^{n}   \, \leq \, c'\big(\frac{b}{s}\big)^{\frac{3n}{2}}\,+\,c'x\big(\frac{b}{s}\big)^{\frac{n}{2}} \,.
\end{align}
By using a telescoping sum, we can write
\begin{align*}
M_{n}^{m }(0)   \,-\,\BA^2\frac{s-1}{b}\frac{1-(\frac{b}{s})^{m}}{1-\frac{b}{s}}\big(\frac{b}{s}\big)^{n-m+1}  \,  =&\,\sum_{k=0}^{m-1} \big(\frac{b}{s}\big)^{m-k-1}\Big(  M_{n}^{k+1 }(0) \,- \, \frac{s}{b}  M_{n}^{k}(0)   \,-\,\BA^2\frac{s-1}{b}\big(\frac{b}{s}\big)^{n}\Big)\,.
\intertext{For any $m$ such that  $ M_{n}^{m-1}(0)< c(b/s)^{n/2} $, ~(\ref{Godoe}) implies that the above is bounded by    }
 \,  \leq &\,c'\sum_{k=0}^{m-1} \big(\frac{b}{s}\big)^{m-k-1}\Big(   \big(\frac{b}{s}\big)^{\frac{3n}{2}}\,+\, \big(\frac{b}{s}\big)^{\frac{n}{2}} M_{n}^{k}(0)\Big)
\\
 \,  \leq &\, \frac{c'}{1-\frac{b}{s}}  \big(\frac{b}{s}\big)^{\frac{3n}{2}}\,+\, \frac{c}{1-\frac{b}{s}} \big(\frac{b}{s}\big)^{\frac{n}{2}} M_{n}^{m}(0)\,,
\end{align*}
where in the last inequality we have used that $M_{n}^{k}(0)$ is increasing with $k$.   It follows that $M_{n}^{m }(0) $ is
 very close to $\BA^2\frac{s-1}{s-b}\big( 1-(b/s)^{m}\big)(b/s)^{n-m} $ for all $m\leq n/2$ and in particular (i) holds.

\vspace{.3cm}

\noindent Part (ii):   Define $\sigma_{k,n}^{(m)} :=\mathbb{E}\big[ |W_{n}(\beta_{n}; g)-1|^m\big]$.  The recursive relation~(\ref{WRecur}) and part (i) implies that there is a $C>0$ such that for $k<m$  $|\sigma_{k,3}^{(n)}|$ is bounded by a constant multiple of  $(b/s)^{2n}$.  Thus, by applying~(\ref{WRecur}) again, there is a $C>0$ such that $\sigma_{k,4}^{(n)}$ satisfies
\begin{align*}
\sigma_{k+1,4}^{(n)}\,\leq \,\frac{s}{b^3}\sigma_{k,4}^{(n)}  \,+    \,\frac{ 6s(s-1)  }{b^3}\big(\sigma_{k,2}^{(n)}\big)^2\,+\,C\big( \frac{b}{s} \big)^{2n}\,.
\end{align*}
By applying part (i) we get the inequality.

\end{proof}

\subsection{Proofs from Section~\ref{SecWeakDisorder}}\label{SecProofsTwo}

We remind the reader that the notations $M_{n}$ and  $\widehat{M}_{n}$ refer to different maps in Section~\ref{SecWeakDisorder} than in  Section~\ref{Secb<s}.

\begin{proof}[Proof of Lemma.~\ref{LemTan}]  Part (\I):  Define $r_{m}\in [0,\kappa_{b})$ as
\begin{align}\label{RecurR}
r_{m} \, := \,     \frac{ \sqrt{2b}   }{\widehat{\beta}(b-1)   }   \tan^{-1} \bigg( \frac{\sqrt{b}}{\widehat{\beta}\sqrt{2}}     \widehat{M}_{n}^{m}(x) \bigg)    \,.
\end{align}
Notice that the recursive form for  $\widehat{M}_{n}^{m}(x)$ implies
\begin{align*}
\tan\bigg( \BA\frac{b-1}{\sqrt{2b}}   r_{m+1}   \bigg)\,= &\,  \tan \bigg( \BA\frac{b-1}{\sqrt{2b}}   r_{m}  \bigg)\, +\, \frac{1}{n}\BA\frac{b-1}{\sqrt{2b}}\sec^2\bigg( \BA\frac{b-1}{\sqrt{2b}}   r_{m}  \bigg)\, .
\end{align*}
The above resembles a linear approximation  around $r_{m}$ with $r_{m+1}\approx r_{m}+\frac{1}{n}$.   Since the derivative of tangent is increasing over the interval $[0,\pi)$, we will always have $r_{m+1} < r_{m}+\frac{1}{n}$.  It follows that
\begin{align}\label{Tiz}
  \widehat{M}_{n}^{m}(x) \,  \leq  \, \frac{\BA\sqrt{2}}{\sqrt{b}}\tan\bigg(\tan^{-1}\bigg( \frac{\sqrt{b}}{\BA\sqrt{2}}x \bigg)+  \BA\frac{b-1}{\sqrt{2b}} \frac{m}{n} \bigg)   \, .
\end{align}
Since $\widehat{M}_{n}'(x)=1+\frac{b-1}{n}x$, by the chain rule we have
\begin{align}
 \frac{d}{dx}\Big[ \widehat{M}_{n}^{m}(x)  \Big] \, = \, \prod_{j=1}^{m} \widehat{M}_{n}'\Big( \widehat{M}_{n}^{j-1}(x)\Big) &= \prod_{j=1}^{m}  \bigg(1+ \frac{b-1}{n}\widehat{M}_{n}^{j-1}(x)\bigg) \nonumber \\ &\leq \exp \bigg \{ \frac{b-1}{n}\sum_{j=1}^{m}\widehat{M}_{n}^{j-1}(x)    \bigg\} \nonumber \\
 &\leq \exp \left \{  \frac{\BA \sqrt{2}(b-1)}{ \sqrt{b} }\int_{0}^{\frac{m}{n}}\tan\bigg(\tan^{-1}\bigg( \frac{\sqrt{b}}{\BA\sqrt{2}}x \bigg) +  \BA\frac{b-1}{\sqrt{2b}} r  \bigg) dr \right \}.\label{Salt}
\end{align}
The second inequality above holds by~(\ref{Tiz}) and a Riemann upper bound.  The right side of \eqref{Salt} is bounded uniformly for all $x\geq 0$ and $m,n\in \mathbb{N}$ with $\gamma(x,\frac{m}{n})<\pi /2-\epsilon $.

\vspace{.4cm}

\noindent Part (\I\I):   Define the function
$$\phi(x,y) \,   = \,  \frac{\BA\sqrt{2}}{\sqrt{b}}\tan\bigg(\tan^{-1}\bigg( \frac{\sqrt{b}}{\BA\sqrt{2}}x \bigg)+  \BA\frac{b-1}{\sqrt{2b}} y \bigg)   \,   .  $$
The difference between $\widehat{M}_{n}^{m}(x)$  and   $  \phi\big(x,\frac{m}{n}\big)$ can be bounded through a telescoping sum:
\begin{align}
\Big|   \phi\Big(x,\frac{m}{n}\Big)    \, - \,   \widehat{M}_{n}^{m}(x)    \Big| \, \leq \,& \sum_{k=0}^{m-1}\bigg|  \widehat{M}_{n}^{m-k-1}\bigg(  \phi\Big(x,\frac{k+1}{n}\Big)   \bigg) \, -\,     \widehat{M}_{n}^{m-k}\bigg(  \phi\Big(x,\frac{k}{n}\Big)   \bigg) \bigg|  \nonumber \\  \leq \,& \Bigg[\sup_{0\leq \ell \leq  m} \frac{d}{da} \widehat{M}_{n}^{m-\ell-1}\bigg(  \phi\Big(x,\frac{\ell}{n}\Big)    \bigg)         \Bigg] \sum_{k=0}^{m-1}\bigg|  \phi\Big(x,\frac{k+1}{n}\Big)  \, -\,   \widehat{M}_{n}\bigg(  \phi\Big(x,\frac{k}{n}\Big)   \bigg)   \bigg|
\nonumber  \\  \leq \,&C\sum_{k=0}^{m-1}\Bigg|  \int_{\frac{k}{n}}^{\frac{k+1}{n}}dr\Big(r-\frac{k}{n}   \Big)  \frac{\partial^2}{\partial r^2}\phi(x,r)  \Bigg| \nonumber
 \\  \leq \,&\frac{C}{n}\int_{0}^{\frac{m}{n}}dr \frac{\partial^2}{\partial r^2}\phi(x,r) \,=\,\frac{C}{n} \frac{\partial}{\partial r}\phi(x,r)\Big|_{r=\frac{m}{n}} \, \leq \, \frac{C}{n} \frac{\partial}{\partial y}\phi(x,y)\,\leq \, \frac{\widehat{C}}{n}\, .\label{Morv}
\end{align}
In the third inequality above, we have used part (\I) to bound the supremum since
$$\text{}\hspace{2.5cm}\gamma\bigg( \phi\Big(x,\frac{\ell}{n}\Big) ,\,\frac{m-\ell -1}{n}     \bigg) \,=\,\gamma\Big(x,\frac{m-1}{n}\Big) \, <\, \frac{\pi}{2}-\epsilon\,,  \hspace{1.7cm}0\leq \ell \leq m \, . $$
  The third inequality also  uses the observation that
$$\widehat{M}_{n}\bigg(  \phi\Big(x,\frac{k}{n}\Big)   \bigg)\, = \, \phi\Big(x,\frac{k}{n}\Big) \, +\,  \frac{1}{n} \frac{d }{dr}\phi(x,r) \Big|_{r=\frac{k}{n}} $$
and  a second-order Taylor expansion to write
\begin{align*}
  \phi\Big(x,\frac{k+1}{n}\Big)    \, = \,\widehat{M}_{n}\bigg(  \phi\Big(x,\frac{k}{n}\Big)   \bigg) \, +\,   \int_{\frac{k}{n}}^{\frac{k+1}{n}}dr\Big(r-\frac{k}{n}   \Big)  \frac{\partial^2}{\partial r^2}\phi(x,r)  \, .
\end{align*}
The final inequality in~(\ref{Morv}) holds for some $\widehat{C}$ by another application of part (\I).

\vspace{.4cm}

\noindent Part (\I\I\I):   We begin with the observation that for any bounded interval $z\in [0,a]$ there is a $c>0 $ such that for all $n\geq 1$
\begin{align}\label{Twiz}
\Big|  M_{n}(z)\,-\, \widehat{M}_{n}(z)   \Big| \, \leq \,  \frac{c}{n^2}  \, .
\end{align}
We can write the difference between $M_{n}^{m}(x)$ and  $\widehat{M}_{n}^{m}(x)$  using the following telescoping sum:
\begin{align}\label{Shiznik}
  M_{n}^{m}(x)\,-\, \widehat{M}_{n}^{m}(x) \, = \, & \sum_{k=1}^{m}\Big[ \widehat{M}_{n}^{m-k}\big(M_{n}^{k}(x)\big)\,-\, \widehat{M}_{n}^{m-k+1}\big(M_{n}^{k-1}(x)\big) \Big]\nonumber \\   \leq  \, & \sum_{k=1}^{m}\bigg[\sup_{0\leq z \leq M_{n}^{k}(x)} \frac{d}{dz}\Big[ \widehat{M}_{n}^{m-k}(z)\Big] \bigg] \Big[ M_{n}\big(M_{n}^{k-1}(x)\big)\,-\, \widehat{M}_{n}\big(M_{n}^{k-1}(x)\big) \Big]\, .
\end{align}
Let $N\in \mathbb{N}$ be   the smallest number such that
\begin{align}\label{CNSRT}
   \max_{0\leq \ell \leq   N} \gamma\Big(  M_{n}^{\ell}(x),\, \frac{N-\ell  }{ n  }  \Big)  \, \geq  \,  \frac{\pi}{2}-\frac{\epsilon}{2} \, .
\end{align}
The above implies that  $M_{n}^{k}(x)$ is bounded by some $c>0$ for all $k<N$ and $n>0$.  Combining (\ref{Shiznik}) with the remark~(\ref{Twiz}) yields that  there is a $C>0$ such that for $\ell< N$
\begin{align*}
  M_{n}^{\ell}(x)\,-\, \widehat{M}_{n}^{\ell}(x) \, \leq \, & \frac{\mathbf{c}}{n^2}\sum_{k=1}^{\ell}\bigg[\sup_{0\leq z \leq M_{n}^{k}(x)} \frac{d}{dz}\Big[ \widehat{M}_{n}^{\ell-k}(z)\Big] \bigg] \, \leq \,\frac{\mathbf{c}'\ell}{n^2} \,,
\end{align*}
where the second inequality holds by part (\I\I).  Plugging the above back into~(\ref{CNSRT}) with $N$ replaced by $m$ yields
\begin{align*}
  \max_{0\leq \ell \leq    m} \gamma\Big(  M_{n}^{\ell}(x),\, \frac{m-\ell  }{ n  }  \Big) \, \leq \, & \max_{0\leq \ell \leq   m} \gamma\Big(  \widehat{M}_{n}^{\ell}(x)\,+\frac{\mathbf{c}'}{n} ,\, \frac{m-\ell  }{ n  }  \Big)\\   \, \leq \,& \max_{0\leq \ell \leq    m} \gamma\Big(  \widehat{M}_{n}^{\ell}(x) ,\, \frac{m-\ell  }{ n  }  \Big) \,+\,\frac{\mathbf{c}'\sqrt{b}  }{\BA \sqrt{2} }\frac{\ell}{n^2} \, .
\intertext{By our previous observation that $\widehat{M}_{n}^{\ell}(x)\leq \phi \big(x,\frac{\ell}{n}\big)$, the above is smaller than}
  \, \leq \,&  \max_{0\leq \ell \leq   m} \gamma\bigg(\phi \Big(x,\frac{\ell}{n}\Big), \frac{m-\ell}{n}  \bigg)   \,+\,\frac{\mathbf{c}'\sqrt{b}  }{\BA \sqrt{2} }\frac{m}{n^2} \,
\\  \, \leq \,& \frac{\pi}{2}-\epsilon  \,+\,\frac{\mathbf{c}'\sqrt{b}  }{\BA \sqrt{2} }\frac{1}{n}\,<\,\frac{\pi}{2}-\frac{\epsilon}{2}\, ,
\end{align*}
where we have used the identity $\gamma\big(\phi \big(x,\frac{k}{n}\big), \frac{m-k}{n}  \big)= \gamma\big(x,\frac{m}{n}\big)$ and our assumption that $ \gamma\big(x,\frac{m}{n}\big)\leq \pi/2-\epsilon $.  The last inequality holds for large enough $n$.  It follows that $m<N$ and
\begin{align*}
  M_{n}^{\ell}(x)\,-\, \widehat{M}_{n}^{\ell}(x) \, \leq \, \frac{\mathbf{c}'m}{n^2}\, =\,\mathit{O}\Big(\frac{1}{n}\Big)\, .
\end{align*}

\end{proof}

\begin{proof}[Proof of Lemma~\ref{LemTanII}] For notational convenience, we will equate $\widetilde{M}^{r}_{n}$ for $r\in \R^{+}$ with $\widetilde{M}^{\lceil r\rceil }_{n}$ in this proof.  \vspace{.2cm}

\noindent Part (\I):  From the proof of Lem~\ref{LemRT} we know that for $n\gg 1$
$$  \widetilde{M}^{n }_{n}(0) \, \approx \,  \frac{6}{b+1}\frac{n}{\log n} \, <  \,  \frac{12}{b+1}\frac{n}{\log n} \, , $$
where the inequality holds for $n$ large enough.  Consider the sequence of numbers $x_{m}^{(n)}:=\widetilde{M}^{m }_{n}\big(\widetilde{M}^{n }_{n}(0) \big) $ and $c_{m}^{(n)}:=x_{m}^{(n)}\frac{\log n}{n}$.
Notice that $c_{m+1}^{(n)} \frac{n}{\log n } = x_{m+1}^{(n)} =  \widetilde{M}_{n}\big(x_{m}^{(n)}\big) $, and so
\begin{align}\label{Dandelion}
c_{m+1}^{(n)} \frac{n}{\log n}\,  = &\, c_{m}^{(n)}\frac{n}{\log n }\Bigg(  1 \, +\, \frac{(b-1)\big(c_{m}^{(n)}\big)^2}{2\log n }\, +\, \frac{(b-1)(b-2)\big(c_{m}^{(n)}\big)^4}{6\log^2 n }\Bigg)\, +\,\frac{  b-1 }{ bn }\, \nonumber \\  < &\, c_{m}^{(n)} \frac{n}{\log n}\Bigg(  1 \, +\, \frac{(b-1)\big(c_{m}^{(n)}\big)^2}{\log n } \, +\, \frac{(b-1)(b-2)\big(c_{m}^{(n)}\big)^4}{6\log^2 n }\Bigg)
\nonumber\,\\  < &\, c_{m}^{(n)} \frac{n}{\log n}\exp \left\{\frac{(b-1)\big(c_{m}^{(n)}\big)^2}{\log n }\, +\, \frac{(b-1)(b-2)\big(c_{m}^{(n)}\big)^4}{6\log^2 n } \right\}\,.
\end{align}
For all $m\in \mathbb{N}$ such that $c_{m}^{(n)}<U$ for some constant $U$, then $\ds c_{m}^{(n)} \, \leq  \,  \frac{12}{b+1} \exp\Big\{ m\frac{(b-1) U^2 }{\log n }\,+\, \frac{(b-1)(b-2)U^4}{6\log^2  n }    \Big\}$.
Thus for $\epsilon>0$ small enough so that $U>  \frac{6}{b+1}\exp\big\{\epsilon (b-1)U^2\}$, then $x_{m}^{(n)}=\widehat{M}^{n+m }_{n}(0)  $ will be bounded by a constant multiple $\frac{n}{\log n}$ for all $m <\epsilon \log n$.

Getting a lower bound is easier since we can use the first line of~(\ref{Dandelion}) to conclude that
$$  c_{m}^{(n)} \, >\,      \bigg(  1 \, +\, \frac{(b-1)(c_{0}^{(n)})^2}{2\log n }\bigg)^m  \,  \approx  \, \exp\left\{\epsilon\frac{b-1}{2}(c_{0}^{(n)})^2  \right\}  \, ,    $$
where the the approximation is for $m=\epsilon\log n $ and $n$ large.

\vspace{.5cm}

\noindent Part (\I\I):
 Define $r_{m}\in [0,1)$ as
\begin{align}\label{RecurR}
r_{m} \, := \,     \frac{ 2  }{\pi  }   \tan^{-1} \bigg( \frac{b-1}{\pi }     \widetilde{M}_{n}^{m}(x) \bigg)    \,.
\end{align}

  The chain rule gives the first equality below:
\begin{align}
 \frac{d}{dx}\widetilde{M}_{n}^{\ell_{n}-m}(x)  \Big|_{x =\widetilde{M}_{n}^{m}(0)   } \, = &\, \prod_{j=1}^{\ell_{n}-m} \widetilde{M}_{n}'\Big( \widetilde{M}_{n}^{m+j-1}(0)\Big)\, .\nonumber
\intertext{Since $\widetilde{M}_{n}'(x)=1+\frac{b-1}{n}x+\frac{(b-1)(b-2)}{2n^2}x^2$, the above is equal to    }
 \, = & \, \prod_{j=1}^{\ell_{n}-m}  \bigg(1+ \frac{b-1}{n}\widetilde{M}_{n}^{m+j-1}(0)+\frac{(b-1)(b-2)}{2n^2}\Big(\widetilde{M}_{n}^{m+j-1}(0)\Big)^2   \bigg) \nonumber \\   \leq  &\, \exp\left\{ \frac{b-1}{n}\Big(   1+\frac{c}{\log n}  \Big)\sum_{j=1}^{\ell_{n}-m}\widetilde{M}_{n}^{m+j-1}(0)    \right\}\,, \nonumber
\intertext{where we have applied part (i) to bound $\widetilde{M}_{n}^{m+j-1}(0)$ by $\frac{cn}{\log n}$ in the quadratic term.  By definition of the values $r_{j}\in [0,1)$, we have the equality    }
   = & \, \exp \left\{ \frac{\pi}{n}\Big(   1+\frac{c}{\log n}  \Big)  \sum_{j=m}^{\ell_{n}-1}\tan\big(\frac{\pi}{2}r_{j}   \big) \right\}\,.
\intertext{The difference between $\frac{1}{n}\sum_{j=m}^{\ell_{n}-1}\tan\big(\frac{\pi}{2}r_{j}   \big)$ and  $\int_{r_{m}}^{r_{\ell_{n}}}\tan (\frac{\pi}{2} s  ) ds$ is uniformly bounded, so we have $C>0$ such that   }
  \leq & \,C\exp\left\{  \pi  \Big(   1+\frac{c}{\log n}  \Big)  \int_{r_{m}}^{r_{\ell_{n}}}\tan \big(\frac{\pi}{2} s  \big) ds \right\}\,\nonumber \\ \leq &\,C\exp\left\{   \pi  \Big(   1+\frac{c}{\log n}  \Big) \Big( \log\big( \cos\big(\frac{\pi}{2}r_{m}\big) \big) - \log\big( \cos\big(\frac{\pi}{2}r_{\ell_{n}}\big)  \big)   \Big)  \right\}\nonumber \\   \leq &\,C'\frac{n^2}{\log^2 n }\frac{1}{1+\big(\frac{\pi}{b-1} \widetilde{M}_{n}^m(0)  \big)^2 }
\, .\label{Jakdah}
\end{align}
The last inequality applies the relations
$$ \frac{1}{\cos^{2}(\frac{\pi}{2}r_{\ell_{n}})}  \, = \, 1+\tan^2\big(\frac{\pi}{2}r_{\ell_{n}}\big) \, = \,  1+\frac{\pi^2}{(b-1)^2}\big(\widetilde{M}_{n}^{\ell_{n}}(0)   \big)^2        \, = \,  \mathit{O}\Big(\frac{n^2}{\log^2 n }\Big)\, ,  $$
where the order equality is by part (\I).

\vspace{.5cm}

\noindent Part (\I\I\I): There is a $C>0$ such that for  $0<x<n$
\begin{align*}
M_{n}(x)\,-\,\widetilde{M}_{n}(x) \, \leq \, \frac{ Cx^2}{n^2} \Big(  \widetilde{M}_{n}(x)\,-\,x\Big)\,+\,\frac{C}{n^2}\,   .
\end{align*}
 By using a telescoping sum,  as in Part (\I\I\I) of Lemma~\ref{LemTan}, we can bound the difference between $M_{n}^{m}(0)$ and  $\widehat{M}_{n}^{m}(0)$  through
\begin{align}\label{ShiznikII}
  M_{n}^{m}(0)\,-\,& \widetilde{M}_{n}^{m}(0)\nonumber  \\   \leq  \, &  \sum_{k=1}^{m}\Bigg[ \frac{d}{dz}\Big[ \widetilde{M}_{n}^{m-k}(z)\Big]\Big|_{z=  M_{n}^{k}(x)  } \Bigg] \Big[ M_{n}\big(M_{n}^{k-1}(0)\big)\,-\, \widetilde{M}_{n}\big(M_{n}^{k-1}(0)\big) \Big]\,
\nonumber \\   \leq  \, &   \sum_{k=1}^{m}\Bigg[ \frac{d}{dz}\Big[ \widetilde{M}_{n}^{m-k}(z)\Big]\Big|_{z=  M_{n}^{k}(x)  }\Bigg]\Bigg[ \frac{C\big(M_{n}^{k-1}(0)\big)^2}{n^2}  \bigg(  \widetilde{M}_{n}\big(M_{n}^{k-1}(0)\big)-  M_{n}^{k-1}(0)\bigg)\,+\,\frac{C}{n^2}\Bigg]\,.
\end{align}
Let $\epsilon>0$ be small and let $\alpha_{n} \in \mathbb{N}$ be the minimum of $n+1$ and the first  $m\in \mathbb{N}$ such that
\begin{align}\label{MB}
    M_{n}^{m}(0) \, >\, \widetilde{M}_{n}^{m+\epsilon \log n }(0) \, .
\end{align}

For $k\leq m < \alpha_{n}$, part (\I\I) implies the third inequality below
\begin{align}\label{Libidu}
 \frac{d}{dz}\Big[ \widetilde{M}_{n}^{m-k}(z)\Big]\Big|_{z=  M_{n}^{k}(0)  } \, \leq \, &\frac{d}{dz}\Big[ \widetilde{M}_{n}^{m-k}(z)\Big] \Big|_{z=  \widetilde{M}_{n}^{k+\epsilon\log n }(0)  } \, \leq \, \frac{d}{dz}\Big[ \widetilde{M}_{n}^{n-k}(z)\Big] \Big|_{z=  \widetilde{M}_{n}^{k+\epsilon\log n }(0)  }  \nonumber   \\ \, \leq  &\,C'\frac{n^2}{\log^2 n }\frac{1}{1+\Big(\frac{\pi}{b-1} \widetilde{M}_{n}^{k+\epsilon\log n -1}(0)  \Big)^2 }   \, .
\end{align}
The first inequality above uses that  $M_{n}^{m}(0) \leq \widetilde{M}_{n}^{m+\epsilon \log n }(0)$   and the fact that the derivative of $\widetilde{M}_{n}$ is an increasing function,  and second inequality uses that $ \widetilde{M}_{n}(x)>x$.

Moreover, for $k<\alpha_{n}$, we have the first inequality below:
\begin{align}\label{Cornwallis}
\big(M_{n}^{k-1}(0)\big)^2\Big( \widetilde{M}_{n}\big(M_{n}^{k-1}(0)\big)\,- \, & M_{n}^{k-1}(0)\Big) \nonumber  \\  \, \leq  \, & \big(\widetilde{M}_{n}^{k+\epsilon\log n-1}(0)\big)^2\Big( \widetilde{M}_{n}^{k+\epsilon\log n }(0)\,-\, \widetilde{M}_{n}^{k+\epsilon\log n -1}(0)\Big)\,.
\end{align}

Combining~(\ref{ShiznikII}),~(\ref{Libidu}), and~(\ref{Cornwallis}) for  $m<\alpha_{n}$ we have that
\begin{align*}
  M_{n}^{m}(0)\,-\, \widetilde{M}_{n}^{m}(0)  \, \leq \, &  \frac{ CC'}{\log^2 n } \sum_{k=1}^{m} \frac{ \big(\widetilde{M}_{n}^{k+\epsilon\log n-1}(0)\big)^2\Big( \widetilde{M}_{n}^{k+\epsilon\log n }(0)\,-\, \widetilde{M}_{n}^{k+\epsilon\log n -1}(0)\Big)\,+\,c }{ 1+\big(\frac{\pi}{b-1}\widetilde{M}_{n}^{k+\epsilon\log n -1}(0) \big)^2      }     \\  \nonumber
 \, \leq \, &  \frac{ CC' }{\log^2 n } \sum_{k=1}^{m+\epsilon\log n} \bigg[ \frac{b-1}{\pi}\big(\widetilde{M}_{n}^{k }(0)- \widetilde{M}_{n}^{k-1}(0)\big)\,+\,  c \bigg]\nonumber \\
 \, = \, &  \frac{ CC'  }{\log^2 n } \bigg[\frac{b-1}{\pi}\widetilde{M}_{n}^{m+\epsilon\log n }(0)\,+\,  c\big( m+\epsilon\log n  \big) \bigg] \, \leq \,   \frac{ C'' n}{\log^2 n } \,,
\end{align*}
where the final inequality holds for some $C''>0$ by part (i) and the constraint $m<\alpha_{n}\leq n+1$.

Hence, we have shown that for $m < \alpha_{n}$
 $$  M_{n}^{m}(0)\,-\, \widetilde{M}_{n}^{m}(0) \, \leq \,   \frac{ C''' n}{\log^2 n }\, . $$
However, by part (\I), there is a $c>0$ such that for large $n$
$$   \widetilde{M}_{n}^{n+\epsilon\log n }(0)  \, >\,  c\frac{n}{\log n}\, .  $$
It follows that $\alpha_{n}=n+1$ for large enough $n$.   Thus we have shown that
 $$ \frac{\log n}{n} \Big(M_{n}^{n}(0)\,-\, \widetilde{M}_{n}^{n}(0)\Big) \, = \, \mathit{O}\Big(  \frac{ 1}{\log (n) }\Big)\, . $$

\end{proof}

\begin{proof}[Proof of Lemma~\ref{LemRT}] By the remark~(\ref{Zag}), $\widetilde{W}_{m}\big(  \kappa_{b}/n;D_{n}\big)=1+\widetilde{R}_{n}\big( \kappa_{b}/n; D_{m}\big)$ is related to the maps $\widetilde{M}_{n}:[0,\infty)\rightarrow [0,\infty)$ as follows:
$$     \textup{Var}\bigg[ \widetilde{W}_{n}\Big( \frac{ \kappa_{b}}{n};D_{n}\Big) \bigg] \, =\,\mathbb{E}\bigg[ \Big| \widetilde{R}_{n}\Big( \frac{ \kappa_{b}}{n};D_{n}\Big)\Big|^2 \bigg] \,=:\, \widetilde{\varrho}_{n}^{(n)}(\BA)\,=\,\widetilde{M}_{n}^{n}(0) \,.    $$
Define the sequence of numbers $r_{m}^{(n)}\in [0,1)$ such that
$$r_{m}^{(n)} \, := \,    \frac{2}{\pi} \tan^{-1} \bigg(\frac{b-1}{\pi}    \widetilde{M}_{n}^{m}(0) \bigg)    \,. $$
It is enough for us to prove  that
$$ 1\,-\,r_{n}^{(n)}\, = \, \frac{b+1}{3(b-1)}   \frac{\log n}{n}\, +\,\mathit{O}\Big( \frac{1}{n}   \Big)  \,  $$
since, in that case,
$$   \widetilde{M}_{n}^{n}(0)\,=\,\frac{\pi}{b-1}\tan\big( \frac{\pi}{2} r_{n}^{(n)}  \big)\,\approx \, \frac{\pi}{b-1}\frac{1}{\frac{\pi}{2}- \frac{\pi}{2} r_{n}^{(n)} } \, =\, \frac{  6   }{ b+1   }\frac{n}{\log n} \, +\,\mathit{O}\Big( \frac{n}{\log^2 n}   \Big) \,.  $$

In the analysis below, we will abuse notation by writing $r_{m}$ in place of $r_{m}^{(n)}$.
As a consequence of~(\ref{Zag}), the $r_{m}$'s obey the recursive equation
\begin{align}\label{Selge}
   r_{m+1}  \,= &\, \frac{2}{\pi}\tan^{-1}\bigg( \tan \big( \frac{\pi}{2}    r_{m}  \big)\, +\, \frac{\pi}{2n}\sec^2\big( \frac{\pi}{2}    r_{m}  \big)\,+\,\frac{\pi^2 (b-2)  }{ 6(b-1)n^2  }  \tan^3 \big( \frac{\pi}{2}    r_{m}  \big)     \bigg)\,.
\end{align}
Moreover, a second-order application of Taylor's theorem   to the function $f(x)=\tan(\frac{\pi}{2}x)$   at the point $x=r_{m}$  guarantees that there is a value $r_{m}^*$ in the interval $(r_{m},r_{m}+\frac{1}{n})$ such that
\begin{align}\label{Hormander}
& r_{m}+\frac{1}{n}\, = \,\frac{2}{\pi}\tan^{-1}\bigg( \tan \big( \frac{\pi}{2}    r_{m} \big)\, +\, \frac{\pi}{2n}\sec^2\big( \frac{\pi}{2}    r_{m}  \big)\,+\, \frac{\pi^2  }{4n^2}\tan\big( \frac{\pi}{2}    r_{m}^*  \big) \sec^2\big( \frac{\pi}{2}    r_{m}^*  \big)     \bigg)\,.
\end{align}
Define $\Delta_{m}\equiv\Delta_{m}^{(n)}$ as
\begin{align*}
\Delta_{m}\,:=\,  &  \frac{\pi^2  }{4n^2}\tan\big( \frac{\pi}{2}    r_{m}^*  \big) \sec^2\big( \frac{\pi}{2}    r_{m}^*  \big)    \,-\, \frac{\pi^2 (b-2)  }{ 6(b-1)n^2  }  \tan^3 \big( \frac{\pi}{2}    r_{m}  \big) \,.
\intertext{The above can be written as follows:     }
\, = \,  &  \frac{2  }{\pi n^2}\frac{ b+1   }{3(b-1)   }  \frac{1}{(1-r_{m})^3} \,+\,\mathit{O}\Big( \frac{1}{n^3(1 -r_{m})^4  }  \Big)\,.
\end{align*}

By  a second-order application of Taylor's theorem to the function $g(x)=\frac{2}{\pi} \tan^{-1} ( x )   $ in~(\ref{Hormander}) at the point $x=\tan(\frac{\pi}{2}r_{m+1})$, there is an
 $r_{m}^{**}\in (r_{m+1}, r_{m}+\frac{1}{n})$   such that
\begin{align}\label{Horbit}
r_{m}+\frac{1}{n}\, = \,&r_{m+1}\, +\,\frac{2}{\pi} \Delta_{m} \frac{1     }{1+\tan^2\big( \frac{\pi}{2}r_{m+1}\big)    } \,  -\,\frac{2}{\pi}\Delta_{m}^2   \frac{   \tan\big( \frac{\pi}{2}r_{m}^{**}\big)    }{\big(1+\tan^2\big( \frac{\pi}{2}r_{m}^{**}\big) \big)^2   }\, .
\end{align}

To bound the difference between $1$ and $r_{n}$, we can apply~(\ref{Horbit}) as follows:
\begin{align*}
1\,-\,r_{n}\, & = \,\sum_{m=0}^{n-1}\Big(r_{m}+\frac{1}{n}-r_{m+1}   \Big)\,\\ & =\,\frac{2}{\pi }\sum_{m=0}^{n-1}\Delta_{m}  \cos^2\big( \frac{\pi}{2}r_{m+1}\big)  \, - \,\frac{2 }{\pi}\sum_{m=0}^{n-1}\Delta_{m}^2 \sin\big( \frac{\pi}{2}r_{m}^{**}\big)   \cos^3\big( \frac{\pi}{2}r_{m}^{**}\big)
\\ & =\,\frac{1}{n^2}\frac{ b+1   }{ 3(b-1)    }\sum_{m=0}^{n-1}\frac{1}{1-r_{m}} \, + \,\mathit{O}\bigg(\frac{1}{n^4}\sum_{m=0}^{n-1}\frac{1}{(1-r_{m})^3}  \bigg)\, .
\intertext{The first term above is a Riemann sum and the second term is smaller by a factor of $\frac{1}{n}$.    }
 & =\,\frac{1 }{n}\frac{ b+1   }{ 3(b-1)    }\bigg(1+\mathit{O}\Big(\frac{n^{-1}}{1-r_{n}}\Big)\bigg)     \int_{0}^{r_{n}}  \frac{1   }{1-x }dx \,
+\, \mathit{O}\bigg(\frac{1}{n^3}\int_{0}^{r_{n}}\frac{1}{(1-x)^3}dx  \bigg)   \\ &  =\, \frac{1}{n}\frac{ b+1   }{ 3(b-1)    }\bigg(1+\mathit{O}\Big(\frac{n^{-1}}{1-r_{n}}\Big)\bigg)\ln\Big(\frac{1}{1-r_{n}}\Big) \,
\end{align*}
The above implies that $1-r_{n} = \frac{ b+1   }{ 3(b-1)    }\frac{\log n}{n}+\mathit{O}\big( \frac{1}{n}  \big)   $, which completes the proof.

\end{proof}

\end{document}